\documentclass{agtart_a}
\pdfoutput=1

\usepackage{pinlabel}

\title{A family of pseudo-Anosov braids with small dilatation}

\author{Eriko Hironaka}
\givenname{Eriko}
\surname{Hironaka}
\address{Department of Mathematics\\
Florida State University\\\newline
Tallahassee FL 32306-4510\\
USA}
\email{hironaka@math.fsu.edu}
\urladdr{}

\author{Eiko Kin}
\givenname{Eiko}
\surname{Kin}
\address{Department of Mathematical and Computing Sciences\\
Tokyo Institute of Technology\\\newline
2-12-1-W8-45 Oh-okayama\\
Meguro-ku\\
Tokyo 152-8552\\
Japan}
\email{kin@is.titech.ac.jp}
\urladdr{}

\volumenumber{6}
\issuenumber{}
\publicationyear{2006}
\papernumber{27}
\startpage{699}
\endpage{738}

\doi{}
\MR{}
\Zbl{}


\received{23 July 2005}
\revised{13 April 2006}
\accepted{26 April 2006}
\published{12 June 2006}
\publishedonline{12 June 2006}
\proposed{}
\seconded{}
\corresponding{}
\editor{}
\version{}

\arxivreference{}

\let\xysavmatrix\xymatrix
\def\xymatrix{\disablesubscriptcorrection\xysavmatrix}
\AtBeginDocument{\let\bar\wbar\let\tilde\wtilde\let\hat\what}

\makeatletter
\def\cnewtheorem#1[#2]#3{\newtheorem{#1}{#3}[section]
\expandafter\let\csname c@#1\endcsname\c@thm}

\newtheorem{thm}{Theorem}[section]
\cnewtheorem{lem}[thm]{Lemma}
\cnewtheorem{cor}[thm]{Corollary}
\cnewtheorem{prop}[thm]{Proposition}
\theoremstyle{remark}
\cnewtheorem{fact}[thm]{Fact}
\cnewtheorem{claim}[thm]{Claim} 
\cnewtheorem{rem}[thm]{Remark}
\cnewtheorem{exm}[thm]{Example}
\cnewtheorem{prob}[thm]{Problem}
\cnewtheorem{ques}[thm]{Question}

\makeatother  

\def \HH {\mathrm{H}}

\def \sT {\mathcal T}
\def\sE{{\mathcal E}}
\def \sF {{\mathcal F}}
\def \sS{\mathcal S}

\def\sK{\mathcal K}
\def\sB{{\mathcal B}}
\def\sM{{\mathcal M}}
\def\sQ{{\mathcal Q}}

\def\sR{{\mathcal R}}
\def\DE{{\mathcal{E}_{\mathrm{dir}}}}

\def\g{\mathfrak g}
\def\f{\mathfrak f}
\def\h{\mathfrak h}
\def\r{\mathfrak r}

\begin{document}

\begin{asciiabstract}
This paper describes a family of pseudo-Anosov braids with small
dilatation.  The smallest dilatations occurring for braids with 3, 4
and 5 strands appear in this family.  A pseudo-Anosov braid with
2g+1 strands determines a hyperelliptic mapping class with the same
dilatation on a genus-g surface.  Penner showed that logarithms of
least dilatations of pseudo-Anosov maps on a genus-g surface grow
asymptotically with the genus like 1/g, and gave explicit examples
of mapping classes with dilatations bounded above by log 11/g.
Bauer later improved this bound to log 6/g.  The braids in this paper
give rise to mapping classes with dilatations bounded above by
log(2+sqrt(3))/g.  They show that least dilatations for hyperelliptic
mapping classes have the same asymptotic behavior as for general mapping
classes on genus-g surfaces.
\end{asciiabstract}

\begin{htmlabstract}
This paper describes a family of pseudo-Anosov braids with small
dilatation.  The smallest dilatations occurring for braids with 3,4 and
5 strands appear in this family.  A pseudo-Anosov braid with 2<i>g</i>+1
strands determines a hyperelliptic mapping class with the same dilatation
on a genus&ndash;<i>g</i> surface.  Penner showed that logarithms
of least dilatations of pseudo-Anosov maps on a genus&ndash;<i>g</i>
surface grow asymptotically with the genus like 1/<i>g</i>, and gave
explicit examples of mapping classes with dilatations bounded above by
log 11/<i>g</i>.  Bauer later improved this bound to log 6/<i>g</i>.
The braids in this paper give rise to mapping classes with dilatations
bounded above by log(2+&radic;3)/<i>g</i>.  They show that least
dilatations for hyperelliptic mapping classes have the same asymptotic
behavior as for general mapping classes on genus&ndash;<i>g</i> surfaces.
\end{htmlabstract}

\begin{abstract}
This paper describes a family of pseudo-Anosov braids with small
dilatation.  The smallest dilatations occurring for braids with 3,4
and 5 strands appear in this family.  A pseudo-Anosov braid with
$2g+1$ strands determines a hyperelliptic mapping class with the same
dilatation on a genus--$g$ surface.  Penner showed that logarithms of
least dilatations of pseudo-Anosov maps on a genus--$g$ surface grow
asymptotically with the genus like $1/g$, and gave explicit examples
of mapping classes with dilatations bounded above by $\log 11/g$.
Bauer later improved this bound to $\log 6/g$.  The braids in this paper
give rise to mapping classes with dilatations bounded above by $\log(2 +
\sqrt{3})/g$.  They show that least dilatations for hyperelliptic mapping
classes have the same asymptotic behavior as for general mapping classes
on genus--$g$ surfaces.
\end{abstract}

\maketitle

\section{Introduction}

In this paper, we study a family of generalizations of these examples to arbitrary numbers of strands.   
Let $\sB(D,s)$ denote the braid group on $D$ with $s$ strands, 
where $D$ denotes a $2$--dimensional closed disk. 
First consider the braids $\beta_{m,n}$ in $\sB(D,m+n+1)$ given by 
$$
\beta_{m,n} = \sigma_1\ldots\sigma_m\sigma_{m+1}^{-1}\ldots \sigma_{m+n}^{-1}.
$$
Matsuoka's example \cite{Ma86} appears as $\beta_{1,1}$, and Ko, Los and
Song's example \cite{SoKoLos02} as $\beta_{2,1}$.   
For any $m,n \ge 1$, $\beta_{m,n}$ is pseudo-Anosov  (\fullref{Bmn-classification-thm}). 
The dilatations of $\beta_{m,m}$ coincide with those found by Brinkmann 
\cite{Brinkmann04} (see also \fullref{fiberedQ-section}),  who also shows that the dilatations 
arising in this family can be made arbitrarily close to $1$. 

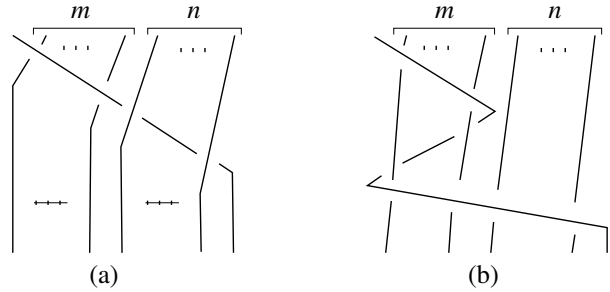
\begin{figure}[htbp]
\labellist\small
\pinlabel {$m$} [b] at 45 173
\pinlabel {$n$} [b] at 118 173
\pinlabel {$m$} [b] at 279 173
\pinlabel {$n$} [b] at 350 173
\pinlabel {(a)} [b] at 60 0
\pinlabel {(b)} [b] at 304 0
\endlabellist
\begin{center}
\includegraphics[height=1.5in]{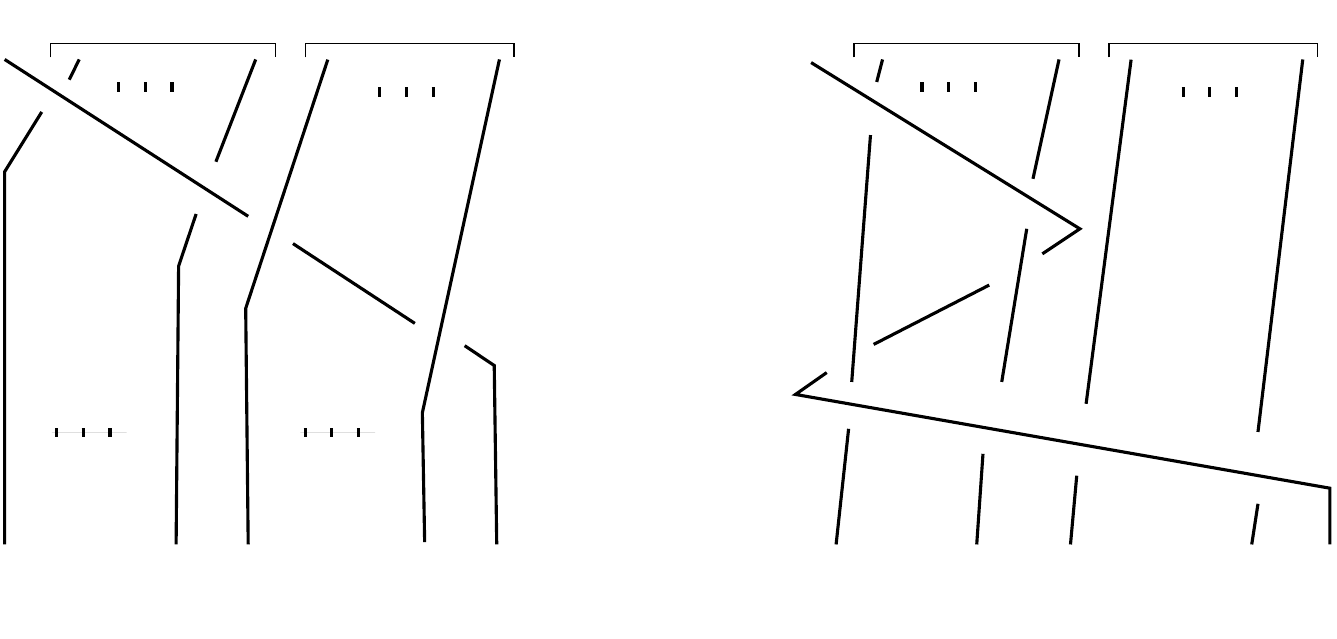}
\caption{Braids (a) $\beta_{m,n}$ and (b) $\sigma_{m,n}$}
\label{SmnBmn-fig}
\end{center}
\end{figure}

It turns out that one may find smaller dilatations 
by passing a strand of $\beta_{m,n}$ once around the remaining strands. 
As a particular example, we consider the braids $\sigma_{m,n}$ defined by taking the rightmost-strand 
of $\beta_{m,n}$ and passing it counter-clockwise once around the remaining strands.  
\fullref{SmnBmn-fig} gives an illustration of $\beta_{m,n}$ and $\sigma_{m,n}$. 
The braid $\sigma_{1,3}$ is conjugate to Ham and Song's braid  $\sigma_1\sigma_2\sigma_3\sigma_4\sigma_1\sigma_2$. 
For $|m-n| \leq 1$, we show that $\sigma_{m,n}$ is periodic or reducible. 
Otherwise $\sigma_{m,n}$ is pseudo-Anosov 
with dilatation strictly less than the dilatation of $\beta_{m,n}$ 
(\fullref{Smn-classification-thm}, \fullref{inequality-cor}). 
The dilatations of $\sigma_{g-1,g+1}$ $(g \ge 2)$ satisfy the inequality 
\begin{eqnarray}
\label{inequality-eqn}
\lambda(\sigma_{g-1,g+1})^g  <  2 + \sqrt{3}
\end{eqnarray}
(\fullref{min-Smn-prop}). 

Let $\sM_g^s$ denote the set of {\it mapping classes} (or isotopy classes) 
of homeomorphisms on the closed orientable genus--$g$ surface $F_g$ 
set-wise preserving $s$ points.  
We denote $\mathcal{M}_g^0$ by $\mathcal{M}_g$. 
For any subset $\Gamma \subset \sM_g^s$, define $\lambda(\Gamma)$ to be the least dilatation 
among pseudo-Anosov elements of $\Gamma$, 
and let $\delta(\Gamma)$ be the logarithm of $\lambda(\Gamma)$.  
For the braid group $\sB(D,s)$, and any subset $\Gamma \subset \sB(D,s)$, 
define $\lambda(\Gamma)$ and $\delta(\Gamma)$ in a similar way. 
By a result of Penner \cite{Penner91} (see also McMullen \cite{McMullen:Poly}), 
$\delta(\sM_g) \asymp \frac{1}{g}$. 

An element of $\sM_g$ is called {\it hyperelliptic} if it commutes with an involution $\iota$ on $F_g$ 
such that the quotient of $F_g$ by $\iota$ is $S^2$. 
Let $\sM_{g,\mathrm{hyp}} \subset \sM_g$ denote the subset of hyperelliptic elements of $\sM_g$.    
Any pseudo-Anosov braid on $2g+1$ strands determines 
a hyperelliptic element of $\sM_g$ with the same dilatation (\fullref{spectrum-prop}).  
Thus, \eqref{inequality-eqn} implies: 

\begin{thm}
\label{inequalities-thm} 
For $g \ge 2$, 
$$\delta(\sM_g)  \leq \delta(\sM_{g,\mathrm{hyp}}) \leq
\delta(\sB(D,2g+1)) < \frac{\log(2 + \sqrt{3})}{g}.$$
\end{thm}
This improves the upper bounds on $\delta(\sM_g)$ found by Penner
$\bigl(\frac{\log 11}{g}\bigr)$ \cite{Penner91} and Bauer
$\bigl(\frac{\log 6}{g}\bigr)$ \cite{Bauer92}. 
\fullref{inequalities-thm} shows the following. 

\begin{thm}
\label{asymp-thm}  
For $g \ge 2$,
$$\delta(\sB(D,2g+1)) \asymp \tfrac{1}{g}\quad \text{and}\quad
\delta(\sM_{g,\mathrm{hyp}})  \asymp \tfrac{1}{g}.$$
\end{thm} 

This paper is organized as follows.  
\fullref{prelim-section} reviews basic terminology and results on mapping class groups. 
In \fullref{main-section}, we determine the Thurston--Nielsen types of
$\beta_{m,n}$ and $\sigma_{m,n}$ by finding efficient graph maps for their
monodromy actions following Bestvina and Handel \cite{BH94}.  
We observe that the associated train tracks have ``star-like" components, 
and their essential forms don't depend on $m$ and $n$ 
(Figures~\ref{Bmn-TT} and \ref{Smn-TT}). 
To find bounds and inequalities among the dilatations, we apply the notion of 
Salem--Boyd sequences \cite{Boyd77,Salem44}, and relate the similar forms 
of the efficient graph maps for $\beta_{m,n}$ and $\sigma_{m,n}$ to similar forms for 
characteristic polynomials of the dilatations. 
In particular, we show that the least dilatation that occurs among $\beta_{m,n}$ and $\sigma_{m,n}$ 
for $m+n = 2g$ $(g \ge 2)$ is realized by $\sigma_{g-1,g+1}$, 
and find bounds for $\lambda(\sigma_{g-1,g+1})$ yielding the inequality  \eqref{inequality-eqn}. 
\fullref{discussion-section} discusses the problem of determining the least dilatations of 
special subclasses of  pseudo-Anosov maps. 
In \fullref{forcing-section}, we briefly describe the relation between the forcing relation 
on  braid types and dilatations, and show how $\sigma_{m,n}$ arise as the braid types of 
periodic orbits of the Smale--horseshoe map. 
In \fullref{fiberedQ-section}, we consider  pseudo-Anosov maps arising as the monodromy of fibered links, and relate our examples to those of Brinkmann. 

\subsection*{Acknowledgements}
The authors thank Hiroyuki Minakawa for valuable discussions, and an 
algebraic trick that improved our original upper bound for $\lambda(\sigma_{g-1,g+1})$. 
The first author thanks the J\,S\,P\,S, Osaka University and host Makoto Sakuma 
for their hospitality and support during the writing of this paper. 
The second author is grateful for  the financial support provided by the research fellowship of 
the 21st century COE program in Kyoto University. 

\section{Preliminaries}
\label{prelim-section} 
In this section, we review basic definitions and properties of braids 
(\fullref{braids-section}), mapping class groups (\fullref{mappings-section}), 
spectra (\fullref{spectrum-section}), and a criterion of the pseudo-Anosov property 
(\fullref{Bestvina-Handel-section}).  
Some results are well-known, and more complete expositions can be found in
the articles by Bestvina--Handel \cite{BH94} and
Fathi--Laudenbach--Poenaru \cite{FLP}, and the books by Birman
\cite{Birman74}, and Casson--Bleiler \cite{CB88}. 
We include them here for the convenience of the reader. 

\subsection{Braids}
\label{braids-section}

Let $F$ be a compact orientable surface with $s$ marked points 
${\sS}=\{p_1,\dots,p_s\} \subset \mathrm{int}(F)$, the interior of $F$. 
A {\it braid representative} $\beta$ on $F$ is the images 
of continuous maps 
$$f_{p_1},\dots,f_{p_s} \colon\thinspace I= [0,1] \rightarrow F \times I,$$
satisfying for $i=1,\dots,s$, 
\begin{enumerate}
\item[(B1)] $f_{p_i} (0) = p_i \times 0$, 
\item[(B2)] $f_{p_i}(1) \in \sS \times 1$, 
\item[(B3)] $f_{p_i}(t) \in F \times t $ for $t \in I$,  and
\item[(B4)] $f_{p_i}(t) \neq f_{p_j}(t)$ for $t$ and $i \neq j$.
\end{enumerate}
Define the product of two braid representatives to be their concatenation. 
Let $\sB(F;\sS)$ be the set of braid representatives up to ambient isotopy fixing the 
boundary of $F$ point-wise. 
The above definition of product determines a well-defined group structure on $\sB(F;\sS)$, 
and the group is called the {\it braid group} on $F$. 

For any partition $\sS = \sS_1 \cup \cdots \cup \sS_r$, let $\sB(F;\sS_1,\dots,\sS_r)$ be 
the subgroup of $\sB(F;\sS)$ consisting of braids $(f_{p_1},\dots,f_{p_s})$ satisfying 
for all $p \in \sS_j$ ($j \in \{1, \ldots, r\}$),   $f_p(1) \in \sS_j$. 

In the rest of this section, we assume that 
$F$ is either a disk $D$ or a sphere $S^2$. 
Then the braid group $\sB(F;\sS)$ has generators  $\sigma_1,\dots,\sigma_{s-1}$, 
where $\sigma_i$ is the braid shown in \fullref{braidgen-fig}.  
When $F = D$, $\sB(D;\sS)$ is called the {\it Artin braid group} and has finite presentation 
$$
\langle \sigma_1,\dots,\sigma_{s-1} \thinspace: \thinspace 
\sigma_i\sigma_{i+1}\sigma_i = \sigma_{i+1}\sigma_i\sigma_{i+1},\ 
\sigma_i\sigma_j = \sigma_j\sigma_i \ \mbox{if $|i-j| \ge 2$}\rangle.
$$
\vspace{-12pt}
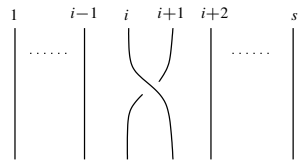
\begin{figure}[htbp]
\labellist\tiny
\pinlabel {$1$} [b] at 2 95
\pinlabel {$i{-}1$} [b] at 52 95
\pinlabel {$i$} [b] at 82 95
\pinlabel {$i{+}1$} [b] at 114 95
\pinlabel {$i{+}2$} [b] at 145 95
\pinlabel {$s$} [b] at 202 95
\endlabellist
\begin{center}
\includegraphics[width=1.5in]{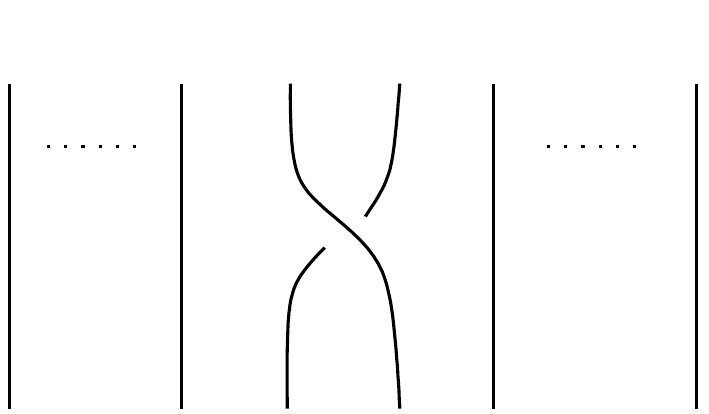}
\caption{Braid generator $\sigma_i$}
\label{braidgen-fig}
\end{center}
\end{figure}
Consider the natural map $c \colon \thinspace D \rightarrow S^2$ given by identifying $\partial D$, 
the boundary of $D$ to a point $p_\infty$ on $S^2$.  
 By abuse of notation, we will write $\sS$ for $c(\sS)$. 
Then there is an induced map: 
\begin{eqnarray}
\label{widehat-eqn}
\sB(D;{\sS}) &\rightarrow&\sB(S^2;\sS,\{p_\infty\})
\\
\beta &\mapsto& \what{\beta}.\nonumber
\end{eqnarray}
For example, $\what{\beta}_{m,n}$ and $\what{\sigma}_{m,n}$ are shown in 
\fullref{spherical-fig} with the strand associated to $p_\infty$ drawn on the right. 
\begin{figure}[htbp]
\labellist\small
\pinlabel {$m$} [b] at 45 173
\pinlabel {$n$} [b] at 118 173
\pinlabel {$m$} [b] at 279 173
\pinlabel {$n$} [b] at 350 173
\pinlabel {(a)} [b] at 60 0
\pinlabel {(b)} [b] at 304 0
\pinlabel {$p_{\infty}$} [b] at 165 170
\pinlabel {$p_{\infty}$} [b] at 400 170
\endlabellist
\begin{center}
\includegraphics[height=1.5in]{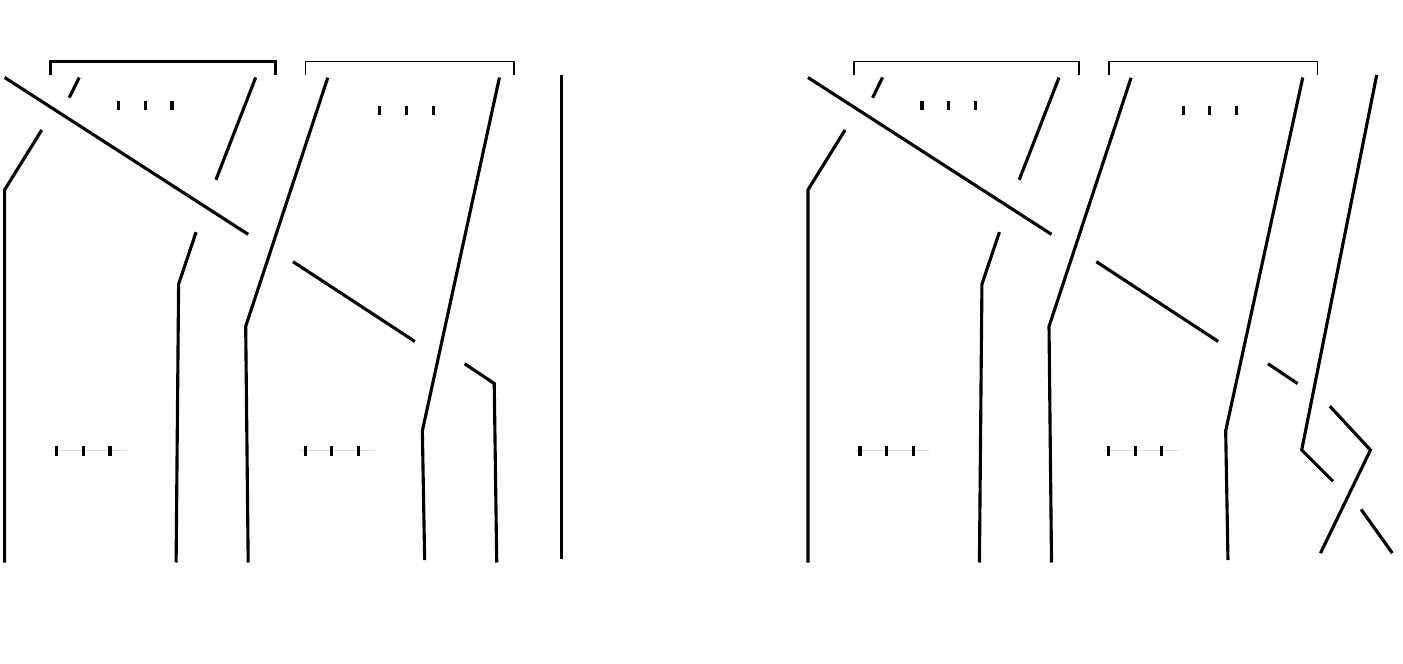}
\caption{Images of (a) $\beta_{m,n}$ and (b) $\sigma_{m,n}$ in $\sB(S^2;\sS,\{p_\infty\})$}
\label{spherical-fig}
\end{center}
\end{figure}
For $\beta \in \sB(D;{\sS})$, let $\wwbar{\beta} \in \sB(S^2;{\sS})$ be the image of 
$\what{\beta}$ under the forgetful map:
\begin{eqnarray}
\label{overline-eqn}
 \sB(S^2;\sS,\{p_\infty\}) &\rightarrow&  \sB(S^2;{\sS})\\
\what{\beta} &\mapsto& \wwbar{\beta}\nonumber
 \end{eqnarray} 
The following lemma can be found in the book by Birman \cite{Birman74}.

\begin{lem} 
The map $\sB(D;\sS) \rightarrow \sB(S^2;\sS)$ 
given by composing the maps in \eqref{widehat-eqn} and \eqref{overline-eqn} has kernel 
normally generated by 
$\xi= \sigma_1\sigma_2 \ldots \sigma_{s-1}^2\sigma_{s-2} \ldots \sigma_1$. 
\end{lem}
For example,  $\beta_{m,n}$ and $\sigma_{m,n}$ shown in \fullref{SmnBmn-fig} 
differ by a conjugate of $\xi$, and hence  we have the following. 

\begin{prop}  
The braids $\beta_{m,n}$ and $\sigma_{m,n}$ satisfy 
$\wbar{\beta}_{m,n}= \wbar{\sigma}_{m,n}$. 
\end{prop}

The final lemma of this section deals with notation.  

\begin{lem}
\label{invariance-lem} 
Let $\sS_1$ and $\sS_2$ be finite subsets of $\mbox{int}(F)$ with the same cardinality, 
and $h \colon\thinspace F \rightarrow F$ any homeomorphism taking $\sS_1$ to $\sS_2$. 
Then conjugation by $h$ defines an isomorphism 
$\sB(F;\sS_1) \rightarrow \sB(F;\sS_2)$. 
\end{lem}
In light of \fullref{invariance-lem} 
if $s$ is the cardinality of $\sS$, we will write $\sB(F,s)$ for $\sB(F;\sS)$. 

\subsection{Mapping class groups}
\label{mappings-section}

For any closed orientable surface $F$ and a finite subset $\sS \subset F$ of marked points, 
let $\sM(F;\sS)$ be the group of isotopy classes of 
orientation preserving homeomorphisms of $F$ set-wise preserving $\sS$.   
The Thurston--Nielsen classification states that any homeomorphism of a surface 
is isotopic to one of three types, which we describe below. 

A map $\Phi \colon\thinspace F \rightarrow F$ set-wise preserving $\sS$ is defined to be {\it periodic} 
if some power of $\Phi$ equals the identity map; 
and {\it reducible} if there is a $\Phi$--invariant closed $1$--submanifold whose complementary components in $F \setminus \sS$ have negative Euler characteristic.   
A mapping class $\phi \in \sM(F;\sS)$ is {\it periodic} (respectively,  {\it reducible}) 
if it contains a representative that is periodic (respectively, reducible).  

Before defining the  third type of mapping class, we will make some preliminary definitions. 
A {\it singular foliation} $\sF$  on $F$ with respect to $\sS$ 
is a partition of $F$ into a union of real intervals $(-\infty,\infty)$ and $[0,\infty)$ 
called {\it leaves} such that for each point $x \in F$,  the foliation $\sF$ near $x$ has 
one of the following types in a local chart around $x$: 
\begin{enumerate}
\item[(F1)] 
$x \in F$ is a {\it regular point} (we will also say a {\it $2$--pronged point}) of 
$\sF$ (\fullref{prongs-fig1}(a)). 
\item[(F2)] 
$x \in F$ is an {\it $n$--pronged singularity of $\sF$} (\fullref{prongs-fig1}(b),(c)), 
where $n \ge 1$ if $x \in  \sS$,  and $n \ge 3$ if $x \in F \setminus \sS$. 
\end{enumerate}
Two singular foliations $\sF^{+}$ and $\sF^{-}$ with respect to $\sS$ are {\it transverse} if 
they have the same set of singularities $\mathcal{S}'$ 
and if  the leaves of $\sF^{+}$ and $\sF^{-}$ intersect transversally on $F \setminus \mathcal{S}'$. 

\begin{figure}[htbp]
\begin{center}
\labellist\small
\pinlabel {$x$} [l] <1pt,1.5pt> at 146 679
\pinlabel {$x$} [l] <.5pt,.5pt> at 251 685
{\hair 1.5pt
\pinlabel {$x$} [tl] at 449 686 }
\pinlabel {(a)\qua $n=2$} [b] at 155 590
\pinlabel {(b)\qua $n=1$} [b] at 300 590
\pinlabel {(c)\qua $n=3$} [b] at 450 590
\endlabellist
\includegraphics[width=3in]{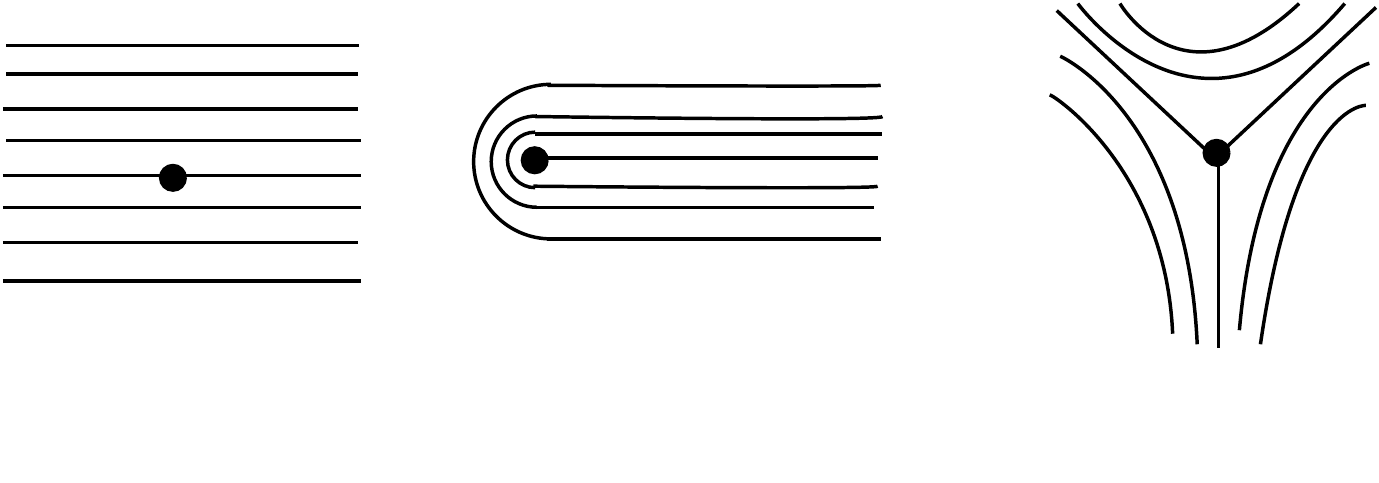}
\caption{Local picture of a singular foliation}
\label{prongs-fig1}
\end{center}
\end{figure}
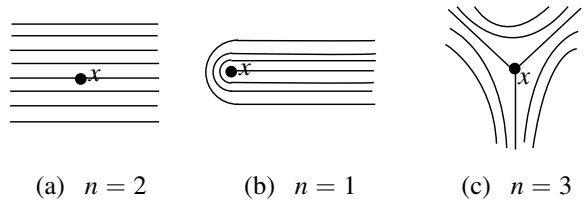

A path $\alpha$ on $F$ is a {\it transverse arc} relative to a singular foliation $\sF$ with respect to $\sS$ 
if $\alpha$ intersects the leaves of $\sF$ transversely. 
 Two transverse arcs $\alpha_0$ and $\alpha_1$ relative to $\sF$ are {\it homotopic} if 
there is a homotopy $\alpha \colon\thinspace I \times I \rightarrow F$ 
such that $\alpha (I \times  0) = \alpha_0$, $\alpha (I \times 1) = \alpha_1$, and 
for all $t \in I$,  $\alpha(t \times I)$ is contained in a leaf of ${\mathcal F}$. 
We say that $\mu$ is a {\it transverse measure} 
on a singular foliation $\sF$ with respect to $\sS$  
if $\mu$ defines a non-negative Borel measure $\mu(\alpha)$ 
on each transverse arc $\alpha$  with the following two properties: 
\begin{enumerate}
\item[(M1)] 
If $\alpha'$ is a subarc of $\alpha$, then $\mu(\alpha')=\mu(\alpha)|_{\alpha'}$. 
\item[(M2)] 
If transverse arcs $\alpha_0$ and $\alpha_1$ relative to $\sF$ are homotopic, 
then $\mu(\alpha_0) = \mu(\alpha_1)$. 
\end{enumerate}
A pair $({\mathcal F}, \mu)$ satisfying (M1) and (M2) is called a {\it measured foliation}. 
Given a measured foliation $(\sF,\mu)$ and a number $\lambda>0$, 
$(\sF, \lambda \mu)$ denotes the measured foliation whose 
leaves are the same as those of $\sF$ such that 
the measure of each transverse arc $\alpha$  relative to ${\mathcal F}$ 
is given by $\lambda \mu(\alpha)$. 
For a homeomorphism $f \colon\thinspace F \rightarrow F$ set-wise preserving $\sS$, 
 $(\sF', \mu')= f({\mathcal F}, \mu)$ is the measured foliation whose leaves  are the images 
of leaves of ${\mathcal F}$ under $f$, 
and the measure $\mu'$ on each arc $\alpha$ transverse to $\sF'$ 
is given by $\mu(f^{-1}(\alpha))$. 

A map $\Phi \colon\thinspace F \rightarrow F$ set-wise preserving $\sS$ is  {\it pseudo-Anosov} if 
there is a number $\lambda >1$ and a pair of transverse measured foliations 
$(\mathcal{F}^\pm, \mu_\pm)$ such that 
$\Phi(\mathcal{F}^\pm, \mu_\pm) = (\mathcal{F}^\pm, \lambda^{\pm 1} \mu_\pm)$. 
The number $\lambda = \lambda(\Phi)$ is called the {\it dilatation} of $\Phi$, and 
$\mathcal{F}^-$ and $\sF^{+1}$ are called the {\it stable} and {\it unstable foliations}  
or the {\it invariant foliations} associated to $\Phi$. 
A mapping class $\phi \in \sM(F;\sS)$ is {\it pseudo-Anosov} if $\phi$ is 
the isotopy class of a pseudo-Anosov map $\Phi$.  
In this case, the dilatation of $\phi$ is  defined to be $\lambda(\phi) = \lambda(\Phi)$. 

\begin{thm} [Thurston--Nielsen Classification Theorem]  
Any element $\phi \in \sM(F;\sS)$ 
is either periodic, reducible or pseudo-Anosov.  
Furthermore, if $\phi$ is  pseudo-Anosov, then the pseudo-Anosov representative of 
$\phi$ is unique up to conjugacy.
\end{thm}

As with braids, for any partition $\sS = \sS_1 \cup  \cdots \cup \sS_r$, there is a subgroup 
$$
\sM(F;\sS_1,\dots,\sS_r) \subset \sM(F;\sS)
$$
that preserves each $\sS_i$ set-wise.  
There is a natural map 
$$
\sM(F;\sS_1,\dots,\sS_r) \rightarrow \sM(F;\sS_1,\dots,\sS_{r-1})
$$
called the {\it forgetful map}.  
For pseudo-Anosov mapping classes $\phi$, 
$\log(\lambda(\phi))$  can be interpreted as the minimal topological 
entropy among all representatives of $\phi$ (see
Fathi--Laudenbach--Poenaru \cite{FLP}).  
We thus have the following inequality on dilatations. 

 \begin{lem}
 \label{closure1-lem}  
 Let $\phi \in \sM(F;\sS_1,\dots,\sS_r)$, and $\psi \in \sM(F;\sS_1,\dots,\sS_{r-1})$ 
 the image of $\phi$  under the forgetful map.  
 If $\phi$ and $\psi$ are both pseudo-Anosov, then $\lambda(\phi) \ge \lambda(\psi)$. 
 \end{lem}

\begin{lem}
\label{closure2-lem} 
Let $\phi \in \sM(F;\sS_1,\dots,\sS_r)$ be pseudo-Anosov. 
Suppose that the pseudo-Anosov representative $\Phi$ of $\phi$ does not have a $1$--pronged
singularity at any point of $S_r$.  
Let $\psi \in \sM(F;\sS_1,\dots,\sS_{r-1})$ be the image of $\phi$ under the forgetful map. 
Then $\psi$ is pseudo-Anosov and $\lambda(\psi)= \lambda(\phi)= \lambda(\Phi)$. 
\end{lem}

\begin{proof} 
Let $\sF^{\pm}$ be singular foliations with respect to $\sS_1\cup\dots\cup\sS_r$, and 
$(\sF^\pm,\mu_\pm)$ a pair of transverse measured foliations associated to $\Phi$. 
Since $\sF^{\pm}$ does not have $1$--pronged singularities at points of $S_r$,  
$\sF^{\pm}$ give well-defined singular foliations with respect to $\sS_1\cup\dots\cup\sS_{r-1}$. 
Thus, $\Phi$ is a pseudo-Anosov representative of $\psi$, and 
hence $\lambda(\psi)= \lambda(\phi)= \lambda(\Phi)$. 
\end{proof}

As in the case of braids, changing the location 
of the points in $\sS$ by a homeomorphism does not change the group $\sM(F;\sS)$. 

\begin{lem}
\label{invarianceM-lem}
Let $\sS_1$ and $\sS_2$ be two finite subsets of $F$ with the same cardinality, 
and $h \colon\thinspace F \rightarrow F$ any homeomorphism taking $\sS_1$ to $\sS_2$. 
Then conjugation by $h$ defines an isomorphism 
$\sM(F;\sS_1) \rightarrow \sM(F;\sS_2)$. 
\end{lem}
If $F$ has genus--$g$, and $\sS$ has cardinality $s$, we will also write $\sM_g^s = \sM(F;\sS)$. 

The theory of mapping class groups on closed surfaces extends to mapping class groups on 
surfaces with boundary.  Let $F^b$ be a compact orientable surface with $b$ boundary 
components, and $\sS \subset \mathrm{int}(F^b)$ a finite set.  
Define $\sM(F^b;\sS)$ to  be the group of isotopy classes 
of orientation preserving homeomorphisms  of $F^b$ set-wise preserving $\sS$ and the boundary components. 
A {\it singular foliation} $\sF$ on $F^b$ with respect to the set of marked points $\sS$ 
is a partition of $F$ into a union of leaves such that each point $x \in \mbox{int}(F)$ has a local chart satisfying one of the conditions (F1), (F2), and 
each boundary component has $n$--prongs for some $n \ge 1$.  
\fullref{prongs-fig2} illustrates representative leaves of a singular foliation with a $1$--pronged (\fullref{prongs-fig2}(a)) and $3$--pronged (\fullref{prongs-fig2}(b)) singularity. 
{\it Periodic}, {\it reducible} and {\it pseudo-Anosov} mapping  classes 
are defined as for the case of closed surfaces using this definition of singular foliations. 

\begin{figure}[htbp]
\begin{center}
\labellist\small
\pinlabel {(a)} [b] at 135 575
\pinlabel {(b)} [b] at 366 575
\endlabellist
\includegraphics[width=2in]{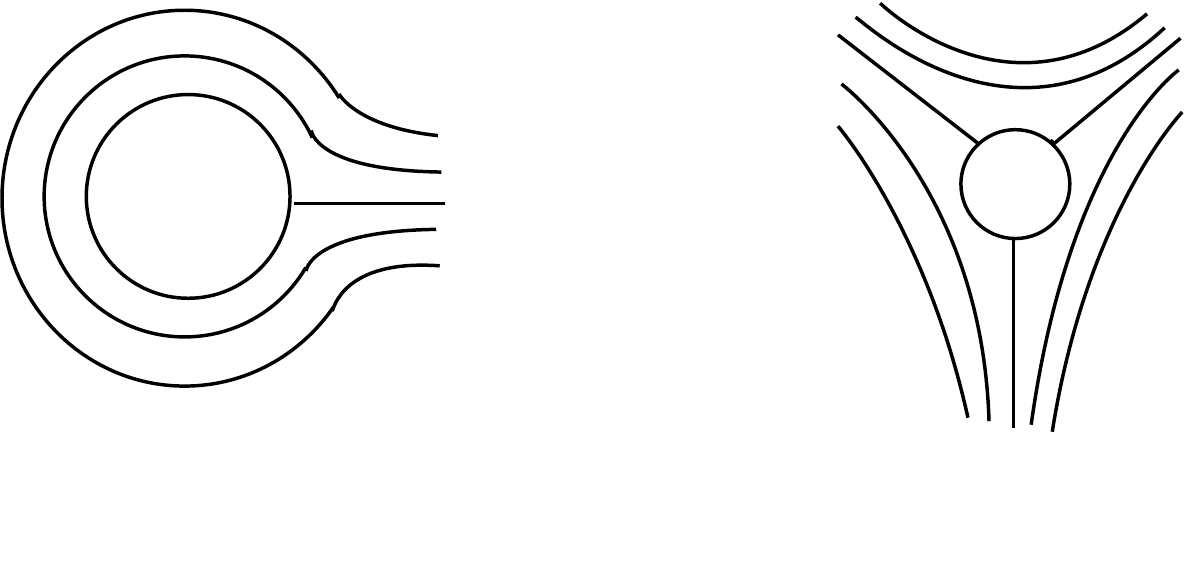}
\caption{Leaves of a singular foliation near a boundary component}
\label{prongs-fig2}
\end{center}
\end{figure}

Let 
\begin{equation}
\label{contraction-map}
c  \colon\thinspace F^b \rightarrow \overline{F^b}
\end{equation}
be the continuous map, where $\overline{F^b}$ is the closed surface obtained 
from $F^b$ by contracting $b$ boundary components to points $q_1,\dots,q_b$. 
As before, we will write $\sS$ for $c(\sS)$. 
Let $\sQ = \{q_1,\dots,q_b\}$.   The above definitions imply the following. 

\begin{lem}
\label{boundary-lem}  
The contraction map $c$ in \eqref{contraction-map} induces an isomorphism 
$$
c_*  \colon\thinspace \sM(F;\sS) \rightarrow \sM(\overline{F^b}; \sS,\sQ),
$$
which preserves the Thurston--Nielsen types of mapping classes. 
Furthermore, if $\sF$ is a singular foliation defined on $F$ which is $n$--pronged along 
a boundary component $A$ of $F$,  then the image of $\sF$ under 
$c_*$ has an $n$--pronged singularity at $c_*(A)$. 
\end{lem}
The isomorphism $c_*$ given in \fullref{boundary-lem} is handy in discussing 
mapping  classes coming from braids. 
Let $F$ be either $D$ or $S^2$.  
There is a natural homomorphism 
\begin{eqnarray}
\label{braidmonodromy-eqn} 
\sB(F;\sS)& \rightarrow& \sM(F;\sS)
\\
\beta &\mapsto& \phi_\beta \nonumber
\end{eqnarray}
defined as follows.  
Let $D_1,\dots,D_{s-1} \subset \mbox{int}(D)$ be disks with 
$D_i \cap D_j = \emptyset$ for $i \ne j$ such that $D_i$ 
contains two points $p_{i}$ and $p_{i+1}$ of $\sS$ and no other points of $\sS$. 
The action of a generator $\sigma_i$ of $\sB(F;\sS)$ 
is the mapping class in $\sM(F;\sS)$ that fixes the exterior of $D_i$ 
and rotates a closed line segment connecting $p_{i}$ and $p_{i+1}$ in $D_i$ by 
$180$ degrees in the counter-clockwise direction as in \fullref{braidmap-fig}. 

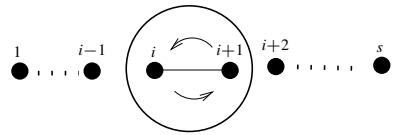
\begin{figure}[htbp]
\labellist\tiny
\pinlabel {$1$} [b] at 4 39
\pinlabel {$i{-}1$} [b] at 47 39
\pinlabel {$i$} [b] at 82 39
\pinlabel {$i{+}1$} [b] at 126 39
\pinlabel {$i{+}2$} [b] at 152 42
\pinlabel {$s$} [b] at 212 42
\endlabellist
\begin{center}
\includegraphics[width=2in]{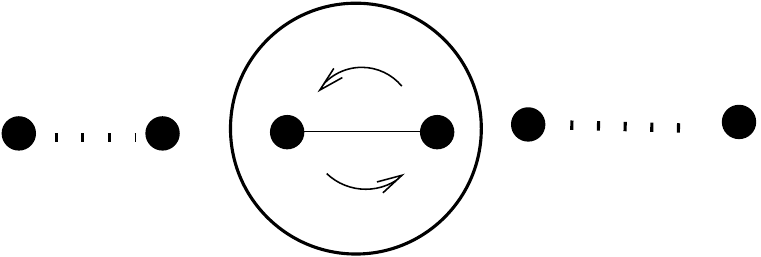}
\caption{Action of $\sigma_i$ as a homeomorphism of $F$}
\label{braidmap-fig}
\end{center}
\end{figure}

Given a braid $\beta \in \sB(D;\sS)$, let 
$\what{\beta}$ be its image in $\sB(S^2;\sS,\{p_\infty\})$ as in \eqref{widehat-eqn}.  
Then $c_*$  satisfies $c_*(\phi_\beta) = \phi_{\what{\beta}}$. 

The following useful lemma can be found in the book by Birman \cite{Birman74}.

\begin{lem}  
\label{braidcenter-lem}
If $\sS$ has cardinality $s$, then the kernel of the map 
\begin{eqnarray*}
\sB(D;\sS)& \rightarrow& \sM(S^2;\sS,\{p_\infty\})
\\
\beta &\mapsto& \phi_{\what{\beta}}
\end{eqnarray*}
is the center $Z(\sB(D;\sS))$ generated by a full twist braid 
$\Delta=(\sigma_1\ldots \sigma_{s-1})^s$. 
\end{lem}
We say that $\beta \in \sB(D;\sS)$ is {\it periodic} (respectively,  {\it reducible}, {\it pseudo-Anosov}), 
 if $\phi_{\what{\beta}} \in \sM(S^2;\sS,\{p_\infty\})$ is periodic 
 (respectively,  reducible, pseudo-Anosov). 
 In the pseudo-Anosov case, we set 
 $\lambda(\beta) = \lambda(\what{\beta}) =  \lambda(\phi_{\what{\beta}})$.   
 
Let $\wbar{\beta}$ be the image of $\what{\beta}$ in $\sB(S^2;\sS)$ 
under the forgetful map  in \eqref{overline-eqn}.  
Then \fullref{closure1-lem} implies that if $\what{\beta}$ and $\wbar{\beta}$ are 
pseudo-Anosov, we have 
$\lambda(\beta) = \lambda(\what{\beta}) \ge \lambda(\wbar{\beta})$, 
and by \fullref{closure2-lem}, 
the equality holds if $p_{\infty}$ is not a $1$--pronged singularity for 
the invariant foliations associated to the pseudo-Anosov representative of
$\phi_{\what{\beta}}$. 

\subsection{The braid spectrum}
\label{spectrum-section}

For any subset $\Gamma \subset \sM_g^s$, let $\Sigma(\Gamma)$ be the set of logarithms 
of dilatations coming from pseudo-Anosov elements of $\Gamma$.   
For any subset $\Gamma \subset \sB(D,s)$, define $\Sigma(\Gamma)$ in a similar way. 
Let $\wwhat{\sB}(D,s) \subset \sM_0^{s+1}$ be the image of $\sB(D,s)$ 
under the map in \fullref{braidcenter-lem},  
and ${\wwhat{\sB}}_{\mathrm{pA}}(D,s)$ the set of pseudo-Anosov elements
of $\wwhat{\sB}(D,s)$. 

\begin{prop}
\label{spectrum-prop} 
For $g \ge 1$, 
$$\Sigma (\sB(D,2g+1)) = \Sigma (\wwhat{\sB}(D,2g+1)) 
\subset \Sigma(\sM_{g,\mathrm{hyp}}) \subset \Sigma(\sM_g).$$
\end{prop}

\begin{proof}  
By using \fullref{braidcenter-lem}, it is easy to see that 
$\Sigma (\sB(D,2g+1)) = \Sigma (\wwhat{\sB}(D,2g+1))$. 

Let $\sS \subset$ int$(D)$ be a subset of $2g+1$ points, and $\wwhat{\sS} = \sS \cup \{p_\infty\}$. 
Let $F$ be the double cover of $S^2$ branched along  $\wwhat{\sS}$.   
Then $F$ has genus--$g$.  We will define a set map 
$$
\wwhat{{\sB}}_{\mathrm{pA}}(D,2g+1) \rightarrow \mathcal{M}_g
$$
whose image consists of hyperelliptic elements which preserves dilatation. 
Let $\phi \in \wwhat{{\sB}}_{\mathrm{pA}}(D,2g+1)$.  
Then $\phi$ has a pseudo-Anosov representative homeomorphism $\Phi$ 
that is unique up to conjugacy.   
Let $\Phi'$ be its lift to $F$ by the covering $F \rightarrow S^2$ 
with invariant foliations given by the lifts of the invariant foliations associated to $\Phi$.   
Then $\Phi'$ is pseudo-Anosov with the same dilatation as $\Phi$.  
Let $\phi'$ be its isotopy class. 
Then $\phi'$ defines a hyperelliptic, pseudo-Anosov mapping class in
$\sM(F;{\wwhat{\sS}}')$ 
with the same dilatation as $\phi$, 
where  ${\wwhat{\sS}}'$ is the preimage of $\wwhat{\sS}$ in $F$. 

Now consider the forgetful map $\sM(F;\wwhat{\sS'}) \rightarrow \sM(F;\emptyset) = \mathcal{M}_g$.  
The invariant foliations associated to $\Phi'$ 
have prong orders at $\wwhat{\sS}'$ that are divisible by the degree of the covering 
$F \rightarrow S^2$. 
Thus, the singularities of $\Phi'$ at $\wwhat{\sS}'$ are all even--pronged. 
It follows that by \fullref{closure2-lem}, the image of $\phi'$ under the forgetful map 
is pseudo-Anosov and has the same dilatations as $\phi'$. 
\end{proof}

\fullref{spectrum-prop} immediately implies the following. 

\begin{cor}
\label{spectrum-cor} 
For $g \ge 1$, $\delta (\sM_g) \leq \delta (\sM_{g,\mathrm{hyp}}) \leq  \delta (\sB(D,2g+1))$. 
\end{cor}

\subsection{Criterion for the pseudo-Anosov property}
\label{Bestvina-Handel-section}

What follows is a criterion for determining when $\beta \in \sB(D;\sS)$ is pseudo-Anosov 
(see Bestvina--Handel \cite{BH94}). 

Let $G$ be a finite graph embedded on an orientable surface $F$, possibly with self-loops, but 
no vertices of valence $1$ or $2$.  
Let  $\DE(G)$ be the set of oriented edges of $G$, 
 $\sE^{\mathrm{tot}}(G)$ the set of unoriented edges, and $\mathcal{V}(G)$ the set of vertices. 
 For $e \in \DE(G)$,  let $i(e)$ and $t(e)$ be the initial vertex and the terminal vertex respectively, 
and $\wbar{e}$  the same edge with opposite orientation. 
An {\it edge path} $\tau$ on $G$ is an oriented path $\tau=e_1\ldots e_\ell$, 
where $e_1,\dots,e_\ell \in \DE(G)$ satisfies $t(e_i) = i(e_{i+1})$ for $i=1,\dots,\ell-1$.  
One can associate a {\it fibered surface} $\mathcal{F}(G) \subset F$ 
with a projection $\pi \colon\thinspace  \mathcal{F}(G) \rightarrow G$ (\fullref{fig_fiber-surface}). 
The fibered surface $\mathcal{F}(G)$ is decomposed into arcs and into polygons modelled on 
$k$--junctions for $k \ge 1$. 
The arcs and the $k$--junctions are called  {\it decomposition elements}. 
Under $\pi$, the preimage of the vertices of valence $k$ of $G$ is the $k$--junctions, and 
the preimage of the edges of $G$ is the {\it strips} fibered by arcs, 
which are complementary components of the set of all junctions of $\mathcal{F}(G)$.

\begin{figure}[htbp]
\label{fig_vector-ab}
\labellist\small
\pinlabel {3--junction} [b] at 18 148
\pinlabel {arc} [bl] at 88 120
\pinlabel {$\pi$} [l] at 49 66
\pinlabel {$\pi$} [l] at 156 66
\pinlabel {1--junction} [l] at 209 133
\pinlabel {strip} [b] at 156 137
\endlabellist
\begin{center}
\includegraphics[width=2.5in]{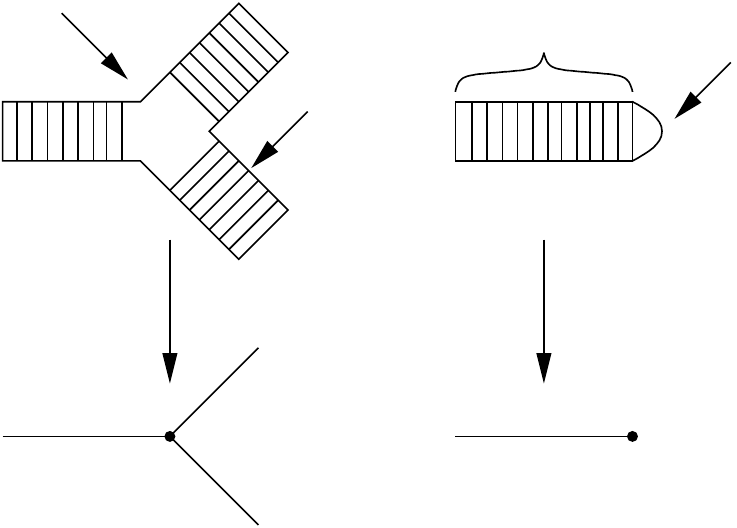}
\caption{Fibered surface}
\label{fig_fiber-surface}
\end{center}
\end{figure}
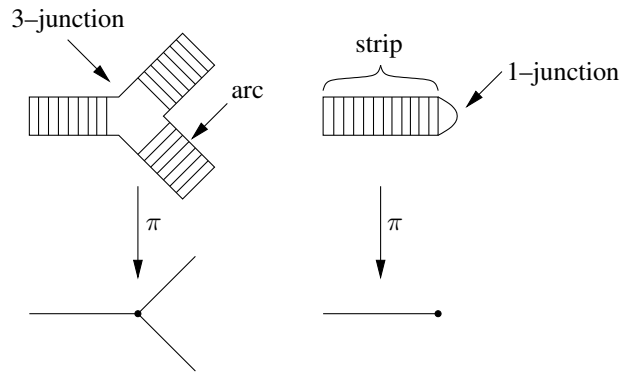

Let $G$ and $H$ be finite graphs embedded on $F$, and 
$f \colon\thinspace F \rightarrow F$ a homeomorphism. 
Assume that  $f$ maps each decomposition element of $\mathcal{F}(G)$  
into a decomposition element of $\mathcal{F}(H)$, 
and each junction of $\mathcal{F}(G)$ into a junction of $\mathcal{F}(H)$. 
Then $f$ induces a graph map $\g \colon\thinspace G \rightarrow H$ which sends 
vertices of $G$ to vertices of $H$, and each edge of $G$ to an edge path of $H$. 
Under this assumption with the case $G=H$, we say that $\mathcal{F}(G)$ {\it carries} $f$. 

Let $V^{\mathrm{tot}}(G)$ be the vector space 
of formal sums $\sum_{i=1}^n a_i e_i,$ 
where $a_i \in \R$ and $e_i \in \sE^{\mathrm{tot}}(G)$.  
Any edge path on $G$ determines an element of $V^{\mathrm{tot}}(G)$ by 
treating each oriented edge as an unoriented edge with coefficient 1, regardless of orientation. 
For a graph map $\g \colon\thinspace G \rightarrow H$, 
define the {\it  transition matrix}  for $\g$ to be the transformation 
$$\sT_\g^{\mathrm{tot}} \colon\thinspace V^{\mathrm{tot}}(G) \rightarrow V^{\mathrm{tot}}(H)$$
taking each $e \in \sE^{\mathrm{tot}}(G)$ to 
$\g(e)$ considered as an element of $V^{\mathrm{tot}}(H)$.  

We  now  restrict to the case $F = D$. 
Let ${\sS}= \{p_1,\dots,p_s\} \subset \mbox{int}(D)$ be a set of marked points, and 
$P_i$ a small circle centered at $p_i$ whose interior disk does not contain any other points of $\sS$. 
We set $P= \bigcup_{i=1}^s P_i$. 
Choose a finite graph $G$ embedded on $D$ 
that is homotopy equivalent to $D \setminus \sS$ such that $P$ is a subgraph of  $G$.  
Given $\beta \in \sB(D;\sS)$, suppose that a fibered surface $\mathcal{F}(G)$ carries some 
homeomorphim $f$ of $\phi_{\beta} \in \sM(D; \sS)$. 
Then the graph map $\g \colon\thinspace G \rightarrow G$, 
called the {\it induced graph map for}  $\phi_{\beta}$, 
preserves $P$ set-wise. 
Let pre$P$ be the set of edges $e \in \sE^{\mathrm{tot}}(G)$  
such that $\g^k(e)$ is contained in $P$ for some $k \ge 1$. 
By the definition of pre$P$, the transition matrix ${\sT}_\g^\mathrm{tot}$ has the following form: 
$$
{\sT}_\g^{\mathrm{tot}}= 
\left(\begin{array}{ccc}\mathcal{ P}& \mathcal{A} & \mathcal{B} \\0 & \mathcal{Z} & \mathcal{C} \\0 & 0 & \sT\end{array}\right), 
$$
where $\mathcal{P}$ and $\mathcal{Z}$ are the transition matrices  associated to 
$P$ and pre$P$ respectively, and 
$\sT$ is the transition matrix associated to the rest of edges $\sE(G)$ called {\it real edges}. 
Let $V(G)$ be the subspace of $V^{\mathrm{tot}}(G)$ spanned by $\sE(G)$. 
The matrix $\sT$ is the restriction of ${\sT}_\g^\mathrm{tot}$ to $V(G)$ and is called 
the {\it transition matrix with respect to the real edges}. 
The spectral radius of $\sT$ is denoted by $\lambda(\sT)$.  

Given a graph map $\g \colon\thinspace G \rightarrow G$, define the {\it derivative} 
$D_\g \colon\thinspace \DE(G) \rightarrow \DE(G)$ as follows: 
For  $e \in \DE(G)$, write $\g(e) = e_1e_2 \ldots e_\ell$, where $e_i \in \DE(G)$. 
The image of $e$ under $D_\g$ is defined by the initial edge $e_1$. 

A graph map $\g \colon\thinspace G \rightarrow G$ is {\it efficient} 
if for any $e \in \DE(G)$ and any $k \ge  0$, 
$\g^k(e) = e_1 e_2 \ldots e_j$ satisfies $D_\g(\wbar{e}_i) \neq D_\g(e_{i+1})$ 
for all $i=1,\dots,j-1$.
We also say in this case that $\g^k$ has no {\it back track} for any $k \ge 0$. 

A nonnegative square matrix  $M$ is {\it irreducible} 
if for every set of indices $i,j$, there is an integer $n_{i,j} > 0 $ such that the $(i,j)$th entry of 
$M^{n_{i,j}}$ is strictly positive.  

\begin{thm}[Bestvina--Handel \cite{BH94}]
\label{BH-thm}
Let $\beta \in \sB(D; \sS)$, and $\g\co G \rightarrow G$ the induced graph map for $\phi_{\beta}$.   
Suppose that

{\rm(BH1)}\qua $\g$ is efficient, and \\
{\rm(BH2)}\qua the transition matrix $\sT$ with respect to the real edges is 
irreducible with $\lambda(\mathcal{T})>1$. 

Then $\beta$ is pseudo-Anosov with dilatation equal to $\lambda(\sT)$. 
\end{thm}

It is not hard to check that the criterion of \fullref{BH-thm} behaves well under conjugation
of maps.  For the case of braids, this yields the following. 

\begin{lem}
\label{conj-lem}  
Let $\alpha_1 \in \sB(D;\sS)$. 
Suppose that a fibered surface $\mathcal{F}(G)$ carries a homeomorphism $f \in \phi_{\alpha_1}$, 
and let $\g_1 \colon\thinspace G \rightarrow G$ be the induced graph map for $\phi_{\alpha_1}$. 
We now consider a conjugate braid $\alpha_2$ with $\alpha_2 = \gamma \alpha_1 \gamma^{-1}$, 
and we take any homeomorphism $h \in \phi_{\gamma}$. 
Then a fibered surface $\mathcal{F}(h(G))$ carries a homeomorphism 
$hfh^{-1} \in \phi_{\alpha_2}$, and hence 
$hfh^{-1}$ induces a graph map $\g_2\co h(G) \rightarrow h(G)$, 
which is the induced graph map for $\phi_{\alpha_2}$. 
If $\g_1$ satisfies  {\rm(BH1)} and {\rm(BH2)}, then $\g_2$ also satisfies  {\rm(BH1)} and {\rm(BH2)}. 
\end{lem}

Let $\phi \in \sM(D;\sS)$ be pseudo-Anosov, and $\g \colon\thinspace G \rightarrow G$ 
the induced graph map for $\phi$ satisfying satisfying {\rm(BH1)} and {\rm(BH2)}.  
We construct an associated {\it train track} obtained by {\it graph smoothing} given  as follows: 
Let $\mathcal{E}_v \subset \DE(G)$ be the set of oriented edges of $G$ 
emanating from a vertex $v$.  For $e_1,e_2 \in \mathcal{E}_v$, 
$e_1$ and $e_2$  are  {\it equivalent}  if $D_\g^k(e_1)= D_\g^k(e_2)$ for some $k \ge 1$. 
A {\it gate} is an  equivalence class in $\sE_v$. 
The  train track $\tau_\g$ associated to $\g$ is constructed using the following steps: 

\textbf{Step T1}\qua
Deform each pair of equivalent edges $e_i,e_j \in \mathcal{E}_v$ in a small neighborhood of $v$ 
so that $e_i$ and $e_j$ are tangent at $v$. 

\textbf{Step T2}\qua
Insert  a small disk $N_v$ at each vertex $v$. 
For each gate $\gamma$, assign a point $p(\gamma)$ on the boundary of $N_v$. 

\textbf{Step T3}\qua
If, for some edge $e$ of $G$ and some $k \ge 1$, 
$\g^k(e)$ contains consecutive edges $\wbar{e}_je_{\ell}$ $(e_j,e_{\ell} \in {\mathcal E}_v)$ 
such that $\gamma_j= [e_j]$ and $\gamma_{\ell}= [e_{\ell}]$ with $\gamma_j \ne \gamma_{\ell}$, then 
join $p(\gamma_j)$ and $p(\gamma_{\ell})$ by a smooth arc in $N_v$ 
satisfying the following: 
The arc intersects the boundary of $N_v$ transversally at $p(\gamma_j)$ 
and $p(\gamma_{\ell})$, and no two such arcs intersect in the interior
of $N_v$.

For example, let $v$ be the initial vertex of four edges $e_1,e_2,e_3,e_4$.  
Assume that there are three gates $\gamma_1=[e_1]=[e_2]$, $\gamma_2=[e_3]$ and 
$\gamma_3 = [e_4]$, 
and that there are edges $f_1$ and $f_2$ of $G$ such that 
$\g^r(f_1) = \ldots \wbar{e}_2e_4 \ldots$ and 
$\g^s(f_2) = \ldots \wbar{e}_3e_4\ldots$ for some $r,s \ge 1$. 
Then \fullref{fig_gate}(a) shows Step T1 
applied to $e_1$ and $e_2$, \fullref{fig_gate}(b) shows Step T2 
applied to $\gamma_1, \gamma_2$ and $\gamma_3$, and \fullref{fig_gate}(c) 
shows Step T3, which yields arcs connecting $p(\gamma_1)$ to $p(\gamma_3)$, and 
$p(\gamma_2)$ to $p(\gamma_3)$.  

\begin{figure}[htbp]
\labellist\small
\hair=1pt
\pinlabel {$e_1$} [tr] at 97 610
\pinlabel {$e_2$} [tl] at 199 610
\pinlabel {$e_3$} [bl] at 199 714
\pinlabel {$e_4$} [br] at 97 714
\pinlabel {$v$} at 148 634
\pinlabel {$N$} at 405 668
\pinlabel {$p(\gamma_1)$} [l] at 405 642
\pinlabel {$p(\gamma_2)$} [tl] at 426 690
\pinlabel {$p(\gamma_3)$} [tr] at 381 690
\pinlabel {(a)} [b] at 224 540
\pinlabel {(b)} [b] at 400 540
\pinlabel {(c)} [b] at 530 540
\endlabellist
\begin{center}
\includegraphics[height=1.3in]{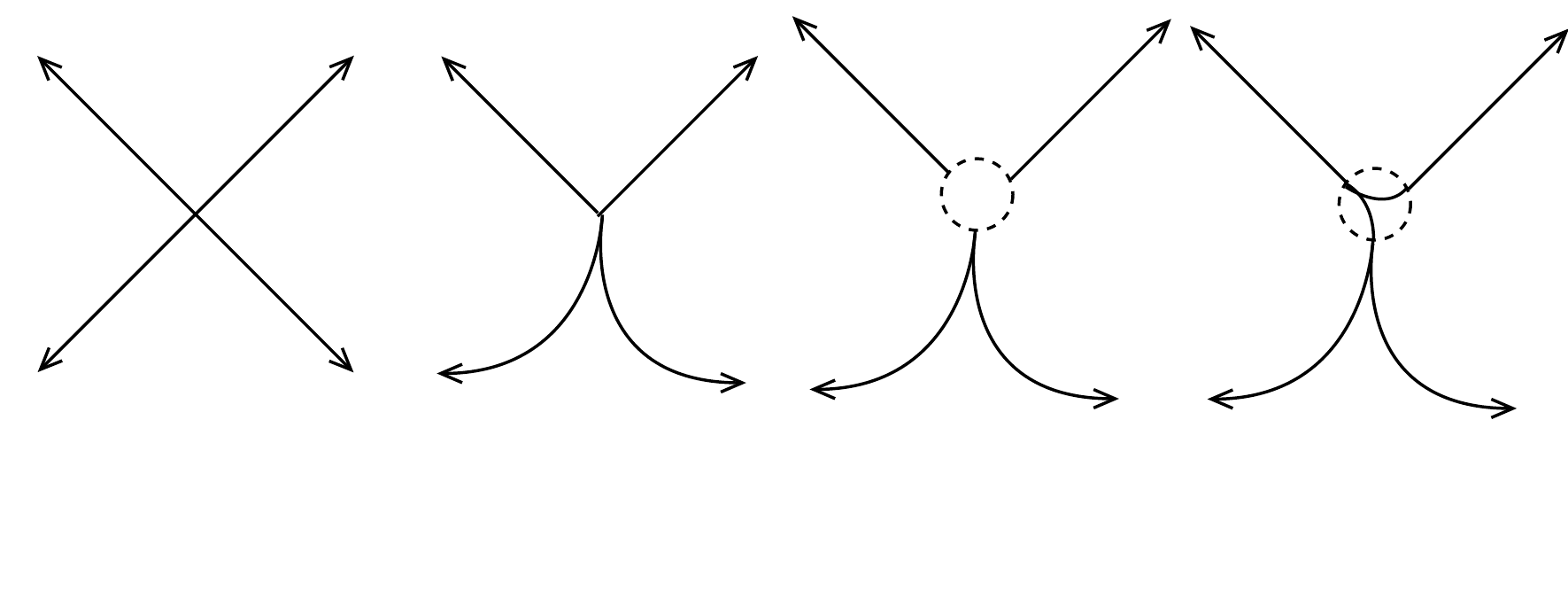}
\caption{Example of a graph smoothing}
\label{fig_gate}
\end{center}
\end{figure}
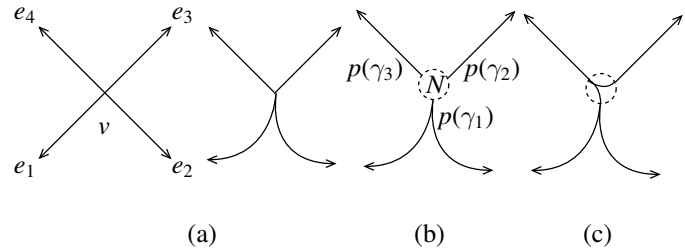

The arcs constructed in Step T3 are called {\it infinitesimal edges}, 
and the points $p(\gamma)$ which join two infinitesimal edges are 
called {\it cusps} of the train track. 

If $\phi \in \sM(D;\sS)$ is pseudo-Anosov, and $\g \colon\thinspace G \rightarrow G$ is the induced 
graph map for $\phi$ satisfying {\rm(BH1)} and {\rm(BH2)}, then $\tau_\g$ constructed above 
determines the invariant foliations $\sF^\pm$ 
associated to the pseudo-Anosov representative $\Phi$ of $\phi$. 
In particular,  the number of prongs at the singularities of $\sF^\pm$ 
can be found in terms of $\tau_\g$. 
Each connected component $A$ of $D \setminus \tau_\g$ is either homeomorphic to an open 
disk, or is a half-open annulus, one of whose boundaries is the boundary $\partial D$ of $D$.  
In the former case, the boundary of the closure of the connected component is a finite union of edges 
and vertices of $\tau_\g$.   If two of these edges meet at a cusp, then that cusp is said to 
 {\it belong to} $A$.    In the latter case, the closure of $A$ has two boundary components. 
The boundary component which is not $\partial D$ 
is a finite union of edges and vertices of $\tau_\g$, and if two of these edges meet at 
a cusp, we call the cusp an {\it exterior cusp} of $\tau_\g$. 

\begin{lem}
\label{prongs-lem} 
Let $A$ be a connected component of $D \setminus \tau_\g$. 
If $A$ is an open disk, then there is one $k$--pronged singularity of $\sF^\pm$ in $A$, 
where $k$ is the number of cusps of $\tau_\g$ belonging to $A$. 
If $A$ is a half-open annulus, then $\partial D$ is $k$--pronged, 
where $k$ is the number of exterior cusps of $\tau_\g$. 
\end{lem}

\section{Main examples}
\label{main-section}

This section contains properties of $\beta_{m,n}$ and $\sigma_{m,n}$.  
In \fullref{symmetry-section}, we show that the Thurston--Nielsen types of 
$\beta_{m,n}$ and $\sigma_{m,n}$ do not depend on the order of $m$ and $n$. 
In \fullref{Graphmap-section}, we find the Thurston--Nielsen types 
of $\beta_{m,n}$ and $\sigma_{m,n}$, and in \fullref{dilatations-section}, 
we compute their dilatations in the pseudo-Anosov cases. 
\fullref{TT-section} gives the train tracks for $\phi_{\beta_{m,n}}$ and $\phi_{\sigma_{m,n}}$ 
In \fullref{Salem-Boyd-section}, 
we apply properties of Salem--Boyd sequences to find the least dilatation 
among $ \lambda(\sigma_{m,n})$ and $\lambda(\beta_{m,n})$  for $m+n = 2g$ fixing $g \ge 2$. 
We also give bounds on these dilatations. 

\subsection{Symmetries of $\beta_{m,n}$ and $\sigma_{m,n}$}
\label{symmetry-section}

Consider the  braid $\beta_{m,n}^+ \in \sB(D;\sS,\{p\},\{q\})$ drawn in \fullref{Bmn-sym-fig}(a). 

\begin{figure}[htbp]
\labellist\small
\pinlabel {$p$} [b] at 5 135
\pinlabel {$m$} [b] at 46 139
\pinlabel {$n$} [b] at 98 139
\pinlabel {$q$} [b] at 130 135
\pinlabel {$p$} [b] at 227 137
\pinlabel {$m$} [b] at 255 141
\pinlabel {$n$} [b] at 305 141
\pinlabel {$q$} [b] at 356 137
\pinlabel {$p$} [b] at 434 139
\pinlabel {$n$} [b] at 473 142
\pinlabel {$m$} [b] at 527 142
\pinlabel {$q$} [b] at 559 135
\pinlabel {(a)} [b] at 80 0
\pinlabel {(b)} [b] at 300 0
\pinlabel {(c)} [b] at 495 0
\endlabellist
\begin{center}
\includegraphics[width=5in]{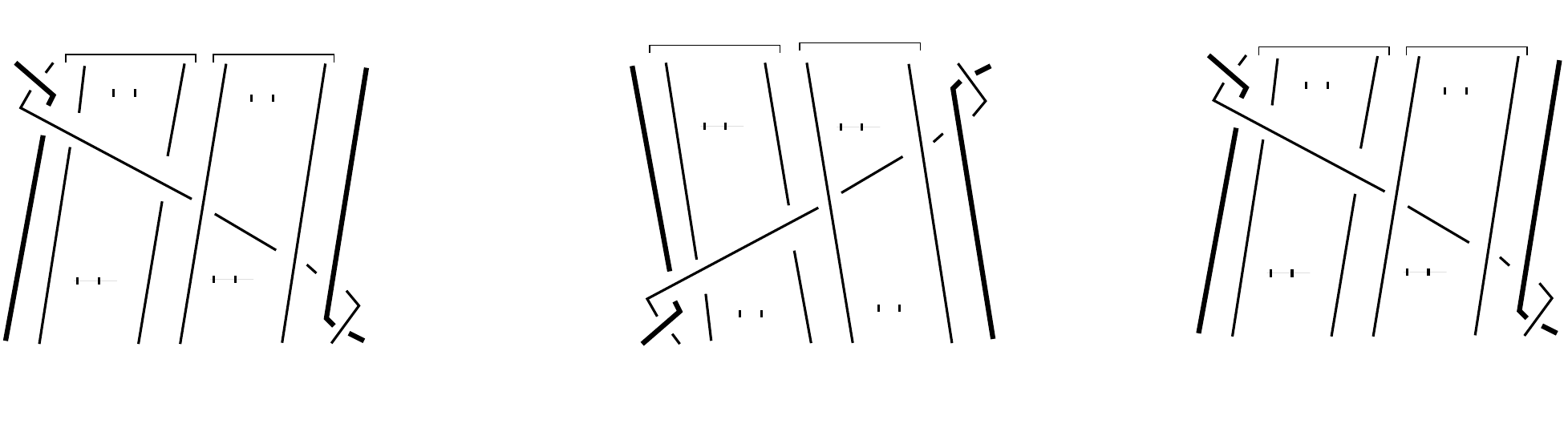}
\caption{Symmetry of $\beta_{m,n}^+$}
\label{Bmn-sym-fig}
\end{center}
\end{figure}
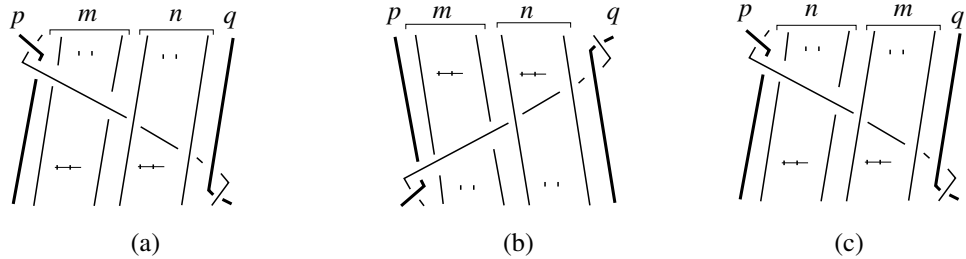

\begin{lem}
\label{Bmnplus-conj-lem} 
The braid $\beta_{n,m}^+$ is conjugate to the inverse of $\beta_{m,n}^+$. 
\end{lem}

\begin{proof} 
The inverse of $\beta_{m,n}^+$ is drawn in \fullref{Bmn-sym-fig}(b). 
Assume without loss of generality that the points of $\sS \cup \{p\}\cup\{q\}$ are 
evenly spaced along a line $\ell$. 
Let $\eta \in \sB(D;\sS\cup\{p,q\})$ be the braid obtained by 
a half-twist of $\ell$ around the barycenter of $\sS \cup \{p\}\cup\{q\}$.  Then 
conjugating the inverse of $\beta_{m,n}^+$ by $\eta$ in $\sB(D;\sS\cup\{p,q\})$ yields 
$\beta_{n,m}^+$ shown in \fullref{Bmn-sym-fig}(c). 
\end{proof}

\begin{lem}
\label{Bmnplus-sym-lem} 
The braid $\beta_{m,n}$ is the image of $\beta_{m,n}^+$ under the forgetful map 
$$\sB(D;\sS,\{p\},\{q\}) \rightarrow \sB(D;\sS),$$
and hence $\beta_{n,m}$ is conjugate to $\beta_{m,n}^{-1}$. 
\end{lem}

\begin{proof} 
Compare \fullref{spherical-fig}(a) with \fullref{Bmn-sym-fig}(a)  to get the first 
part of the claim.  Since homomorphisms preserve inverses and conjugates, the 
rest follows from \fullref{Bmnplus-conj-lem}. 
\end{proof}

\fullref{Bmnplus-sym-lem} together with the homomorphism in \eqref{braidmonodromy-eqn} 
shows the following. 

\begin{lem}
\label{Bmnplus-lem}  
The mapping class $\phi_{\beta_{n,m}}$  is conjugate to $\phi_{\beta_{m,n}}^{-1}$. 
\end{lem}

\begin{prop}
\label{Bmn-conj-prop} 
The Thurston--Nielsen type of $\beta_{n,m}$ is the same as that of $\beta_{m,n}$. 
\end{prop}

\begin{proof} 
The Thurston--Nielsen type of a mapping class is preserved under inverses and conjugates.  
Thus, the claim follows from \fullref{boundary-lem} and  \fullref{Bmnplus-lem}. 
\end{proof}

We now turn to $\sigma_{m,n}$. 
Let $\what{\beta}_{m,n}^+$ and $\nu$ be the spherical braids 
drawn in Figures~\ref{SphericalBmnplus-sym-fig}(a) and \ref{ConjBraid-fig} respectively. 

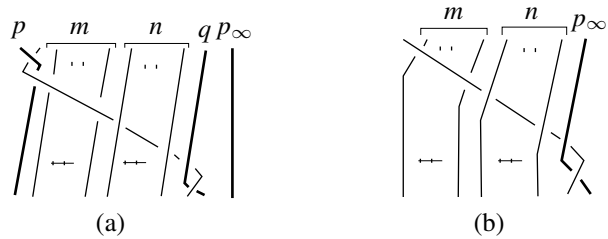
\begin{figure}[htbp]
\begin{center}
\labellist\small
\pinlabel {$p$} [b] at 5 134
\pinlabel {$m$} [b] at 46 137
\pinlabel {$n$} [b] at 98 137
\pinlabel {$q$} [b] at 131 131
\pinlabel {$p_{\infty}$} [b] at 152 133
\pinlabel {$m$} [b] at 300 147
\pinlabel {$n$} [b] at 353 147
\pinlabel {$p_{\infty}$} [b] at 393 139
\pinlabel {(a)} [b] at 67 0
\pinlabel {(b)} [b] at 324 0
\endlabellist
\includegraphics[height=1.25in]{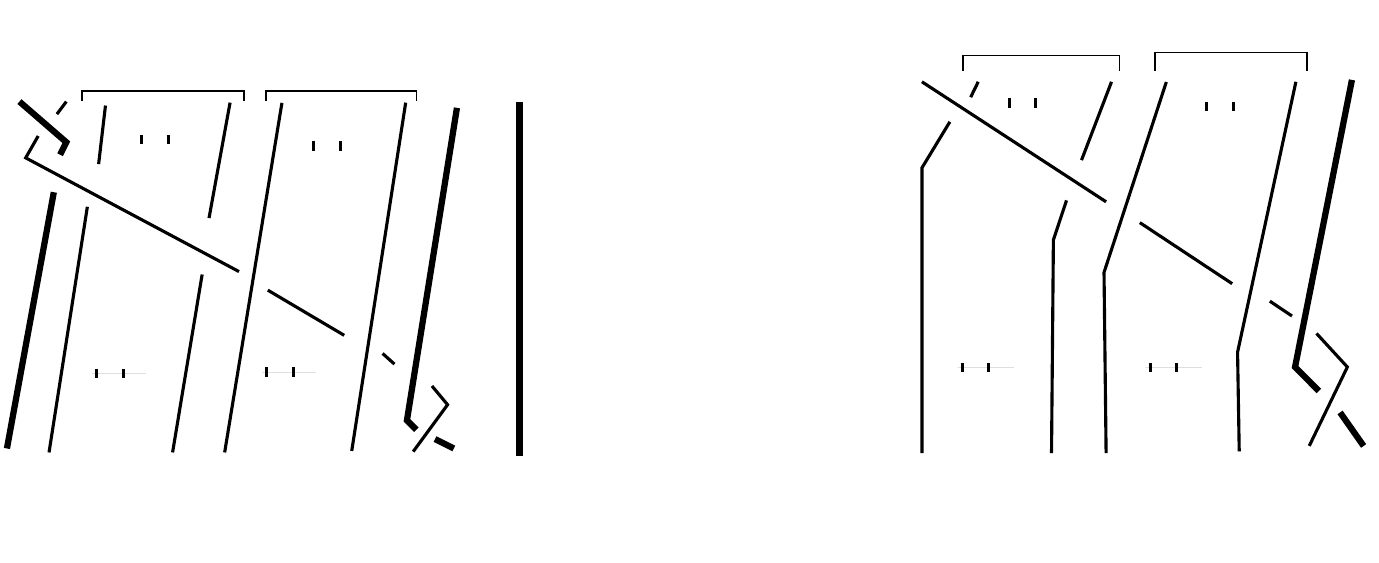}
\caption{Spherical braids (a) $\what\beta_{m,n}^+$ and (b) $\what\sigma_{m,n}$}
\label{SphericalBmnplus-sym-fig}
\end{center}
\end{figure}

\begin{lem}
\label{BmnSmn-lem} 
The spherical braid $\what{\sigma}_{m,n}$ is the image of 
$\nu\what{\beta}_{m,n}^+\nu^{-1}$ under the forgetful map 
$\sB(S^2;\sS,\{p\},\{q\},\{p_\infty\}) \rightarrow \sB(S^2;\sS,\{p_\infty\})$, 
and hence $\what{\sigma}_{n,m}$ is conjugate to $\what{\sigma}_{m,n}^{-1}$. 
\end{lem}

\begin{proof} 
Compare \fullref{SphericalBmnplus-sym-fig}(a) and \fullref{SphericalBmnplus-sym-fig}(b) 
to get the first part of the claim. 
The rest follows by using the same argument in the proof of \fullref{Bmnplus-sym-lem}. 
\end{proof}

\begin{rem} 
In the statement of \fullref{BmnSmn-lem}, $\nu$ could be replaced 
by any braid which is the identity on $p$ and $\sS$, 
and interchanges $q$ and $p_\infty$. 
\end{rem}

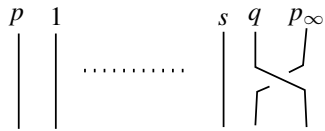
\begin{figure}[htbp]
\labellist\small
\pinlabel {$p$} [b] at 2 50
\pinlabel {$1$} [b] at 22 50
\pinlabel {$s$} [b] at 107 50
\pinlabel {$q$} [b] at 123 50
\pinlabel {$p_{\infty}$} [b] at 150 50
\endlabellist
\begin{center}
\includegraphics[height=0.75in]{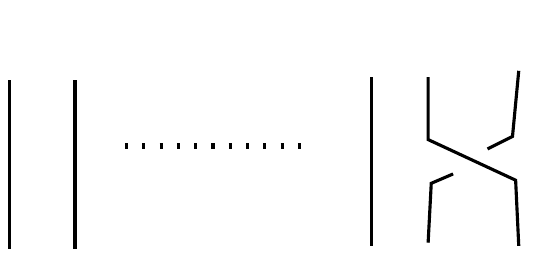}
\caption{Spherical braid $\nu$: Switching the roles of $p_\infty$ and $q$}
\label{ConjBraid-fig}
\end{center}
\end{figure}

\begin{lem}
\label{Smn=Snm-lem} 
The mapping class $\phi_{\sigma_{n,m}}$ is conjugate to $\phi_{\sigma_{m,n}}^{-1}$. 
\end{lem}

\begin{proof} 
By \fullref{BmnSmn-lem}, 
$\phi_{\what{\sigma}_{n,m}}$ is conjugate to $\phi_{\what{\sigma}_{m,n}}^{-1}$. 
Since the contraction map $c$ in \eqref{contraction-map} induces the isomorphism $c_*$ 
on mapping class groups, the claim follows. 
\end{proof}

\fullref{Smn=Snm-lem} immediately shows the following. 

\begin{prop}
\label{Smn-sym-prop}  
The Thurston--Nielsen type of $\sigma_{m,n}$ is the same as that for $\sigma_{n,m}$. 
\end{prop}

\subsection{Graph maps}
\label{Graphmap-section}

\begin{thm} 
\label{Bmn-classification-thm}  
The braid $\beta_{m,n}$ is pseudo-Anosov for all $m,n \ge 1$, and 
$\lambda(\beta_{m,n}) = \lambda(\beta_{n,m})$. 
\end{thm} 

Let $G_{m,n}$ be the graph with vertices $1, \ldots, m+n+1,p$ and $q$ in 
\fullref{BmnGraph-fig}(left). 
Consider the graph map $\g= \g_{m,n}  \colon\thinspace G_{m,n} \rightarrow G_{m,n}$  given in 
\fullref{BmnGraph-fig}, where the ordering of the loop edges of 
$G_{m,n}$ corresponds to the left-to-right ordering of $\beta_{m,n}$. 
We denote the oriented edge with the initial vertex $a$ and the terminal vertex $b$ by $e(a,b)$. 

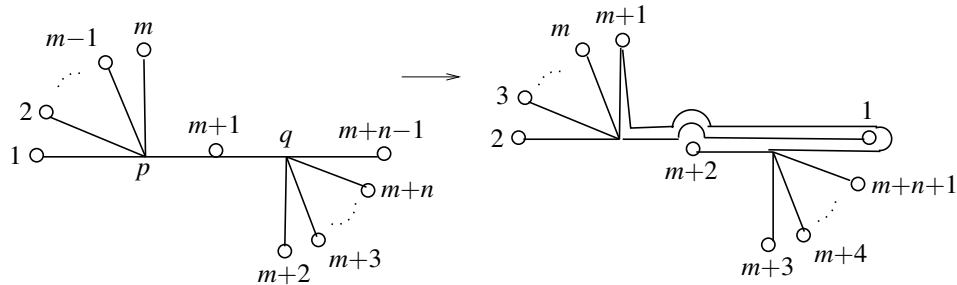
\begin{figure}[htbp]
\labellist\small
\pinlabel {$1$} [r] at 17 72
\pinlabel {$2$} [r] at 23 97
\pinlabel {$m{-}1$} [br] at 59 129
\pinlabel {$m$} [b] at 81 136
\pinlabel {$p$} [t] at 81 72
\pinlabel {$m{+}1$} [b] at 121 79
\pinlabel {$q$} [b] at 161 72
\pinlabel {$m{+}2$} [t] at 160 15
\pinlabel {$m{+}3$} [tl] at 182 22
\pinlabel {$m{+}n$} [l] at 210 52
\pinlabel {$m{+}n{-}1$} [b] at 215 76
\pinlabel {$2$} [r] at 288 82
\pinlabel {$3$} [r] at 291 106
\pinlabel {$m$} [br] at 324 136
\pinlabel {$m{+}1$} [b] at 350 141
\pinlabel {$m{+}2$} [t] at 389 73
\pinlabel {$m{+}3$} [t] at 432 19
\pinlabel {$m{+}4$} [tl] at 455 25
\pinlabel {$m{+}n{+}1$} [l] at 486 55
\pinlabel {$1$} [b] at 488 88
\endlabellist
\begin{center}
\includegraphics[width=5in]{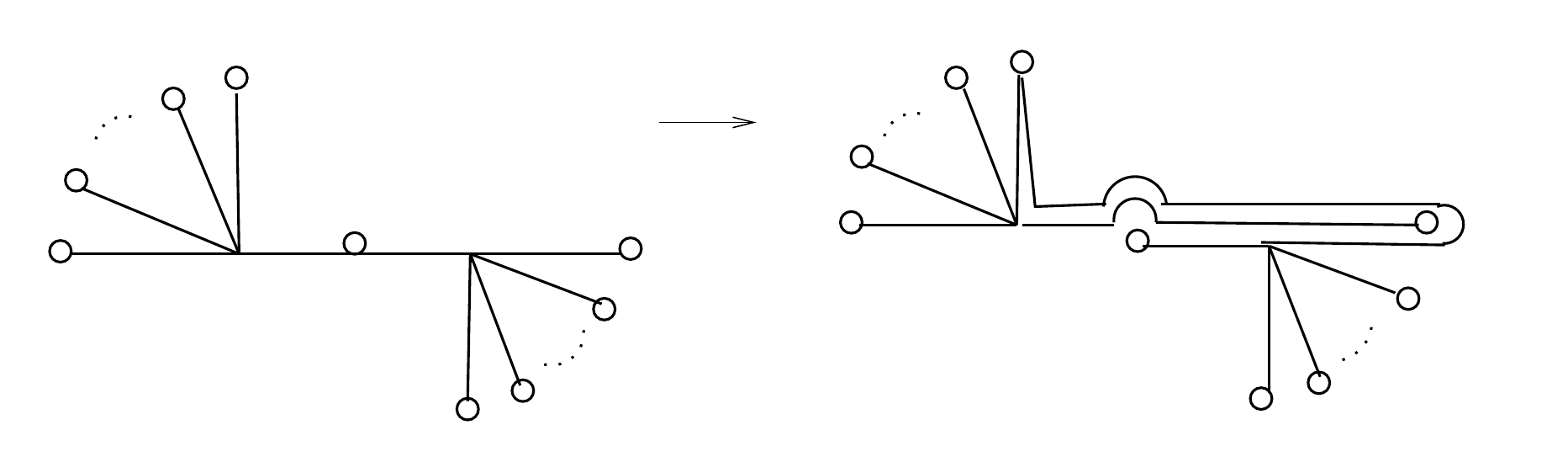}
\caption{Graph map $\g_{m,n}$ for $\phi_{\beta_{m,n}}$}
\label{BmnGraph-fig}
\end{center}
\end{figure}

\begin{prop}
\label{BmnGraph-prop}
The graph map $\g_{m,n} \colon\thinspace  G_{m,n} \rightarrow G_{m,n}$ is the induced graph map 
for $\phi_{\beta_{m,n}}$ satisfying {\rm(BH1)} and {\rm(BH2)}. 
\end{prop}

\begin{proof} 
It is easy to see that the fibered surface $\mathcal{F}(G_{m,n})$ carries a homeomorphism of 
$\phi_{\beta_{m,n}}$, and hence $\g_{m,n} \colon\thinspace G_{m,n} \rightarrow G_{m,n}$ 
is the induced graph map for $\phi_{\beta_{m,n}}$. 

As shown in \fullref{BmnGraph-fig}, any back track must occur at 
$e(p,m)$, that is, if $\g^k$ has back tracks, and $k$ is chosen minimally, 
then there is an edge $e \in \DE(G_{m,n})$ such that 
\begin{eqnarray}
\label{badedge-eqn}
\g^{k-1}(e) = \ldots \wbar{e}_1\cdot e_2 \ldots\ \mbox{with}\ 
D_\g(e_1) = D_\g(e_2) = e(m,p). 
\end{eqnarray}
This implies that $\wbar{e}_1 = e(p,m+1)$ and $e_2 = e(m+1,q)$ 
(or $\wbar{e}_1=e(q,m+1)$ and $e_2=e(m+1,p)$).    As can be seen by 
\fullref{BmnGraph-fig}, one can verify that 
there can be no edge of the form given in \eqref{badedge-eqn}. 
This proves {\rm(BH1)}.  
 
To prove {\rm(BH2)}, it suffices to note that $\g^{m+n}(e(q,m+1))$ crosses all non-loop edges 
 of $G_{m,n}$ in either direction, and for any non-loop edge $e$ of $G_{m,n}$, $\g^k(e)$ crosses $e(q,m+1)$ in either direction for some $k \ge 1$. 
 \end{proof}

\begin{proof}[Proof of \fullref{Bmn-classification-thm}]
By \fullref{BmnGraph-prop}, $\beta_{m,n}$ is pseudo-Anosov
for all $m,n \ge 1$.  By \fullref{Bmnplus-lem}, we have
$\lambda(\beta_{m,n})=\lambda(\beta_{n,m})$.
\end{proof}

We now turn to $\sigma_{m,n}$. 

\begin{figure}[htbp]
\labellist\small
\pinlabel {$a$} [b] at 32 109
\pinlabel {$b$} [t] at 62 18
\pinlabel {$s{-}b{-}1$} [b] at 112 113
\pinlabel {$p_{\infty}$} [b] at 143 104
\endlabellist
\begin{center}
\includegraphics[height=1.25in]{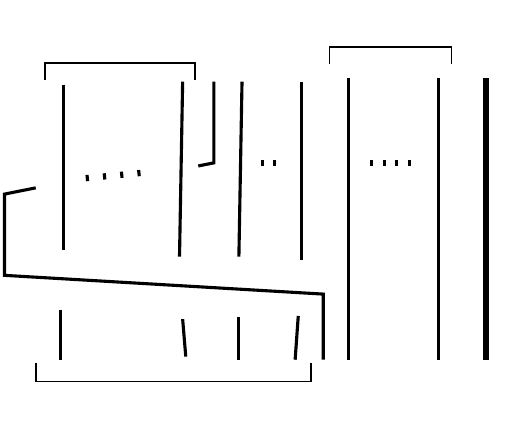}
\caption{Spherical braid $\tau_{a,b} \in \sB(S^2;\sS,\{p_\infty\})$}
\label{wrapbraid-fig}
\end{center}
\end{figure}

\begin{thm} 
\label{Smn-classification-thm}  
The braid $\sigma_{m,n}$ is 
pseudo-Anosov for all $m,n \ge 1$ satisfying $|m-n| \ge 2$.  
In these cases $\lambda(\sigma_{m,n}) = \lambda(\sigma_{n,m})$. 
For any $m \ge 1$,  $\sigma_{m,m}$ is periodic, 
and $\sigma_{m,m+1}$ and $\sigma_{m+1,m}$ are reducible.
\end{thm} 
 In light of \fullref{Smn-sym-prop}, 
 we will consider only $\sigma_{m,n}$ when  $n \ge m \ge 1$.   
 To prove \fullref{Smn-classification-thm}, we first redraw $\sigma_{m,n}$ 
 in a conjugate form  using induction. 
 Let $\tau_{a,b}$ be the spherical braid drawn in \fullref{wrapbraid-fig}.  
  Roughly speaking, conjugation by $\tau_{a,b}$ on  $\sigma_{m,n}$ 
  is the same as passing a strand counterclockwise around the other 
 strands, and then compensating below after a shift of indices. 

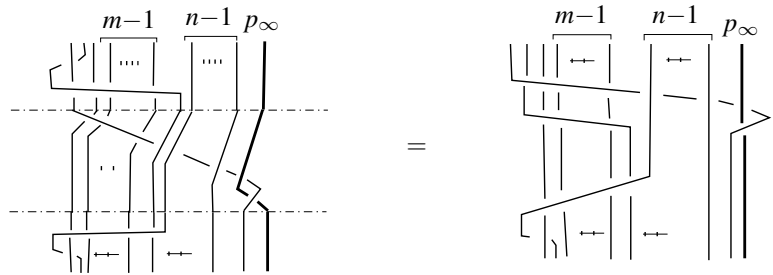
\begin{figure}[htbp]
\labellist\small
\pinlabel {$m{-}1$} [b] at 77 151
\pinlabel {$n{-}1$} [b] at 128 153
\pinlabel {$p_{\infty}$} [b] at 162 150
\pinlabel {$m{-}1$} [b] at 365 154
\pinlabel {$n{-}1$} [b] at 425 154
\pinlabel {$p_{\infty}$} [b] at 467 147
\pinlabel {$=$} at 260 81
\endlabellist
\begin{center}
\includegraphics[width=4in]{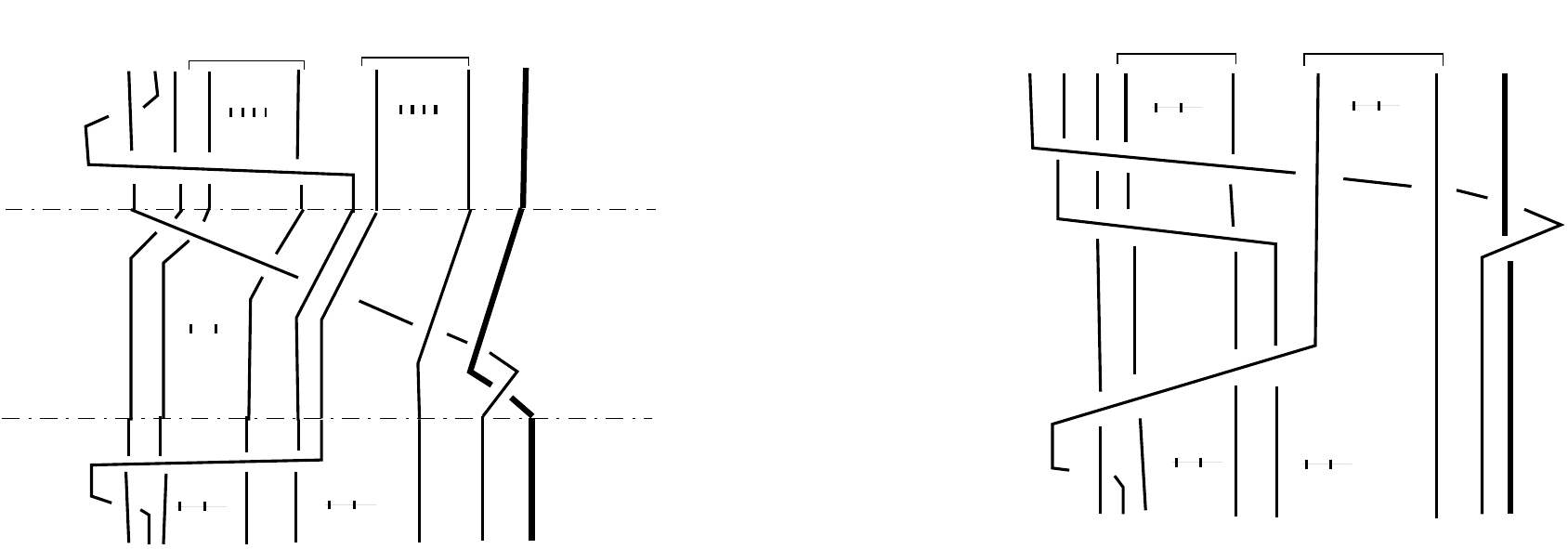}
\caption{Conjugating $\what\sigma_{m,n}$: Initial step}
\label{firstwrap-fig}
\end{center}
\end{figure} 
 Let ${\what{\sigma}_{m,n}}^{(0)} = \what{\sigma}_{m,n}$ be the image of $\sigma_{m,n}$ in $\sB(S^2;\sS,\{p_\infty\})$ as drawn in \fullref{SphericalBmnplus-sym-fig}(b). 
Let 
$$
{\what{\sigma}_{m,n}}^{(1)} = \tau_{1,m+1} \what{\sigma}_{m,n} \tau_{1,m+1}^{-1} 
$$
shown in \fullref{firstwrap-fig}.  The inductive step is illustrated in \fullref{inductive-fig}.  
The $k$th braid $\what{\sigma}_{m,n}^{(k)}$ 
is constructed from the $(k{-}1)$st braid by conjugating by $\tau_{2k+1,m+k+1}$ for $k=1,\dots,m-1$. 
The resulting braid $\what{\sigma}_{m,n}^{(m-1)}$ takes one of three forms: 
\fullref{interbraid-fig}(a) shows the general case when $n \ge m+2$, 
\fullref{interbraid-fig}(b) shows the case when $n=m+1$, and 
\fullref{interbraid-fig}(c) shows the case when $n=m$. 

\begin{figure}[htbp]
\labellist\tiny
\pinlabel {$2k{-}1$} [b] at 52 187
\pinlabel {$m{-}k$} [b] at 90 187
\pinlabel {$n{-}k{-}1$} [b] at 139 187
\pinlabel {$p_{\infty}$} [b] at 172 182
\pinlabel {$2k{-}1$} [b] at 328 195
\pinlabel {$m{-}k{-}1$} [b] at 382 195
\pinlabel {$n{-}k{-}1$} [b] at 432 195
\pinlabel {$p_{\infty}$} [b] at 474 185
\small
\pinlabel {(a)} [b] at 100 0
\pinlabel {(b)} [b] at 400 0
\endlabellist
\begin{center}
\includegraphics[width=4.75in]{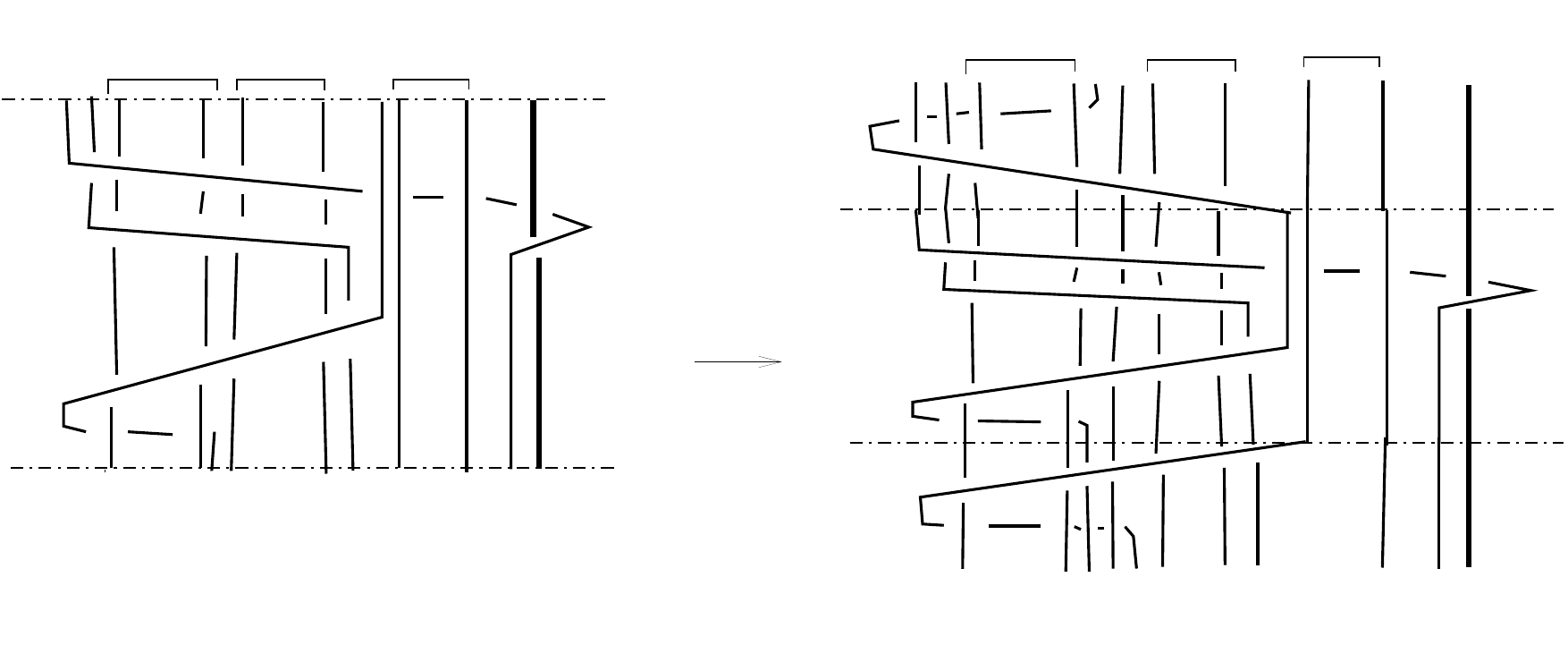}
\caption{Induction step}
\label{inductive-fig}
\end{center}
\end{figure}
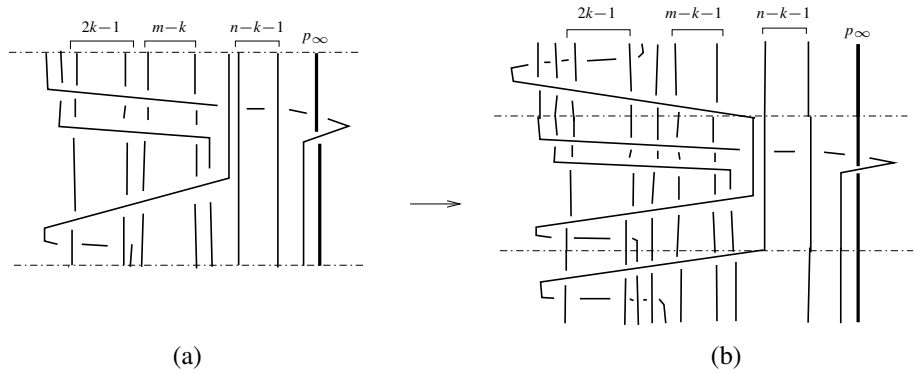

\begin{figure}[htbp]
\labellist\tiny
\pinlabel {$2m{-}1$} [b] at 32 160
\pinlabel {$n{-}m{-}1$} [b] at 97 160
\pinlabel {$p_{\infty}$} [b] at 137 152
\pinlabel {$2m{-}1$} [b] at 252 160
\pinlabel {$p_{\infty}$} [b] at 314 155
\pinlabel {$2m{-}1$} [b] at 428 160
\pinlabel {$p_{\infty}$} [b] at 471 155
\small
\pinlabel {(a)} [b] at 75 0
\pinlabel {(b)} [b] at 275 0
\pinlabel {(c)} [b] at 450 0
\endlabellist
\begin{center}
\includegraphics[height=1.4in]{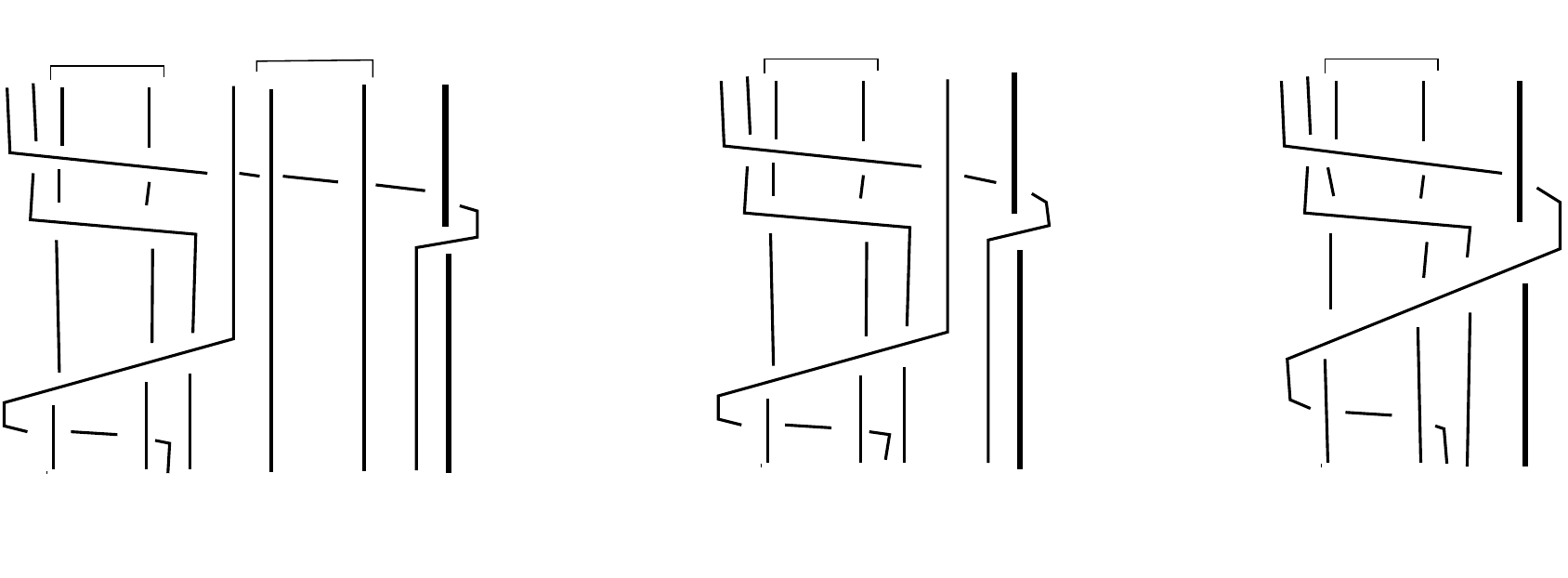}
\caption{After $(m-1)$ inductive steps: (a) $n \ge m+2$ (b) $n = m+1$ (c) $n=m$}
\label{interbraid-fig}
\end{center}
\end{figure}

\begin{figure}[htbp]
\labellist\tiny
\pinlabel {$2m$} [b] at 36 159
\pinlabel {$p_{\infty}$} [b] at 87 155
\pinlabel {$p_{\infty}$} [b] at 277 155
\pinlabel {$2m{-}1$} [b] at 390 162
\pinlabel {$p_{\infty}$} [b] at 443 153
\pinlabel {$2m{-}1$} [b] at 554 158
\pinlabel {$p_{\infty}$} [b] at 608 149
\small
\pinlabel {$=$} at 130 100
\pinlabel {$=$} at 490 100
\pinlabel {(a)} [b] at 130 0
\pinlabel {(b)} [b] at 490 0
\endlabellist
\begin{center}
\includegraphics[height=1.4in]{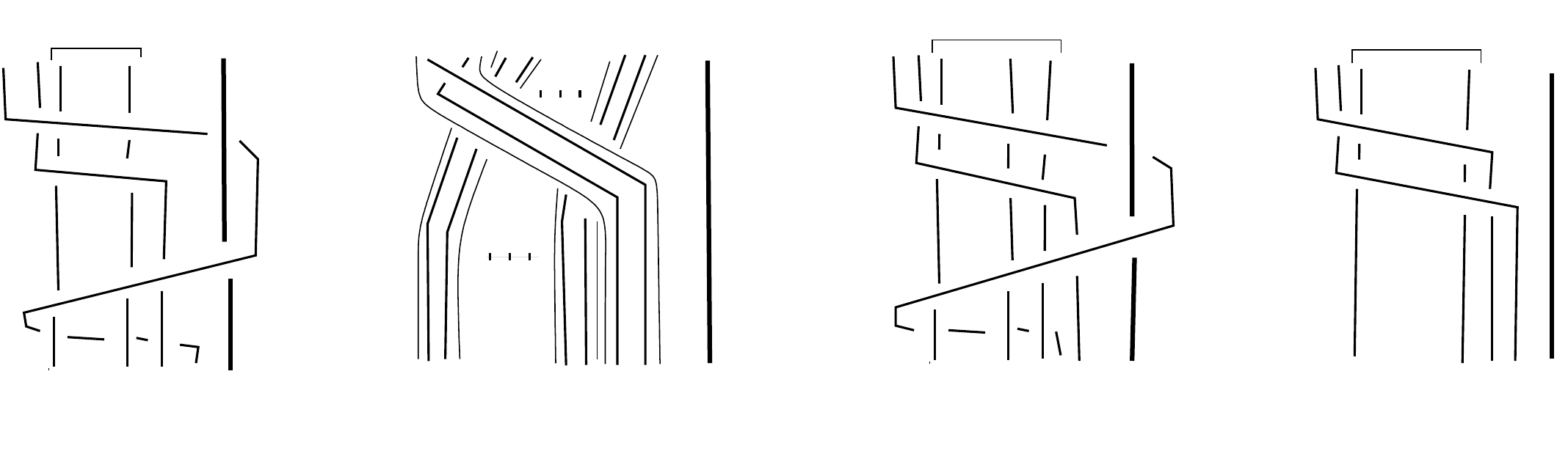}
\caption{(a) reducible and (b) periodic cases}
\label{perred-fig}
\end{center}
\end{figure}

\begin{prop}
\label{reducible-prop}
When $n=m+1$, $\sigma_{m,n}$ is a reducible braid. 
\end{prop}

\begin{proof}  
By applying one more conjugation by $\tau_{2m+1,2m+1}$, we obtain the left-hand braid 
in  \fullref{perred-fig}(a), which equals the right-hand braid.  
One sees that there is a collection of disjoint disks enclosing pairs of 
marked points in $S^2$ whose boundaries are invariant by
$\phi_{\what{\sigma}_{m,n}}$. 
The claim now follows from \fullref{boundary-lem}. 
\end{proof}

\begin{prop}
\label{periodic-prop}
When $n=m$, $\sigma_{m,n}$ is a periodic braid.
\end{prop}

\begin{proof}  
\fullref{perred-fig}(b) shows an equivalence of spherical braids.  It is not hard to see 
that the right-hand braid is periodic in $\sB(S^2;\sS,\{p_\infty\})$.  
The rest follows from \fullref{boundary-lem}. 
\end{proof}

The general case when $n \ge m+2$ is shown in \fullref{pA-fig}.  
The transition from \fullref{pA-fig}(a) to \ref{pA-fig}(b) is given by doing successive conjugations by 
$\tau_{2m+k,2m+k}$ for $k=1,\dots,n-m$.  
The braid in \ref{pA-fig}(b) equals the braid in \ref{pA-fig}(c). 

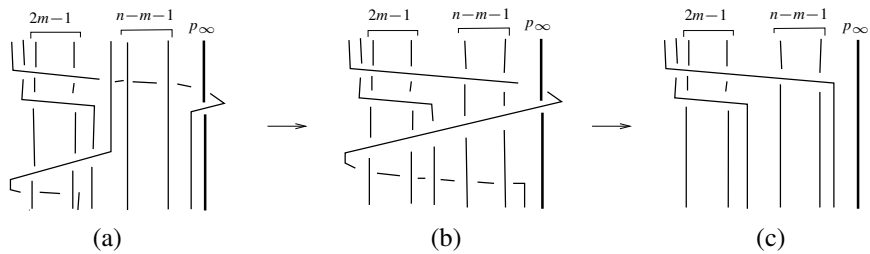
\begin{figure}[htbp]
\labellist\tiny
\pinlabel {$2m{-}1$} [b] at 30 160
\pinlabel {$n{-}m{-}1$} [b] at 97 162
\pinlabel {$p_{\infty}$} [b] at  138 154
\pinlabel {$2m{-}1$} [b] at 270 162
\pinlabel {$n{-}m{-}1$} [b] at 335 163
\pinlabel {$p_{\infty}$} [b] at  376 156
\pinlabel {$2m{-}1$} [b] at 494 162
\pinlabel {$n{-}m{-}1$} [b] at 560 163
\pinlabel {$p_{\infty}$} [b] at  601 154
\small
\pinlabel {(a)} [b] at 70 0
\pinlabel {(b)} [b] at 310 0
\pinlabel {(c)} [b] at 540 0
\endlabellist
\begin{center}
\includegraphics[width=4.5in]{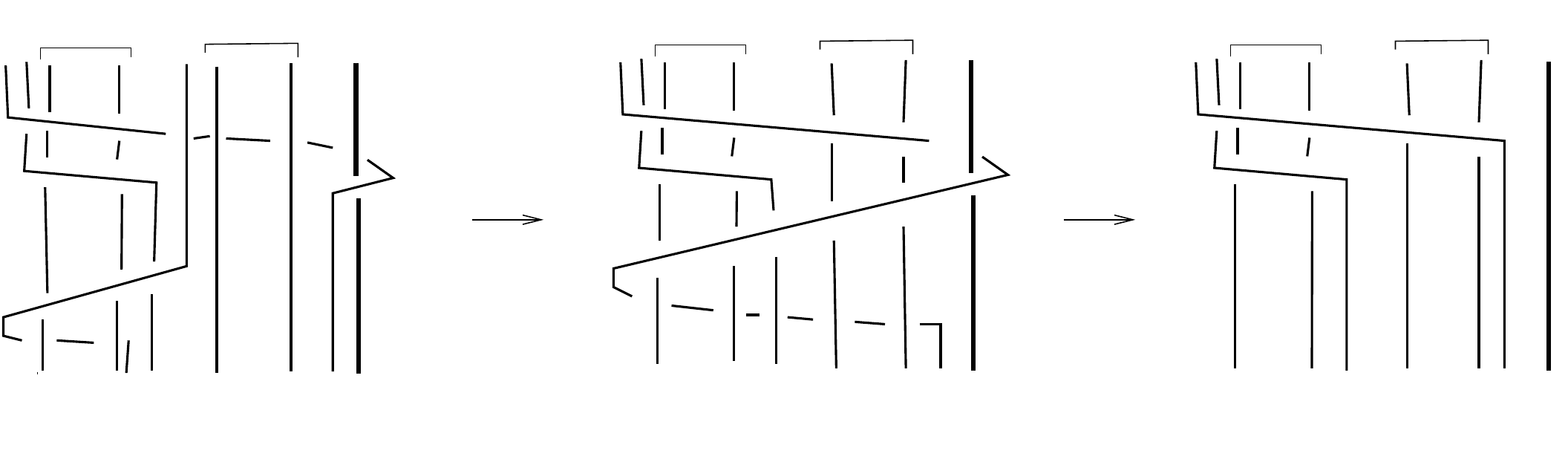}
\caption{General case}
\label{pA-fig}
\end{center}
\end{figure}

Let  $H_{m,n}$ be the graph with vertices $1, \ldots, m+n+1$  and $p$ in 
\fullref{SmnGraph-fig}(left), and 
we consider the graph map $\h_{m,n} \colon\thinspace  H_{m,n} \rightarrow H_{m,n}$ drawn 
in \fullref{SmnGraph-fig}. 
The unusual numbering of vertices comes from the left-to-right ordering 
of the strands (excluding $p_\infty$) of $\what{\sigma}_{m,n}$ 
shown in \fullref{spherical-fig}(b).  
This ordering proves useful for comparing the transition matrices of $\phi_{\beta_{m,n}}$ and 
$\phi_{\sigma_{m,n}}$ in \fullref{Salem-Boyd-section}. 

Let $\sigma'_{m,n} \in \sB(D; \sS)$ be the braid given by the preimage of the braid 
in \fullref{pA-fig}(c) under the contraction map of \fullref{boundary-lem}.  
(Hence $\sigma_{m,n}$ is obtained from the braid in \fullref{pA-fig}(c) 
by removing the strand $p_{\infty}$.) 

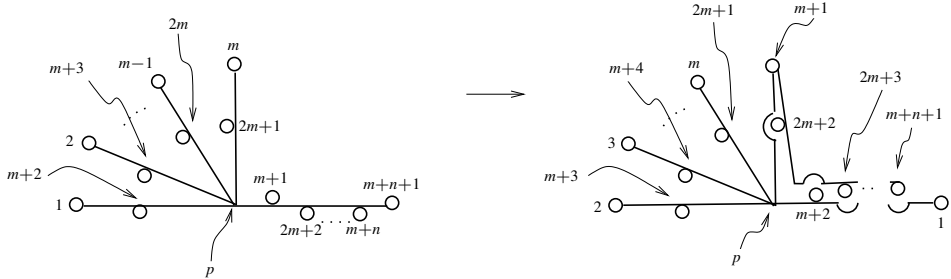
\begin{figure}[htbp]
\labellist\tiny
\hair=2pt
\pinlabel {$1$} [r] at 34 44
\pinlabel {$2$} [r] at 40 80
\pinlabel {$m{+}2$} [r] at 22 61
\pinlabel {$m{+}3$} [br] at 45 115
\pinlabel {$m{-}1$} [br] at 84 118
\pinlabel {$2m$} [b] at 96 140
\pinlabel {$m$} [b] at 127 128
\pinlabel {$2m{+}1$} [l] at 127 88
\pinlabel {$m{+}1$} [b] at 148 52
\pinlabel {$2m{+}2$} [t] at 165 36
\pinlabel {$m{+}n$} [t] at 200 36
\pinlabel {$m{+}n{+}1$} [b] at 216 49
\pinlabel {$p$} [t] at 114 12

\pinlabel {$2$} [r] at 337 43
\pinlabel {$3$} [r] at 345 79
\pinlabel {$m{+}3$} [r] at 324 60
\pinlabel {$m{+}4$} [b] at 348 114
\pinlabel {$m$} [b] at 386 118
\pinlabel {$2m{+}1$} [b] at 397 147
\pinlabel {$m{+}1$} [b] at 452 148
\pinlabel {$2m{+}2$} [l] at 439 90
\pinlabel {$2m{+}3$} [b] at 491 109
\pinlabel {$p$} [t] at 410 19
\pinlabel {$m{+}2$} [t] at 453 45
\pinlabel {$m{+}n{+}1$} [b] at 510 88
\pinlabel {$1$} [t] at 525 40
\endlabellist
\begin{center}
\includegraphics[width=5in]{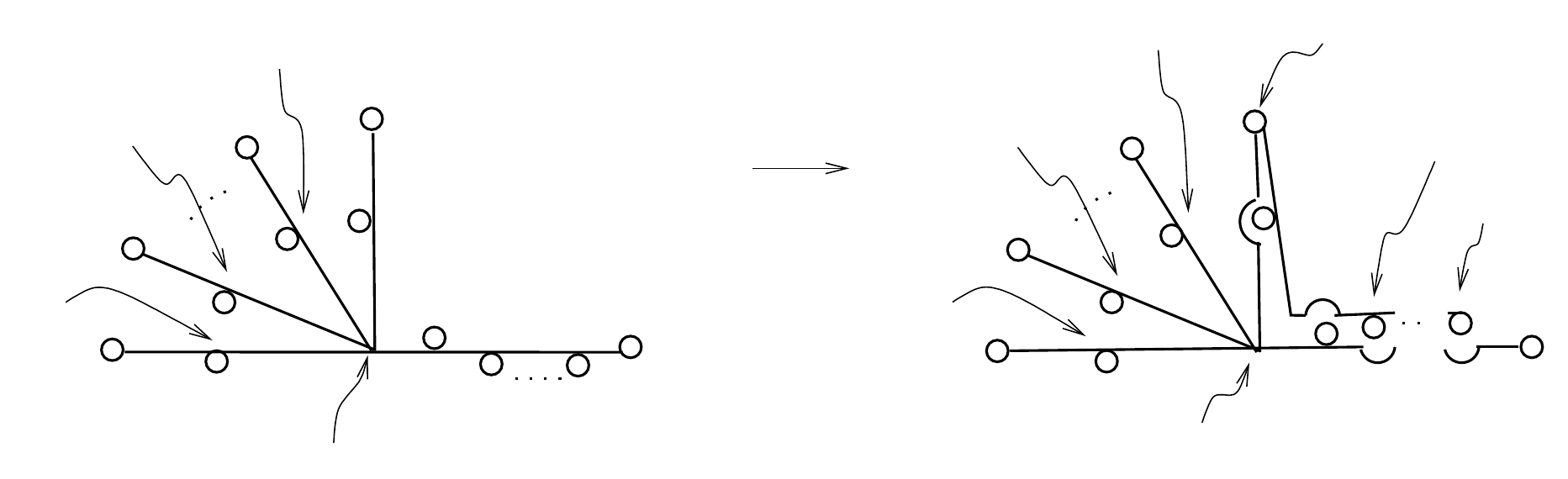}
\caption{Graph map $\h_{m,n}$ for $\phi_{\sigma_{m,n}}$}
\label{SmnGraph-fig}
\end{center}
\end{figure}

 \begin{prop}
 \label{SmnGraph-prop} 
 For $n \ge m+2$, the graph map $\h_{m,n} \colon\thinspace H_{m,n} \rightarrow H_{m,n}$ 
 is the induced graph map for $\phi_{\sigma'_{m,n}}$ satisfying {\rm(BH1)} and {\rm(BH2)}. 
 \end{prop}
 
\begin{proof} 
One can see that the fibered surface $\mathcal{F}(H_{m,n})$ carries a homeomorphism of 
$\phi_{\sigma'_{m,n}}$, and hence 
$\h_{m,n} \colon\thinspace H_{m,n} \rightarrow H_{m,n}$ is the induced graph map for $\phi_{\sigma'_{m,n}}$. 
The proof that $\h_{m,n}$ satisfies {\rm(BH1)} and {\rm(BH2)} 
is similar to that of \fullref{BmnGraph-prop}.
\end{proof}

\begin{proof}[Proof of \fullref{Smn-classification-thm}]
By \fullref{Smn-sym-prop}, it suffices to classify the braids $\sigma_{m,n}$ 
with $n \ge m \ge 1$.  
By \fullref{reducible-prop}, $\sigma_{m,n}$ is reducible if $n = m+1$, and 
by \fullref{periodic-prop} 
$\sigma_{m,n}$ is periodic if $n=m$.  
In all other cases, \fullref{SmnGraph-prop} shows 
that $\sigma_{m,n}$ is pseudo-Anosov since $\sigma_{m,n}$ is conjugate to $\sigma'_{m,n}$, 
and \fullref{Smn=Snm-lem} implies that 
$\lambda(\sigma_{m,n}) = \lambda(\sigma_{n,m})$. 
\end{proof}

\subsection{Train tracks} 
\label{TT-section}

By using the graph smoothing in \fullref{Bestvina-Handel-section}, 
the train track $\tau_{\g_{m,n}}$ for $\phi_{\beta_{m,n}}$ and the train track $\tau_{\h_{m,n}}$ 
for $\phi_{\sigma'_{m,n}}$ are  given in Figures~\ref{Bmn-TT} and \ref{Smn-TT}. 
Applying  \fullref{prongs-lem} to $\tau_{\g_{m,n}}$ and $\tau_{\h_{m,n}}$, 
we immediately see the following. 

\begin{figure}[htbp]
\labellist\small
\pinlabel {$(m{+}1)$--gon} [t] at 195 592
\pinlabel {$(n{+}1)$--gon} [b] at 392 660
\endlabellist
\begin{center}
\includegraphics[width=2.5in]{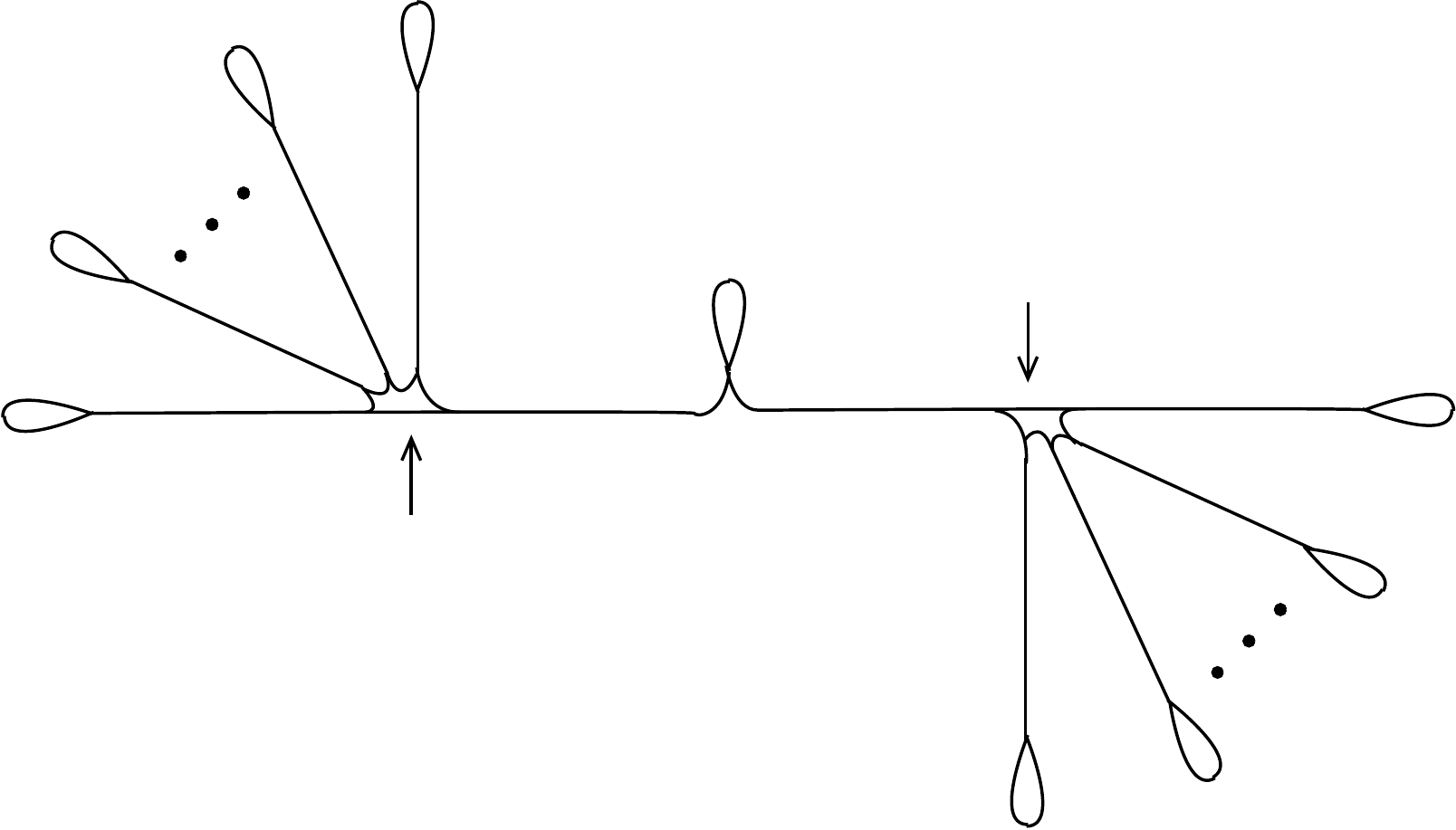}
\caption{Train track $\tau_{\g_{m,n}}$}
\label{Bmn-TT}
\end{center}
\end{figure}

\begin{figure}[htbp]
\labellist\small
\pinlabel {$(m{+}1)$--gon} [t] at 210 617
\pinlabel {$(n{-}m{-}2)$} [t] at 406 614
\endlabellist
\begin{center}
\includegraphics[height=1.3in]{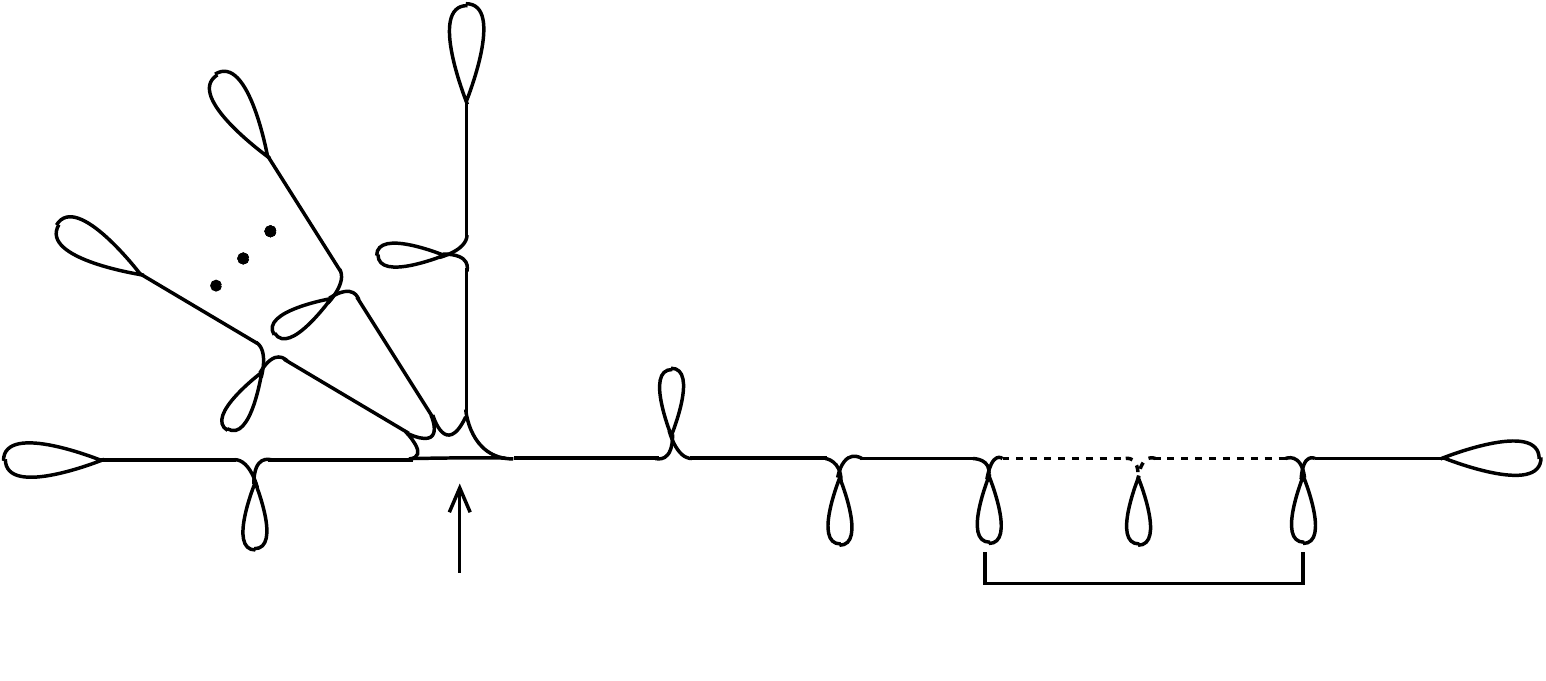}
\caption{Train track $\tau_{\h_{m,n}}$} 
\label{Smn-TT}
\end{center}
\end{figure}

\begin{lem}
\label{Bmn-prongs-lem} 
The invariant foliations associated to  the pseudo-Anosov representative 
$\smash{\Phi_{\what{\beta}_{m,n}}}$ of 
 $\smash{\phi_{\what{\beta}_{m,n}}} \in \sM(S^2; \sS,\{p_\infty\})$ have a $1$--pronged singularity 
 at each point of $\sS \cup \{p_{\infty}\}$, 
an $(m+1)$--pronged singularity at a point $p \in S^2 \setminus (\sS \cup \{p_{\infty}\})$, and 
an $(n+1)$--pronged singularly at a point $q \in S^2 \setminus (\sS \cup \{p_{\infty}\})$. 
\end{lem}

\begin{lem}
\label{Smn-prongs-lem}  
For $n \ge m+2$, the invariant 
foliations associated to the pseudo-Anosov representative
$\Phi_{\what{\sigma}_{m,n}}$ of 
$\phi_{\what{\sigma}_{m,n}} \in \sM(S^2; \sS, \{p_{\infty}\})$ have a $1$--pronged singularity 
at each point of $\sS$, an $n$--pronged singularity  at  $p_{\infty}$, and  
an $(m+1)$--pronged singularity  at a point $p \in S^2 \setminus (\sS \cup \{p_{\infty}\})$. 
\end{lem}
By Lemmas \ref{closure2-lem} and \ref{Smn-prongs-lem}, we have the following. 

\begin{cor}
\label{Overline-Smn-lem}  
For $n \ge m+2$, $\wbar{\sigma}_{m,n}$ is pseudo-Anosov, and 
$\lambda(\sigma_{m,n}) = \lambda(\wbar{\sigma}_{m,n})$. 
\end{cor}

A pseudo-Anosov map $\Phi$ is said to be {\it orientable} if 
the stable and unstable foliations associated to $\Phi$ are orientable. 

\begin{prop}
\label{Bmn-lift-prop} 
Let $m+n=2g$. 
If both $m$ and $n$ are odd, 
there is a pseudo-Anosov element of $\sM_g$ whose pseudo-Anosov representative is 
orientable with the same dilatation as $\beta_{m,n}$. 
\end{prop}

\begin{proof}  
Let $\Phi'_{\smash{\what{\beta}_{m,n}}}$ be the lift of
$\Phi_{\smash{\what{\beta}_{m,n}}}$ 
to the double branched covering $F_g$ of $S^2$ branched along 
$\sS \cup \{p_\infty\}$, and denote by $\wwhat{\sS'}$ the preimage of
$\wwhat{\sS}=\sS \cup \{p_\infty\}$ in $F_g$. 
By the proof of \fullref{spectrum-prop}, 
$\Phi'_{\smash{\what{\beta}_{m,n}}}$ is a pseudo-Anosov map with 
$\lambda\bigl(\Phi'_{\smash{\what{\beta}_{m,n}}}\bigr)=
\lambda\bigl(\Phi_{\smash{\what{\beta}_{m,n}}}\bigr)= 
\lambda\bigl(\beta_{m,n}\bigr)$. 
By \fullref{Bmn-prongs-lem}, 
$\Phi_{\smash{\what{\beta}_{m,n}}}'$ has an 
invariant foliation $\mathcal{F}^{\pm}$ with two $(m+1)$--pronged singularities and 
two $(n+1)$--pronged singularities at points of $F_g \setminus \wwhat{\sS}'$, and 
regular points of $\wwhat{\sS}'$. 
Hence all singularities of $\mathcal{F}^{\pm}$ are even--pronged. 

To show that $\mathcal{F}^{\pm}$ is orientable, it suffices to note that 
the natural map from  the fundamental group of $F_g$ to $Z/2Z$ induced by $\mathcal{F}^{\pm}$ 
is trivial. 
Consider the invariant foliation on $S^2$ associated to
$\Phi_{\smash{\what{\beta}_{m,n}}}$ 
with $1$--pronged singularity at each point of $\wwhat{\sS}$ and even--pronged singularity 
elsewhere. 
The punctured sphere $S^2 \setminus \wwhat{\sS}$ has fundamental group 
generated by loops emanating from a basepoint, following a path $\gamma_p$ 
to a point near a marked point $p \in \wwhat{\sS}$, going around a small circle 
centered at $p$, then returning in the reverse direction along $\gamma_p$ back to the basepoint. 
Consider the double unbranched covering of $S^2 \setminus \wwhat{\sS}$. 
Then by construction,  the natural map from the fundamental 
group of the covering surface to $Z/2Z$ defined by the lifted foliation  is trivial.   
The same is true for the fundamental group of the branched covering surface $F_g$, 
and hence the natural map from the fundamental group of $F_g$ to $Z/2Z$ defined by the lifted foliation 
$\mathcal{F}^{\pm}$ is trivial. 
\end{proof}

\begin{prop}
\label{Smn-lift-prop} 
Let $m+n=2g$. 
For each $m,n \ge 1$ with $|m-n| \ge 2$, 
there is a pseudo-Anosov element of $\sM_g$ whose pseudo-Anosov representative is 
orientable with the same dilatation as $\sigma_{m,n}$. 
\end{prop}

\begin{proof}  
By \fullref{Smn=Snm-lem}, we can assume $n \ge m+2$. 
\fullref{Smn-prongs-lem} says that 
the invariant foliations associated to $\smash{\Phi_{\what{\sigma}_{m,n}}}$
have an $n$--pronged singularity and an $(m{+}1)$--pronged singularity. 
Since $m+n=2g$, $(m+1)$ and $n$ have the opposite parity. 

Let $F_g$ be the branched covering of $S^2$ branched along $\sS$
and either an $(m{+}1)$--pronged singularity if $(m{+}1)$ is odd,
or $p_\infty$ if $n$ is odd.  Let $\smash{\Phi_{\what{\sigma}_{m,n}}'}$
be the lift of $\smash{\Phi_{\what{\sigma}_{m,n}}}$  to $F_g$.  Then
$\smash{\Phi_{\what{\sigma}_{m,n}}'}$ is pseudo-Anosov with dilatation equal
that of $\smash{\Phi_{\what{\sigma}_{m,n}}}$.  Furthermore, by our choice of
branch points, the invariant foliations have only even order prongs.
One shows that  they are orientable by using the same arguments as in
the proof of \fullref{Bmn-lift-prop}.
\end{proof}

We conclude this section by relating $\g_{m,n}$ and $\h_{m,n}$
in a way that is compatible with the conjugations used in
\fullref{symmetry-section}.

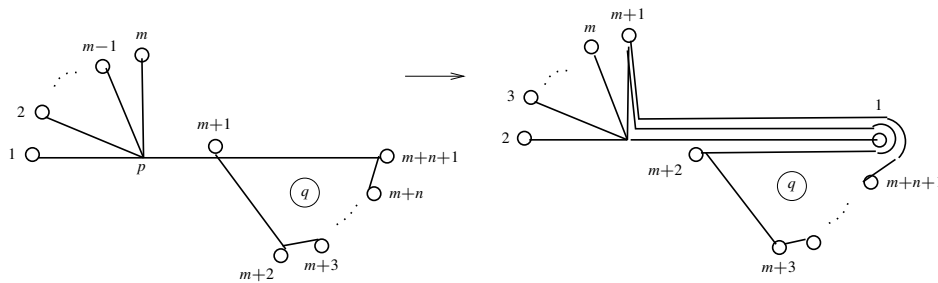
\begin{figure}[htbp]
\labellist\tiny
\pinlabel {$1$} [r] at 16 68
\pinlabel {$2$} [r] at 21 91
\pinlabel {$m{-}1$} [b] at 58 119
\pinlabel {$m$} [b] at 81 127
\pinlabel {$p$} [t] at 81 67
\pinlabel {$m{+}1$} [b] at 121 76
\pinlabel {$m{+}2$} [tr] at 157 8
\pinlabel {$m{+}3$} [t] at 180 14
\pinlabel {$m{+}n$} [l] at 213 46
\pinlabel {$m{+}n{+}1$} [l] at 220 66
\pinlabel {$q$} at 171 46
\pinlabel {$1$} [b] at 487 88
\pinlabel {$2$} [r] at 288 77
\pinlabel {$3$} [r] at 291 100
\pinlabel {$m$} [b] at 328 131
\pinlabel {$m{+}1$} [b] at 349 137
\pinlabel {$m{+}2$} [tr] at 382 65
\pinlabel {$m{+}3$} [t] at 432 13
\pinlabel {$m{+}n{+}1$} [l] at 487 52
\pinlabel {$q$} at 439 50
\endlabellist
\begin{center}
\includegraphics[width=5in]{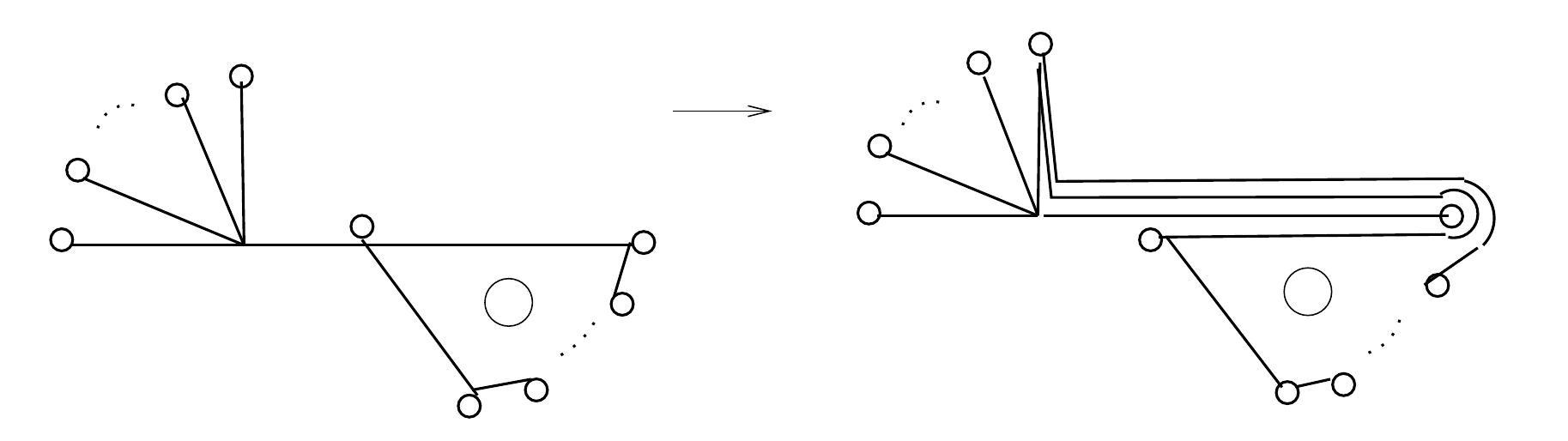}
\caption{Graph map $\g_{m,n}'$ for $\phi_{m,n}$}
\label{BmnGraphNew-fig}
\end{center}
\end{figure}

Since $q$ is a fixed point for $\g_{m,n}$ (see \fullref{BmnGraph-fig}),
$\smash{\Phi_{\what{\beta}_{m,n}}}$ determines a mapping class $\phi_{m,n} =
\bigl[\smash{\Phi_{\what{\beta}_{m,n}}}\bigr]$ in $\sM(S^2;\sS,\{q\}, \{p_\infty\})$.
Let $\g_{m,n}' \colon\thinspace  G_{m,n}' \rightarrow G_{m,n}'$ be the
graph map obtained from $\g_{m,n}$ after  puncturing $D$ at $q$ as in
\fullref{BmnGraphNew-fig}.  Then $\g'_{m,n} \colon\thinspace G'_{m,n}
\rightarrow G'_{m,n}$ is the induced graph (satisfying  {\rm(BH1)}
and {\rm(BH2)}) for the mapping class which is the preimage of
$\phi_{m,n}$ under the map from $\sM(D;\sS,\{q\})$ to $\sM(S^2;\sS,\{q\},
\{p_\infty\})$.  Identify $\g'_{m,n}$ with the graph map on $S^2$ obtained
by pushed forward by the contraction map in \fullref{boundary-lem}.

\begin{figure}[htbp]
\labellist\tiny
\hair=2pt
\pinlabel {$1$} [r] at 28 44
\pinlabel {$m{+}2$} [r] at 22 60
\pinlabel {$2$} [r] at 35 80
\pinlabel {$m{+}3$} [br] at 44 115
\pinlabel {$m{-}1$} [br] at 79 122
\pinlabel {$2m$} [b] at 98 142
\pinlabel {$m$} [b] at 127 136
\pinlabel {$p$} [t] at 115 13
\pinlabel {$m{+}1$} [tl] at 149 44
\pinlabel {$2m{+}1$} [l] at 139 94
\pinlabel {$2m{+}2$} [b] at 161 61
\pinlabel {$m{+}n$} [bl] at 192 62
\pinlabel {$m{+}n{+}1$} [tl] at 225 19
\pinlabel {$\scriptstyle\infty$} at 230 53

\pinlabel {$1$} [tl] at 527 38
\pinlabel {$2$} [r] at 325 46
\pinlabel {$3$} [r] at 334 82
\pinlabel {$m{+}3$} [r] at 321 61
\pinlabel {$m{+}4$} [b] at 343 113
\pinlabel {$m$} [br] at 377 121
\pinlabel {$2m{+}1$} [b] at 391 146
\pinlabel {$m{+}1$} [bl] at 448 149
\pinlabel {$p$} [t] at 407 20
\pinlabel {$2m{+}2$} [bl] at 448 96
\pinlabel {$2m{+}3$} [b] at 489 111
\pinlabel {$m{+}n$} [bl] at 503 96
\pinlabel {$m{+}n{+}1$} [bl] at 516 82
\pinlabel {$\scriptstyle\infty$} at 536 63
\endlabellist
\begin{center}
\includegraphics[width=5in]{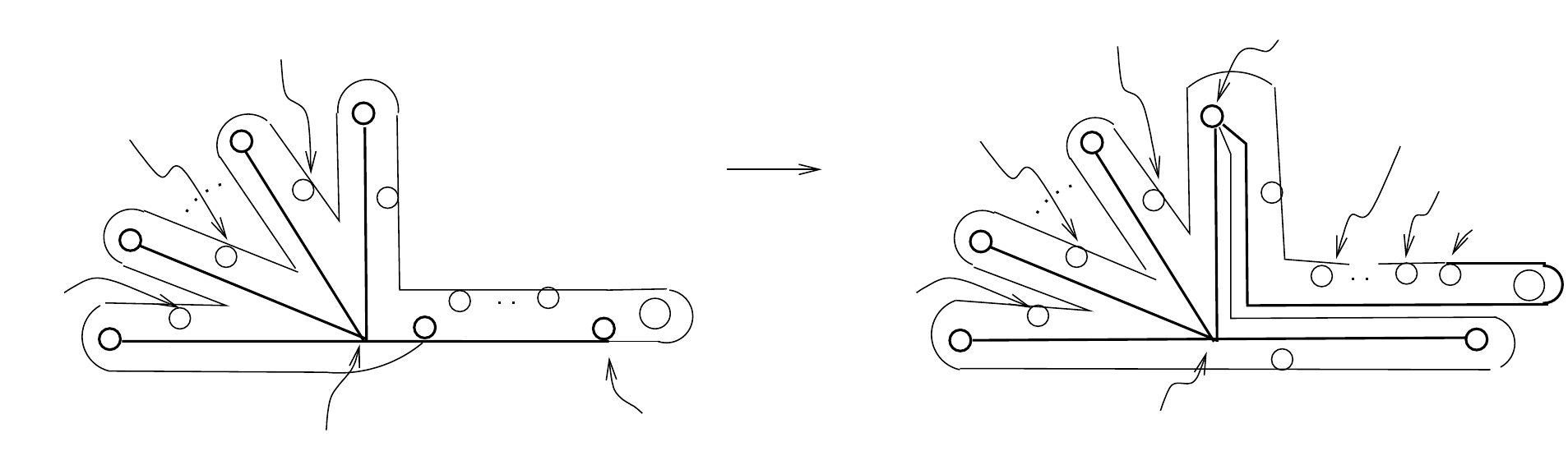}
\caption{Exchanging the roles of $q$ and $p_\infty$ for $\g_{m,n}'$: 
$\infty$ in the figure indicates $p_{\infty}$}
\label{wrap-fig}
\end{center}
\end{figure}
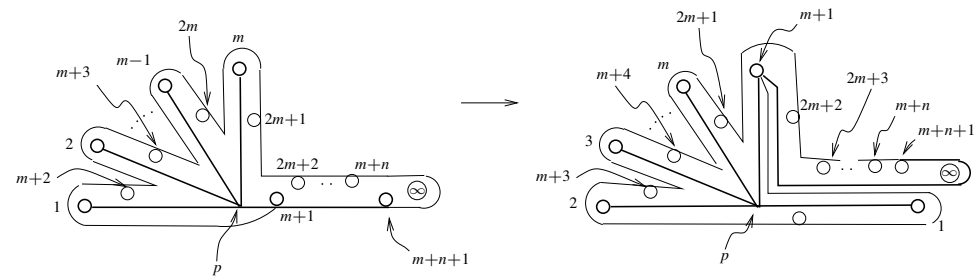

Exchanging the roles of $q$ and $p_\infty$ (i.e., bringing 
$p_\infty$ into the visual plane) yields the  graph map shown in  \fullref{wrap-fig}, 
which is equivalent to $\g_{m,n}'$. 
Now remove  $p_\infty$,  and consider the graph map 
\begin{eqnarray}
\label{graphmap-map}
\f_{m,n} \colon\thinspace  G_{m,n}' \rightarrow H_{m,n} 
\end{eqnarray}
obtained by a natural identification of edges of $G'_{m,n}$ to edges of the graph in 
\fullref{BmnGraphNew-fig}(left) removing $p_\infty$. 
\fullref{identify-fig} shows the natural projection map applied to the image 
of the edges of $G'_{m,n}$ under $\f_{m,n}$. 
 The map $\h_{m,n} \colon\thinspace H_{m,n} \rightarrow H_{m,n}$ 
in \fullref{SmnGraph-fig} is the one induced by pushing forward 
$\g_{m,n}'$  by the map $\f_{m,n}$. 

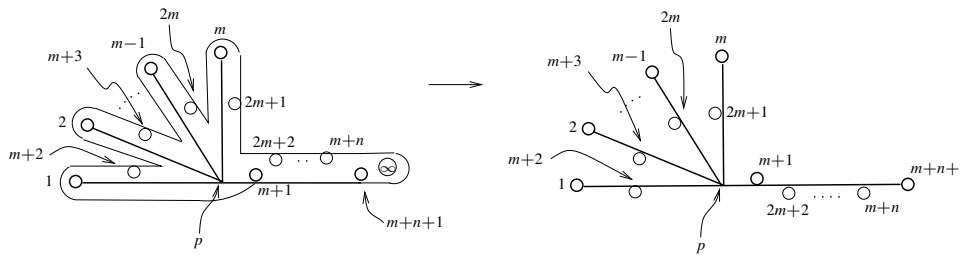
\begin{figure}[htbp]
\labellist\tiny
\hair=2pt
\pinlabel {$1$} [r] at 28 46
\pinlabel {$2$} [r] at 36 82
\pinlabel {$m{+}2$} [r] at 22 62
\pinlabel {$m{+}3$} [br] at 45 116
\pinlabel {$m{-}1$} [br] at 84 124
\pinlabel {$2m$} [b] at 96 142
\pinlabel {$m$} [b] at 127 133
\pinlabel {$2m{+}1$} [l] at 139 94
\pinlabel {$2m{+}2$} [b] at 160 62
\pinlabel {$m{+}n$} [bl] at 192 64
\pinlabel {$m{+}1$} [tl] at 148 46
\pinlabel {$m{+}n{+}1$} [l] at 226 20
\pinlabel {$p$} [t] at 114 14
\pinlabel {$\scriptstyle\infty$} at 230 53

\pinlabel {$1$} [r] at 342 44
\pinlabel {$2$} [r] at 349 79
\pinlabel {$m{+}2$} [r] at 328 59
\pinlabel {$m{+}3$} [br] at 351 113
\pinlabel {$m{-}1$} [br] at 391 116
\pinlabel {$2m$} [b] at 403 140
\pinlabel {$m$} [b] at 434 127
\pinlabel {$2m{+}1$} [l] at 434 88
\pinlabel {$2m{+}2$} [t] at 475 35
\pinlabel {$m{+}n$} [tl] at 520 35
\pinlabel {$m{+}1$} [bl] at 455 52
\pinlabel {$m{+}n{+}1$} [bl] at 548 47
\pinlabel {$p$} [t] at 422 11

\endlabellist
\begin{center}
\includegraphics[width=5in]{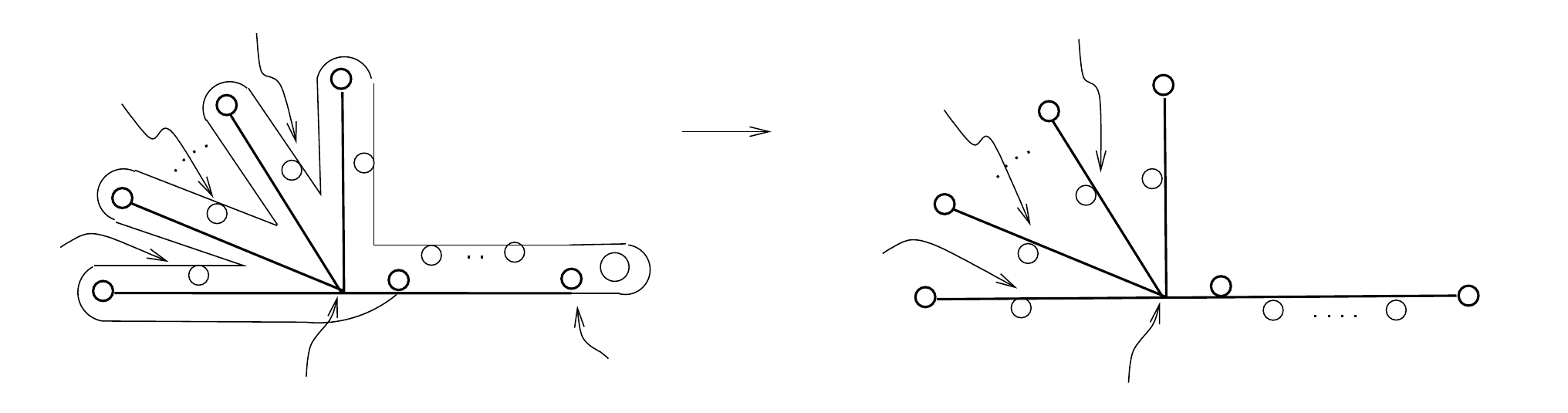}
\caption{Identifying edges of $G_{m,n}'$ with edges of $H_{m,n}$}
\label{identify-fig}
\end{center}
\end{figure}

\subsection{Characteristic equations for dilatation}
\label{dilatations-section}

Consider the graph map $\r_m \colon\thinspace  \Gamma_m \rightarrow \Gamma_m$, 
shown in \fullref{RmnGraph-fig}. 
As seen in Figures~\ref{BmnGraph-fig} and \ref{SmnGraph-fig}, the graph maps for 
 $\phi_{\beta_{m,n}}$ and $\phi_{\sigma_{m,n}}$ ``contain" $\r_m$ as the 
action on a subgraph. 

\begin{figure}[htbp]
\labellist\small
\pinlabel {$1$} [r] at 13 15
\pinlabel {$m{-}1$} [br] at 30 60
\pinlabel {$m$} [b] at 58 66
\pinlabel {$p$} [t] at 58 15
\pinlabel {$m{+}1$} [l] at 104 15
\pinlabel {$2$} [r] at 214 16
\pinlabel {$m$} [br] at 232 61
\pinlabel {$p$} [t] at 259 18
\pinlabel {$m{+}1$} [bl] at 304 24
\pinlabel {$1$} [tl] at 304 16
\endlabellist
\begin{center}
\includegraphics[height=0.9in]{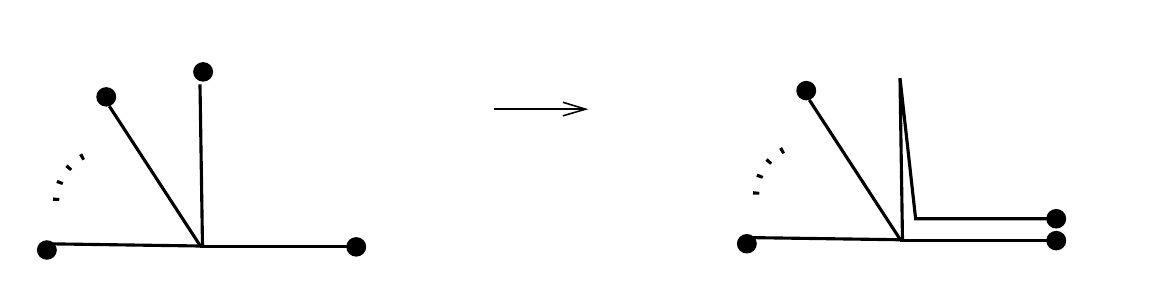}
\caption{Graph map $\r_m \colon\thinspace \Gamma_m \rightarrow \Gamma_m$}
\label{RmnGraph-fig}
\end{center}
\end{figure}

The transition matrix for $\r_m$ has the following form with respect to the basis of 
edges $e(p,1), \dots, e(p,m+1)$: 
$$
\sR_m = \left[
\begin{array}{cccccc}
0 & 1  &  0 & \dots&0 & 0\\
0 & 0 & 1 & \dots &0& 0\\
\vdots\\
0 &0 &0& \dots & 1 & 0\\
0 &0 & 0 & \dots & 0& 2\\
1 &0 &0& \dots & 0 & 1\\
\end{array}
\right ].
$$
The characteristic polynomial for $\sR_m$ is $R_m(t) = t^m(t-1) - 2$. 
As we will see in the proof of \fullref{CharEq-thm}, 
the appearance of $\sR_m$ within the transition matrices of $\phi_{\beta_{m,n}}$ and 
 $\phi_{\sigma_{m,n}}$ gives rise to a similar form for their characteristic equations. 
 
Given a polynomial $f(t)$ of degree $d$, the {\it reciprocal} of $f(t)$ is $f_*(t) = t^df(1/t)$. 

\begin{thm}
\label{CharEq-thm}  
\begin{enumerate}
\item[(1)] For $m,n \ge 1$, 
$\lambda(\beta_{m,n})$ is the largest root of 
$$T_{m,n}(t) = t^{n+1} R_m(t) + (R_m)_*(t).$$
\item[(2)] For $m,n\ge 1$ with $|m-n| \ge 2$, 
$\lambda(\sigma_{m,n})$ is the largest root of 
$$S_{m,n}(t) = t^{n+1} R_m(t) - (R_m)_*(t).$$
\end{enumerate}
\end{thm}

\begin{proof} 
We note that the spectral radius of the transition matrix for $\g_{m,n}$ 
is equal to  that for $\g'_{m,n}$ by the construction of $\g'_{m,n}$ (\fullref{TT-section}). 
Thus, to find characteristic polynomials for $\lambda(\beta_{m,n})$ and $\lambda(\sigma_{m,n})$, 
it is enough to compute the transition matrices for $\g_{m,n}'$ and $\h_{m,n}$ respectively. 

Consider the  basis for $V^{\mathrm{tot}}(G_{m,n}')$:
 \begin{eqnarray*}
v_k &=& e(p,k),\quad k=1,\dots,m, 
\\
v_{m+1} &=&  e(p,m+1)+e(m+1,m+n+1),  
\\
v_{m+1+k} &=& e(m+k,m+k+1), \quad k=1,\dots,n, \ \mbox{and} 
\\
v_{m+n+2} &=& e(p,m+1).
\end{eqnarray*}
The corresponding transition matrix $\sT'_{m,n}$ for $\g_{m,n}'$ is given by: 
$$
\sT'_{m,n} = 
\left [
\begin{array}{cccccc|cccccc}
0 & 1 & 0 & \dots & 0 & 0 &0&0&\dots&&0\\
0 & 0 & 1 & \dots & 0 & 0&0&0& \dots &&0\\
&&&\dots&&&&&\dots&&\\
0 &0 &0&\dots&1&0&0&0&\dots&&0\\
0 & 0 & 0 & \dots & 0 & 2 &1&0& \dots&&(1)_b\\
1 & 0 & 0 & \dots & 0 & 1 &2 & 0 &\dots &&0\\
\hline
0 &0& \dots &&0&0 &0&1&0&\dots &0\\
0 &0& \dots &&0&0&0&0&1&\dots &0\\
&&\dots&&&&&&\dots&\\
0 &0& \dots &&0&0&0&0&\dots &1&0\\
0 &0& \dots &&0&(1)_{ab}&0&0&\dots &0&0\\
0&0&\dots &&0&0&(-1)_b&0&\dots &0&(0)_a\\
\end{array}
\right ]
$$
We will show that the characteristic polynomial for  $\sT'_{m,n}$ is given by 
$$
T_{m,n}(t) = t^{n+1} R_m(t) + (R_m)_*(t).
$$
The upper left  block matrix of $\sT'_{m,n}$ corresponding to the vectors $v_1,\dots,v_{m+1}$ 
is identical to $\mathcal{R}_m$.  Multiplying the characteristic polynomials of the upper 
left and lower right diagonal blocks gives $t^{n+1}R_m$.  
The rest of the characteristic polynomial has two nonzero summands.  
One corresponds to the matrix entries marked $a$, and is given by 
$$
t(-1)^{^{n+1}}
\left |
\begin{array} {ccccc}
-1&0&\cdots&0&0\\
t&-1&\cdots&0&0\\
&&\cdots&&\\
0&0&\cdots&t&-1\\
\end{array}
\right |_{_{(n-1)\times (n-1)}}
\left |
\begin{array}{cccccc}
t&-1&\cdots&0&0\\
0&t&\cdots&0&0\\
&&\cdots&&\\
0&0&\cdots&t&-1\\
-1&0&\cdots&0&-2
\end{array}
\right |_{_{(m+1)\times (m+1)}} 
$$
which yields $-t(2t^m +1)$.  
The other summand corresponds to the matrix entries marked $b$ and is given by 
$$
(-1)^{^{n+1}}
\left |
\begin{array} {ccccc}
-1&0&\cdots&0&0\\
t&-1&\cdots&0&0\\
&&\cdots&&\\
0&0&\cdots&t&-1\\
\end{array}
\right |_{_{(n-2)\times (n-2)}}
\left |
\begin{array}{cccccc}
t&-1&0&\cdots&0&0\\
0&t&-1&\cdots&0&0\\
&&&\cdots&&\\
0&0&0&\cdots&t&-1\\
-1&0&0&\cdots&0&0\\
\end{array}
\right |_{_{m \times m}}
$$
which yields $1$.
This completes the proof of (1). 

Let $\sS_{m,n}$ be the transition matrix for $\h_{m,n} \colon\thinspace H_{m,n} \rightarrow H_{m,n}$.  
We will pull back $\sS_{m,n}$ to an invertible linear transformation 
on $V^{\mathrm{tot}}(G_{m,n}')$ using $\f_{m,n}$ given in \eqref{graphmap-map}.   
Let $\h_{m,n}' (v_i)$ be the image of $\g_{m,n}'(v_i)$ under $\f_{m,n}$.  
Then the transition matrix $\sS_{m,n}'$ for  $\h_{m,n}' \colon\thinspace G_{m,n}' \rightarrow H_{m,n}$  
is the same as $\sT_{m,n}'$ except at the vector $v_{m+n+1}$.  
As can be seen in \fullref{wrap-fig}, we have
$$
\sS_{m,n}'(v_{m+1}) = \sT_{m,n}'(v_{m+1}) - 2v_{m+n+1}
$$ 
Thus, $\sS_{m,n}'$ differs from $\sT_{m,n}'$ only by changing the entry labeled by both 
$a$ and $b$ from $1$ to $-1$. 

Recall that  the sign of the entry marked both $a$ and $b$ in $\sT_{m,n}'$ determines the sign of 
in front of $(R_{m})_*$.  Since this sign is the only difference between $\sS_{m,n}'$ 
and $\sT_{m,n}'$, the characteristic polynomial for $\sS'_{m,n}$ is given by 
$$
S_{m,n}(t) = t^{n+1}R_m(t) - (R_m)_*(t).
$$
To finish the proof of (2),  we have left to check that $\lambda(\sigma_{m,n})$ is 
the largest root of $S_{m,n}$. 
Thus (2) follows if we can show that the extra eigenvalue of $\sS_{m,n}'$ has 
absolute value $1$.  
From \fullref{identify-fig}, we see that the kernel of the linear map induced by 
$\f_{m,n}$ is spanned by 
$$
w = 2(v_1 + \cdots + v_m) + v_{m+1} - (v_{m+2} + \cdots + v_{m+n+1})  + v_{m+n+2}.
$$
Under $\h_{m,n}'$, we have 
\begin{eqnarray*}
2(v_1+ \cdots v_m) &\mapsto& 2(v_{m+1} + v_1 + \cdots + v_{m-1}),\\
v_{m+1} &\mapsto& 2v_m + v_{m+1} - v_{m+n+1},\\
v_{m+2} + \cdots +v_{m+n+1} &\mapsto& v_{m} + 2 v_{m+1}  + 
v_{m+2} + \cdots + v_{m+n} - v_{m+n+2},\\
v_{m+n+2} &\mapsto& v_{m},
\end{eqnarray*}
and hence, $\h_{m,n}'(w) =  w$.   
Thus, the characteristic polynomial for $\sS_{m,n}'$ 
differs from that for $\sS_{m,n}$ by a factor of $(t-1)$. 
\end{proof}

\begin{rem} 
Minakawa independently discovered the 
pseudo-Anosov maps on $F_g$ constructed in the proof of \fullref{Smn-lift-prop} for the 
case when  $(m,n) = (g-1,g+1)$ using a beautiful new method for constructing 
orientable pseudo-Anosov maps on $F_g$ (see Minakawa \cite{Min}).  
He also directly computes their dilatation using different techniques from ours. 
\end{rem}

\subsection{Dilatations and Salem--Boyd sequences}
\label{Salem-Boyd-section} 

Recall that given a polynomial $f(t)$ of degree $d$, the reciprocal of $f(t)$ is $f_*(t) = t^df(1/t)$.   
The polynomial $f$ satisfying $f=f_*$ (respectively,  $f=-f_*$) 
 is a {\it reciprocal polynomial} (respectively, {\it anti-reciprocal polynomial}). 
For a monic integer polynomial $P(t)$ of degree $d$, 
the sequence 
$$
Q^\pm_n(t) = t^n P(t) \pm P_*(t)
$$
is called the  {\it Salem--Boyd sequence} associated to $P$.   

\begin{thm}
\label{Salem-Boyd-theorem}
Let $Q_n$ be a Salem--Boyd sequence associated to $P$. 
Then $Q_n$ is a reciprocal or an anti-reciprocal polynomial, and the set of roots of 
$Q_n$ outside the unit circle converge to those of $P$ as $n$ goes to infinity.  
\end{thm} 
\fullref{Salem-Boyd-theorem} is a consequence of Rouch\'e's Theorem applied to the sum
$ \frac{P(t)}{t^d} \pm \frac{P_*(t)}{t^{n+d}}$ 
considered as a holomorphic function on the Riemann sphere minus the unit disk.

For a  monic integer polynomial $f(t)$, let $N(f)$ be the number of roots of $f$ outside 
the unit circle, $\lambda(f)$ the maximum norm of roots of $f$, 
and $M(f)$ the product of the norms of roots outside the unit circle, 
which is called the {\it  Mahler measure} of $f$. 
By \fullref{Salem-Boyd-theorem}, we have the following. 

\begin{cor}
\label{Salem-Boyd-cor}  
Let $Q_n$ be a Salem--Boyd sequence associated to $P$.  
Then 
$$ \lim_{n\rightarrow \infty} M(Q_n) = M(P)\ \mbox{and}\ 
 \lim_{n\rightarrow \infty} \lambda(Q_n) = \lambda(P).$$
\end{cor}

Any algebraic integer on the unit circle has a (anti-)reciprocal minimal polynomial.  
Suppose that $P(t) = P_0(t) R(t)$, 
where $R$ is a (anti-)reciprocal and $P_0$ has no roots on the unit circle. 
Then the Salem--Boyd sequence associated to $P$ satisfies 
$$
Q_n(t) = R(t) (t^nP_0(t) \pm (P_0)_*(t)).
$$
We have thus shown the following. 

\begin{lem}
All roots of $P$ on the unit circle are also roots of $Q_n$ for all $n$. 
\end{lem}
The following theorem can be proved by first restricting to the case when $P$ 
has no roots on the unit circle, and then by defining a natural deformation of 
the roots of $P(t)$ to those of $Q_n(t)$, which don't cross the unit
circle (see Boyd \cite{Boyd77}).  

\begin{thm}
\label{number-of-roots-theorem}
Let $Q_n$ be a Salem--Boyd sequence associated to $P$.  
Then $N(Q_n) \leq N(P)$ for all $n$. 
\end{thm}

We now apply the above results to the Salem--Boyd sequences $S_{m,n}$ and 
$T_{m,n}$ associated to $R_m$ of \fullref{CharEq-thm}.  
To do this,  we first study $R_m$.

\begin{lem}
\label{MRm-lem}  
For all $m \ge 1$, $M(R_m) = 2$.
\end{lem}

\begin{proof}  
For $|t| < 1$, we have $|t^m(t-1)| < 2$, 
and hence $R_m$ has no roots strictly within the unit circle. 
Therefore, the Mahler measure of $R_m$ must equal the absolute value of the constant 
coefficient, namely $2$. 
\end{proof}

Applying \fullref{Salem-Boyd-cor}, we have the following. 

\begin{cor}
\label{asymp-cor} 
Fixing $m \ge 1$ and letting $n$ increase, 
the Mahler measures of $T_{m,n}$ and $S_{m,n}$ converge to $2$.
\end{cor}

\begin{lem}
\label{Rm-lem}  
The polynomial $R_m$ has one real root outside the unit circle.  
This root  is simple and greater than $1$.
\end{lem}

\begin{proof} 
Taking the derivative $R_m'(t) = (m+1)t^m - mt^{m-1}$,  
we see that $R_m$ is increasing for $t > \frac{m}{m+1}$, and hence also for $t \ge 1$. 
Since $R_m(1) = -2 < 0$ and $R_m(2) > 0$, 
it follows that $R_m$ has a simple root $\mu_m$ with $1 < \mu_m < 2$.  
Similarly, we can show that for $t < 0$, $R_m$ has no roots for 
$m$ even, and one root if $m$ is odd.  In the odd case, $R_m(-1) = 0$, 
so $R_m$ has no real roots strictly less than $-1$. 
\end{proof}

\begin{lem}
\label{Rm2-lem}  
The sequence $\lambda(R_m)$ converges monotonically to $1$ from above.
\end{lem}

\begin{proof}
 Since $M(R_m) = 2$, we know that $\mu_m = \lambda(R_m) > 1$.  
Take any $\epsilon > 0$. 
Let $D_\epsilon$ be the disk of radius $1+\epsilon$ around the origin in the complex plane. 
Let $g(t) = \frac{t-1}{t}$ and $h_m(t) = \frac{-2}{t^{m+1}}$.  
Then for large enough $m$, we have 
$$
|g(t)| = \left | \frac{t-1}{t} \right | > \left | \frac{2}{t^{m+1}} \right | = |h_m(t)|
$$
for all $t$ on the boundary of $D_\epsilon$, 
and $g(t)$ and $h_m(t)$ are holomorphic on the complement of $D_\epsilon$ in the Riemann sphere.  
By Rouch\'e's theorem, $g(t)$, $g(t) + h_m(t)$, and hence $R_m(t)$ 
have the same number of roots outside $D_\epsilon$, which is zero. 

To show the monotonicity consider $R_m(\mu_{m+1})$.  
Note that $(\mu_{m+1})^{m+1}(t-1) -2 = 0$. 
Hence we have 
\begin{eqnarray*}
R_m(\mu_{m+1}) &=&( \mu_{m+1})^m (t-1) -2\\
&=& ((\mu_{m+1})^m - (\mu_{m+1})^{m+1}) (t-1)\\
&<& 0.
\end{eqnarray*}
Since $R_m(t)$ is an increasing function for $t> 1$, 
we conclude that $\mu_{m+1}< \mu_m$.
\end{proof}

\fullref{Salem-Boyd-cor} and \fullref{Rm2-lem} imply the following. 

\begin{cor}
Fixing $m \ge 1$, the sequences $\lambda(\beta_{m,n})$ and $\lambda(\sigma_{m,n})$
converge to $\lambda(R_m)$ as sequences in $n$.  
Furthermore, we can make 
$\lambda(\beta_{m,n})$ and $\lambda(\sigma_{m,n})$ arbitrarily close to $1$  by taking 
$m$ and $n$ large enough.  
\end{cor}

We now determine the monotonicity of $\lambda(\beta_{m,n})$ and 
$\lambda(\sigma_{m,n})$ for fixed $m \ge 1$. 

\begin{prop}
\label{Perron-prop}
Fixing $m \ge 1$,  
the dilatations $\lambda(\beta_{m,n})$ are strictly monotone decreasing, 
and the dilatations $\lambda(\sigma_{m,n})$ are strictly monotone increasing for $n \ge m+2$. 
\end{prop}

\begin{proof}  
Consider $f(t) = (R_m)_*(t) = -2t^{m+1} - t + 1$. 
Then, for $t > 0$, 
$$f'(t) = -2(m+1)t^m - 1 < 0.$$
Also $f(1) = -2 < 0$. 
Since
$b_{m,n} = \lambda(\beta_{m,n}) > 1$, and for $n \ge m+2$, $s_{m,n} = \lambda(\sigma_{m,n}) > 1$, 
it follows that $(R_m)_*(b_{m,n})$ and $(R_m)_*(s_{m,n})$ are both negative. 
We have
\begin{eqnarray*}
0 &=& T_{m,n}(b_{m,n}) = (b_{m,n})^{n+1} R_m(b_{m,n}) + (R_m)_*(b_{m,n}),\hspace{2mm}
\mbox{and}
\\
0 &=& S_{m,n}(s_{m,n}) = (s_{m,n})^{n+1} R_m(s_{m,n}) - (R_m)_*(s_{m,n}),
\end{eqnarray*}
which imply that $R_m(b_{m,n})  >  0$ and $R_m(s_{m,n}) < 0$. 
Since $R_m$ is increasing for $t > 1$, we have 
\begin{equation}
\label{s-m-b-equation}
s_{m,n} < \mu_m < b_{m,n}. 
\end{equation}
Plug $b_{m,n}$ into $T_{m,n-1}$, and subtract $T_{m,n}(b_{m,n}) = 0$: 
\begin{eqnarray*}
T_{m,n-1}(b_{m,n}) &=&(b_{m,n})^{n-1}R_m(b_{m,n}) + (R_m)_*(b_{m,n})\\
&=&((b_{m,n})^{n-1} - (b_{m,n})^n)R_m(b_{m,n})\\
&<& 0
\end{eqnarray*}
Since $b_{m,n-1}$ is the largest real root of $T_{m,n-1}$, we have $b_{m,n} < b_{m,n-1}$. 

We can show that $s_{m,n} < s_{m,n+1}$ for $n \ge m+2$ in a similar way, 
by adding the formula for $S_{m,n}(s_{m,n})$ to $S_{m,n+1}(s_{m,n})$. 
\end{proof}

The inequalities \eqref{s-m-b-equation} give the following. 

\begin{cor} 
\label{inequality-cor} 
For all $m,n \ge 1$ with $|m-n| \ge 2$, $\lambda(\beta_{m,n})> \lambda(\sigma_{m,n})$. 
\end{cor}

We now fix $2g=m+n$ ($g \ge 2$), and show that among the braids $\beta_{m,n}$ and 
$\sigma_{m,n}$, $\sigma_{g-1,g+1}$ has the least dilatation. 

\begin{prop}
\label{min-prop}  
\begin{enumerate}
 \item[(1)] 
 For $k=1,\dots,m-1$, 
 \begin{eqnarray*}
 \lambda(\beta_{m,m}) &<& \lambda(\beta_{m-k,m+k}), \hspace{2mm}\mbox{and}
\\
\lambda(\beta_{m,m+1}) &<& \lambda(\beta_{m-k,m+k+1}). 
 \end{eqnarray*}
 \item[(2)]
 For $k=2,\dots,m-1$, 
\begin{eqnarray*}
\lambda(\sigma_{m-1,m+1}) &<& \lambda(\sigma_{m-k,m+k}), \hspace{2mm}\mbox{and}
\\
\lambda(\sigma_{m-1,m+2}) &<& \lambda(\sigma_{m-k,m+k+1}). 
\end{eqnarray*}
\end{enumerate}
\end{prop}

\begin{proof}  
Let $\lambda= \lambda(\beta_{m,m})$.  
Then plugging $\lambda$ into $T_{m-k,m+k}$ gives 
\begin{eqnarray*}
T_{m-k,m+k}(\lambda) &=& 
\lambda^{m+k+1}(\lambda^{m-k} (\lambda - 1) - 2) - 2 \lambda^{m-k+1} - \lambda + 1\\
&=& \lambda^{2m+2} - \lambda^{2m+1} - 2\lambda^{m+k+1} - 2\lambda^{m-k+1} - \lambda +1. 
\end{eqnarray*}
Subtracting 
$$
0 = T_{m,m}(\lambda) = \lambda^{2m+2} - \lambda^{2m+1} - 4\lambda^{m+1} - \lambda + 1, 
$$
we obtain
\begin{eqnarray*}
T_{m-k,m+k}(\lambda) = 4 \lambda^{m+1} - 2\lambda^{m+k+1} - 2\lambda^{m-k+1}
= -2\lambda^{m-k+1}(\lambda^k-1)^2
< 0.
\end{eqnarray*}
Since $\lambda(\beta_{m-k,m+k})$ is the largest real root of $T_{m-k,m+k}$, we have 
$\lambda(\beta_{m,m})<\lambda(\beta_{m-k,m+k})$. 

The other inequalities are proved similarly. 
\end{proof}

\begin{prop}
\label{min-Smn-Bmm-prop}
For $m \ge 2$, 
\begin{eqnarray*}
\lambda(\beta_{m,m}) &>& \lambda(\sigma_{m-1,m+1}), \hspace{2mm}{and}
\\
\lambda(\beta_{m,m+1}) &\ge& \lambda(\sigma_{m-1,m+2})
\end{eqnarray*}
with equality if and only if $m=2$.
\end{prop}

\begin{proof} 
Let $\lambda=\lambda(\sigma_{m-1,m+1})$.  
Then 
$T_{m,m}(\lambda) = \lambda^{2m+2} - \lambda^{2m+1} - 4\lambda^{m+1} - \lambda + 1$. 
Plugging in the identity 
$$
0 = S_{m-1,m+1}(\lambda) = 
\lambda^{2m+2}-\lambda^{2m+1} - 2\lambda^{m+2} + 2\lambda^m + \lambda - 1,
$$
and subtracting this from $T_{m,m}(\lambda)$, we have 
\begin{eqnarray*}
T_{m,m}(\lambda) = 2\lambda^{m+2} - 4\lambda^{m+1} - 2\lambda^m - 2\lambda + 2
= 2\lambda^m(\lambda^2 - 2\lambda + 1) + 2(1-\lambda).
\end{eqnarray*}
The roots of $t^2 - 2t +1$ are $1 \pm \sqrt{2}$.  
Since $1 - \sqrt{2} < 1 < \lambda < 2 < 1 + \sqrt{2}$, 
$\lambda^2-2\lambda+1$ and $1-\lambda$ are both negative, and hence $T_{m,m}(\lambda) < 0$. 
Since $\lambda(\beta_{m,m})$ is the largest real root of $T_{m,m}(t)$, it follows that 
$\lambda(\sigma_{m-1,m+1}) = \lambda < \lambda(\beta_{m,m})$. 

For the second inequality, we plug in $\lambda = \lambda(\sigma_{m-1,m+2})$ 
into $T_{m,m+1}$.  This gives 
$$
T_{m,m+1}(\lambda)=2\lambda^m(\lambda^3 - \lambda^2 - \lambda - 1) - \lambda  - 1.
$$
Thus,  $\lambda^3 - \lambda^2 - \lambda - 1 < 0$ would imply $T_{m,m+1} (\lambda) < 0$.  
The polynomial $g(t) = t^3 - t^2 - t -1$ has one real root ($\approx 1.83929$) 
and is increasing for $t >1$. 
Since $\lambda(R_m)$ is decreasing with $m$, 
and $\lambda < \lambda(R_2) \approx 1.69562 < 1.8$ by \eqref{s-m-b-equation}, 
we see that $T_{m,m+1}(\lambda) < 0$ for $m \ge 3$.   
For the remaining case, we check that $T_{2,3}  = S_{1,4}$. 
\end{proof}

Propositions~\ref{min-prop} and \ref{min-Smn-Bmm-prop} show the following. 
\begin{cor}
\label{min-braid-cor}  
The least dilatation among $\sigma_{m,n}$ and $\beta_{m,n}$ 
for $m+n=2g$ $(g \ge 2)$ is given by $\lambda(\sigma_{g-1,g+1})$. 
\end{cor}

By \fullref{Salem-Boyd-cor}, 
\fullref{Rm2-lem} and \fullref{Perron-prop}, 
 the dilatations $\lambda(\sigma_{m,n})$ for $n \ge m+2$ 
converge to $1$ as $m,n$ approach infinity.  
We prove the following stronger statement, which implies Theorems~\ref{inequalities-thm} and \ref{asymp-thm}. 

\begin{prop}
\label{min-Smn-prop} 
For $g \ge 2$, 
$$
\frac{\log(2+\sqrt{3})}{g+1} < \log (\lambda(\sigma_{g-1,g+1})) < \frac{ \log (2 +\sqrt{3})}{g}.
$$
\end{prop}

\begin{proof} 
Using \fullref{CharEq-thm}, we see that $\lambda = \lambda(\sigma_{g-1,g+1})$ satisfies 
\begin{eqnarray}
\label{dil-eqn}
0=\lambda^{2g+1} - 2\lambda^{g+1} - 2\lambda^g + 1
=\lambda(\lambda^g)^2 - 2(\lambda+1)\lambda^g + 1. 
\end{eqnarray}
Since $\lambda$ is the largest real solution, the quadratic formula gives 
\begin{eqnarray*}
\lambda^g = \frac{2(\lambda+1) + \sqrt{4(\lambda+1)^2 - 4\lambda}}{2\lambda}
= \frac{\lambda+1 + \sqrt{\lambda^2 +  \lambda + 1}}{\lambda}. 
\end{eqnarray*}
It follows that 
\begin{eqnarray}
\label{dilatation-eqn}
\lambda^{g+1} &=& \lambda+1 + \sqrt{\lambda^2 + \lambda+1}.
\end{eqnarray} 
Since $1 < \lambda < 2$ for all $g \ge 2$, \eqref{dilatation-eqn} implies 
$2 + \sqrt{3} < \lambda^{g+1} < 3+\sqrt{7}$. 

We improve the upper bound using an argument conveyed to us by Minakawa. 
Rewrite \eqref{dil-eqn} as follows 
\begin{eqnarray*}
0 =\lambda^{2g+1} + \lambda^{2g} - \lambda^{2g} - 2(\lambda+1)\lambda^g + 1
=\lambda^{2g}(\lambda+1) - (\lambda^{2g} - 1) - 2(\lambda+1)\lambda^g. 
\end{eqnarray*}
Factoring out  $(\lambda+1)$ gives 
$$
0 = \lambda^{2g} - \frac{\lambda^{2g}-1}{\lambda+1} - 2 \lambda^g.
$$
On the other hand, since $\lambda > 1$, we have 
$$
\frac{\lambda^{2g} -1}{\lambda+1} < \frac{1}{2}(\lambda^{2g}-1).
$$
This implies the inequality 
$$
x^{2g} - \frac{x^{2g}-1}{x+1} - 2x^g >  
x^{2g} - \frac{1}{2}(x^{2g}-1)-2x^g=\frac{1}{2}(x^{2g} - 4x^g + 1) =: p(x)
$$
for $x$ near $\lambda$.  Thus, $p(x)$ has a real root $\mu$ larger than $\lambda$. 
Using the quadratic formula again, we see that $\mu^g = 2 + \sqrt{3}$, 
and hence $\lambda^g < \mu^g = 2+\sqrt{3}$. 
\end{proof}

\section{Further discussion and questions}
\label{discussion-section}

By Propositions~\ref{min-prop} and \ref{min-Smn-Bmm-prop}, 
for $s \ge 5$ strands, the minimal dilatations by our construction 
come from $\sigma_{g-1,g+1}$ when $s = 2g+1$; 
and $\sigma_{g-1,g+2}$ when $s=2g+2$.   
For $s$ even, there is an example of a braid with smaller dilatation 
than that of $\sigma_{g-1,g+2}$ (see the end of \fullref{forcing-section}), 
but for $s$ odd, we know of no such examples. 

Since $\Sigma(\sB(D,2g+1)) \subset \Sigma(\sM_g)$ (\fullref{spectrum-prop}), 
Penner's lower bound \cite{Penner91} for elements of $\Sigma(\sM_g)$ extend to 
$\Sigma(\sB(D,2g+1))$. 
Hence we have 
$$
\delta(\sB(D,2g+1)) \ge \delta(\sM_g) \ge  \frac{\log 2}{12 g - 12}. 
$$
For $g =2$, Zhirov  shows \cite{Zhirov95} that if $\phi \in \sM_2$ is pseudo-Anosov 
with orientable invariant foliations, then $\lambda(\phi)$ is bounded below by the largest root of 
$x^4 -x^3 - x^2 -x+1$. 
For $s=5$, $\sigma_{1,3}$ is pseudo-Anosov, and its lift to $F_2$ is orientable.  
Our formula shows that the dilatation of $\sigma_{1,3}$  is the largest root of Zhirov's equation, 
and hence $\sigma_{1,3}$ achieves the least dilatation 
among orientable pseudo-Anosov maps on $F_2$. 
This yields the following weaker version of Ham and Song's result \cite{HamSong05}, 
which doesn't assume any conditions on the combinatorics of train tracks. 

\begin{cor} 
The braid $\sigma_{1,3}$ is pseudo-Anosov with the least dilatation among braids 
$\beta \in \sB(D;\sS)$ on $5$ strands such that 
all singularities of  $S^2 \setminus (\sS \cup \{p_{\infty}\})$ 
for the invariant foliations associated to  the pseudo-Anosov map $\Phi_{\what{\beta}}$ 
are even--pronged. 
\end{cor}

We discuss the following general question and related work on the forcing relation in \fullref{forcing-section}.
\begin{ques} 
Is there a braid $\beta \in \sB(D, 2g+1)$ such that 
$\lambda(\beta)< \lambda(\sigma_{g-1,g+1})$? 
\end{ques}

Let $\sK_g^s \subset \sM_g^s$ be the subset of mapping classes that arise as the monodromy of 
a fibered link  $(K,F)$ in $S^3$, where the fiber $F$ has genus--$g$ and the link $K$ has 
$s$ components. 
\begin{ques}\label{fiberedlink-ques} Is there a strict inequality 
$\delta(\sM_g^s) < \delta(\sK_g^s)$? 
\end{ques}
In \fullref{fiberedQ-section}, we briefly discuss what is known about bounds on dilatations of 
pseudo-Anosov monodromies of fibered links, and show how the braids $\beta_{m,n}$ arise 
in this class.

\subsection{The forcing relation on the braid types}
\label{forcing-section}

The existence of periodic orbits of dynamical systems can imply the existence of other periodic orbits. 
Continuous maps of the interval give typical examples for such phenomena. 
Boyland introduced the notion of {\it braid types}, and defined 
a relation on the set of braid types to study an analogous phenomena  in the $2$--dimensional case. 
Recall that there is an isomorphism 
$$
\mathcal{B}(D;\sS)/Z(\mathcal{B}(D; \sS)) \rightarrow M(D; \sS).
$$
Let $f \colon\thinspace D \rightarrow D$ be an orientation preserving homeomorphism 
with a single periodic orbit $\sS$. 
The isotopy class of $f$ relative to $\sS$ is represented by 
$\beta Z(\mathcal{B}(D; \sS))$ for some braid $\beta \in \mathcal{B}(D;\sS)$ 
by using the isomorphism above. 
The {\it braid type} of $\sS$ for $f$, denoted by $bt(\sS,f)$, is  the conjugacy class 
$[ \beta Z(\mathcal{B}(D; \sS))]$ in the group $\mathcal{B}(D;\sS)/Z(\mathcal{B}(D; \sS))$. 
To simplify the notation, 
we will write $[ \beta ]$ for $[ \beta Z(\mathcal{B}(D; \sS))]$. 
Let
\begin{eqnarray*}
bt(f) &=& \{bt(P,f)\ |\ P\ \mbox{is\ a\ single\ periodic\ orbit\ for\ }f\},
\end{eqnarray*}
and $BT$ the set of all braid types for all homeomorphisms of $D$. 
A relation $\succeq$ on $BT$ is defined as follows: 
For $b_i \in BT$ ($i=1,2$), 
\begin{center}
$b_2 \succeq b_1 \Longleftrightarrow$ 
(For any $f\co D \rightarrow D$,  if $b_2 \in bt(f)$, then  $b_1 \in bt(f))$. 
\end{center}
We say that $b_2$ {\it forces} $b_1$ if $b_2 \succeq b_1$. 
It is known that $\succeq $ gives a partial order on $BT$ (see Boyland
\cite{Boy92} and Los \cite{Los}), and 
we call the relation the {\it forcing relation}. 

The topological entropy gives a measure of orbits complexity for a continuous map 
of the compact space (see Walters \cite{Wal}). Let $h(f) \ge 0$ be the topological entropy of $f$. 
For a pseudo-Anosov braid $\beta \in \sB(D;\sS)$, 
$\log(\lambda(\beta))$ is equal to $h(\beta)$,  
which in turn is the least $h(f)$ among all $f$ with an invariant set $\sS$ such that 
$ bt(\sS,f) = [ \beta ]$ (see Fathi--Laudenbach--Poenaru
\cite[Expos\'{e}~10]{FLP}).
One of the relations between the forcing relation and the dilatations is as follows. 

\begin{thm}[Los \cite{Los}]
\label{thm_Los} 
Let $\beta_1$ and $\beta_2$ be pseudo-Anosov braids. 
If $[ \beta_2 ] \succeq  [ \beta_1 ]$ and $[ \beta_2 ]  \ne [ \beta_1 ]$, then 
$\lambda(\beta_2) > \lambda(\beta_1)$.
\end{thm}

The forcing relation on braids $\beta_{m,n}$ and $\sigma_{m,n}$ was
studied by Kin \cite{Kin}. 

\begin{thm}
\label{thm_order1} 
For any $m,n \ge 1$, 
\begin{enumerate}
\item[(1)] $[\beta_{m,n}]  \succeq [\beta_{m,n+1}]$, 
\item[(2)] $[\beta_{m,n}]  \succeq [\beta_{m+1,n}]$, 
\item[(3)] $[\beta_{m,n}]  \succeq [\sigma_{m,\ell}]$ if $\ell \ge m+2$, and 
\item[(4)] $[ \sigma_{m,n} ]   \succeq [ \sigma_{m,\ell} ]$ 
if $n \ge \ell \ge m+2$. 
\end{enumerate}
\end{thm}

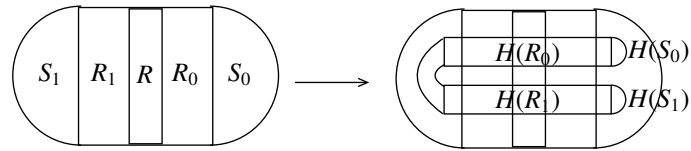
\begin{figure}[htbp]
\labellist\small
\hair=1pt
\pinlabel {$S_1$} at 80 750
\pinlabel {$R_1$} at 120 750
\pinlabel {$R$} at 150 750
\pinlabel {$R_0$} at 180 750
\pinlabel {$S_0$} at 220 750

\pinlabel {$H(R_0)$} at 433 765
\pinlabel {$H(R_1)$} at 433 730
\pinlabel {$H(S_0)$} [l] at 505 767
\pinlabel {$H(S_1)$} [l] at 505 731
\endlabellist
\begin{center}
\includegraphics[width=9cm]{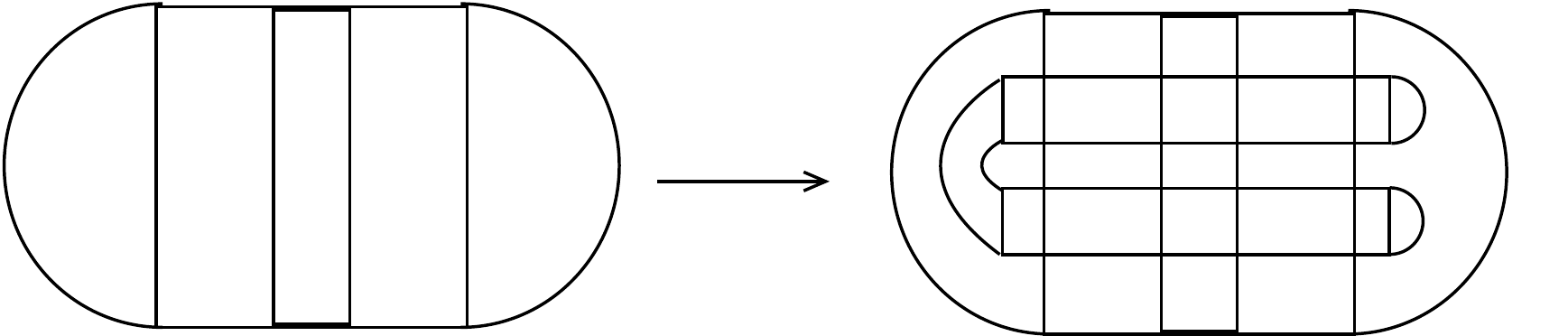}
\caption{Smale--horseshoe map}
\label{fig_smale-h}
\end{center}
\end{figure}

The Smale--horseshoe map $\mathtt{H} \colon\thinspace D \rightarrow D$ is a diffeomorphism  
such that the action of $\mathtt{H}$ on three rectangles $R_0,R_1$, and $R$ and two half disks $S_0,S_1$ is given in \fullref{fig_smale-h}. 
The restriction of $\mathtt{H}$ to $R_i$ ($i=0,1$) is an affine map such that 
$\mathtt{H}$ contracts $R_i$ vertically and stretches horizontally. 
The restriction of $\mathtt{H}$ to $S_i$ ($i=0,1$) is a contraction map. 
Katok showed \cite{Katok80} that 
any $C^{1 + \epsilon}$ surface diffeomorphism ($\epsilon >0$) with  positive topological entropy has 
a horseshoe  in some iterate. 
This suggests that the Smale--horseshoe map is a fundamental model for chaotic dynamics. 

The set
$$\Omega = \displaystyle\bigcap_{n \in {\Z}} \mathtt{H}^n (R_0 \cup R_1)$$
is invariant under $\mathtt{H}$. 
Let $\Sigma_2= \{0,1\}^{\Z}$, and 
\begin{eqnarray*}
\sigma \colon\thinspace \Sigma_2 &\rightarrow& \Sigma_2
\\
(\ldots w_{-1}\cdot w_0 w_1 \ldots) &\mapsto& (\ldots w_{-1}w_0 \cdot w_1 \ldots),
\quad w_i \in \{0,1\}
\end{eqnarray*}
the shift map. 
There is a conjugacy $\mathcal{K} \colon\thinspace \Omega \rightarrow \Sigma_2$ 
between the two maps 
$\mathtt{H}|_{\Omega} \colon\thinspace \Omega \rightarrow \Omega$ and 
$\sigma \colon\thinspace \Sigma_2 \rightarrow \Sigma_2$ as follows: 
\begin{eqnarray*}
\mathcal{K} \colon\thinspace \Omega &\rightarrow& \Sigma_2
\\
x &\mapsto& (\ldots \mathcal{K}_{-1}(x) \mathcal{K}_0(x) \mathcal{K}_1(x) \ldots),
\end{eqnarray*}
where
\[
\mathcal{K}_i(x) =
\left\{
\begin{array}{ll}
0 \hspace{3mm}\  \mbox{if\ }&\mathtt{H}^i(x) \in R_0,
\\
1 \hspace{3mm}\  \mbox{if\ }&\mathtt{H}^i(x) \in R_1.
\end{array}
\right.
\]
If $x$ is a period $k$ periodic point, then 
the finite word $ (\mathcal{K}_0(x) \mathcal{K}_1(x) \ldots \mathcal{K}_{k-1}(x))$ is called 
the {\it code} for $x$. 
We say that a braid $\beta $ is a {\it horseshoe braid} 
if there is a periodic orbit for the Smale--horseshoe map whose braid type is $[ \beta]$. 
We define a horseshoe braid type in a similar manner. 
For the study of  the restricted forcing relation on the set of horseshoe braid types, see 
the papers \cite{dCH04,Hall94} by de Carvalho and Hall. 

\begin{figure}[htbp]
\labellist\small
\pinlabel {$R_1$} [b] at 155 251
\pinlabel {$R_0$} [b] at 277 251
\pinlabel {$a$} [t] at 137 183
\pinlabel {$b$} [t] at 166 224
\pinlabel {$c$} [b] <0pt,1pt> at 259 81
\pinlabel {$d$} [t] at 277 210
\pinlabel {$e$} [t] at 297 123
\pinlabel {$a$} [b] at 461 221
\pinlabel {$b$} [b] at 496 221
\pinlabel {$c$} [b] at 572 221
\pinlabel {$d$} [b] <1pt,0pt> at 607 221
\pinlabel {$e$} [b] at 645 221
\pinlabel {$1$} [b] at 482 271
\pinlabel {$0$} [b] at 613 271
\endlabellist
\begin{center}
\includegraphics[width=7cm]{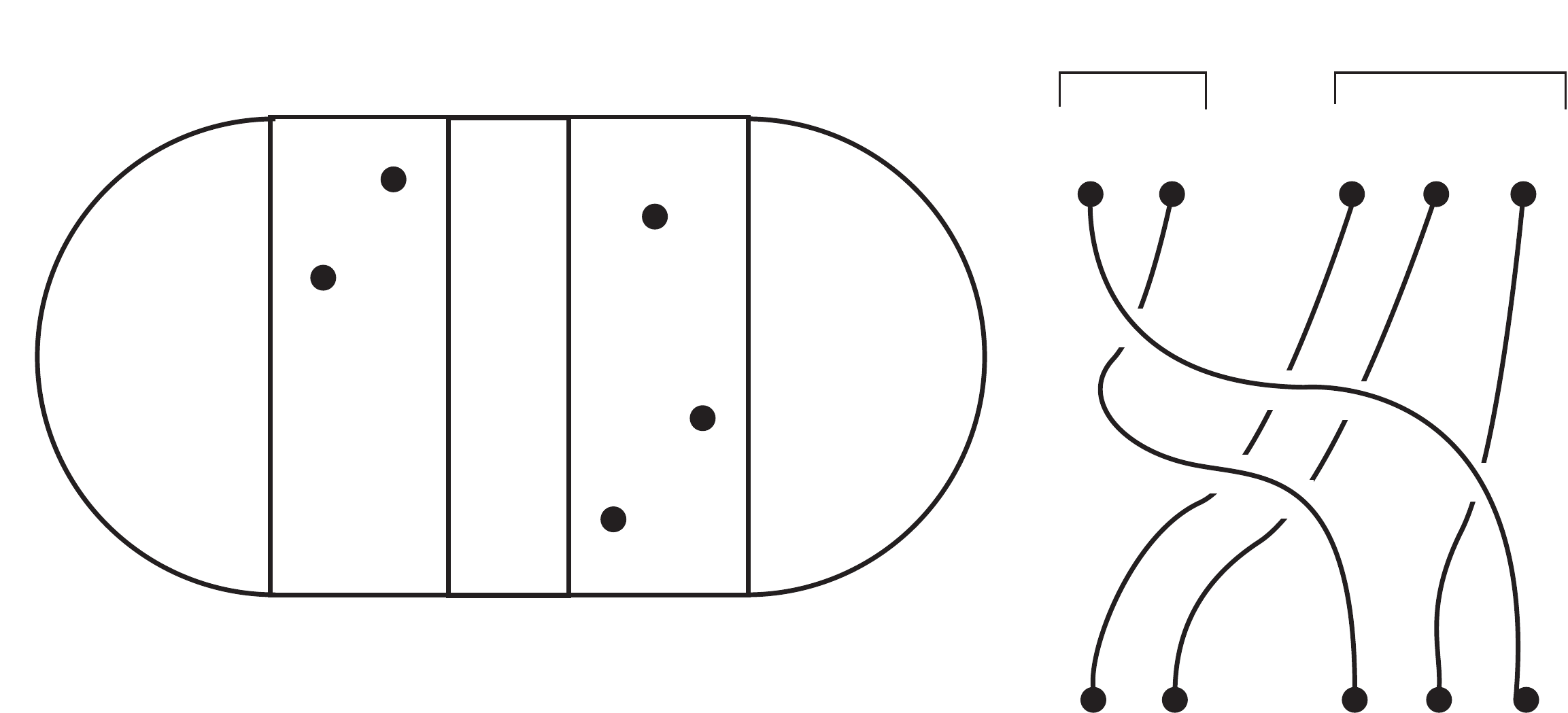}
\caption{Periodic orbit with the code $10010$ and its braid representative}
\label{fig_smale-h-orbit}
\end{center}
\end{figure}
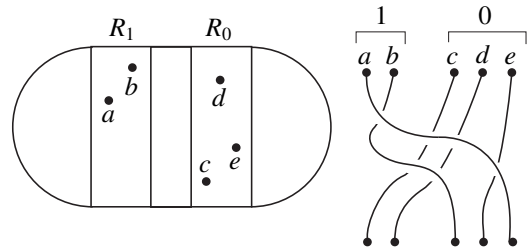

The result by Katok together with \fullref{thm_Los} implies that 
horseshoe braids are relevant candidates realizing the least dilatation. 
It is not hard to see that the braid type of the periodic orbit with the code 
$$1 \underbrace{0 \ldots 0}_{n-1} 1 \underbrace{0 \ldots 0}_{m}
\quad\text{or}\quad
1 \underbrace{0 \ldots 0}_{n-1} 1 \underbrace{0 \ldots 0}_{m-1}1
\quad(n \ge m+2)$$ 
is represented by $[\sigma'_{m,n}](= [\sigma_{m,n}])$ (For the definition of $\sigma'_{m,n}$, 
see the end of \fullref{Graphmap-section}). 
Hence, $\sigma_{m,n}$ $(n \ge m+2)$ is a horseshoe braid. 
\fullref{fig_smale-h-orbit}  illustrates the periodic orbit with the code $10010$ 
and its braid representative. 

For the case of even strands, 
there is a horseshoe braid having dilatation less than our examples. 
In fact, the braid type of period $8$ periodic orbit with the code $10010110$ is given by 
$[\beta= (\sigma_1 \sigma_2 \sigma_3 \sigma_4 \sigma_5 \sigma_6)^3 \sigma_7]$, 
which satisfies
$\lambda(\beta)= 1.4134 \ldots < \lambda(\sigma_{2,5}) = 1.5823 \ldots$. 

\subsection{Fibered links}
\label{fiberedQ-section}

For a fibered link $(K,F)$ with fibering surface $F$, the {\it monodromy} 
$\Phi_{(K,F)} \colon\thinspace F \rightarrow F$ 
is the homeomorphism defined up to isotopy such that the complement of a regular neighborhood of 
$K$ in $S^3$ is a mapping torus for $\Phi_{(K,F)}$. 
Define $\Delta_{(K,F)}$ to be the characteristic 
polynomial for the monodromy $\Phi_{(K,F)}$ restricted to first homology $\HH_1(F,\R)$. 
If $K$ is a fibered knot, then  $\Delta_{(K,F)}$ is the Alexander polynomial of $K$ 
(see Kawauchi \cite{Kaw:Survey} and Rolfsen \cite{Rolfsen76}). 

The {\it homological dilatation} of a pseudo-Anosov map $\Phi \colon\thinspace F \rightarrow F$ 
is defined to be $\lambda(f)$, where $f$ is the characteristic polynomial for the restriction of $\Phi$ to 
$\HH_1(F;\R)$.  Thus, if $(K,F)$ is a fibered link and $\Phi_{(K,F)}$ is the 
monodromy, then $\lambda(\Delta_{(K,F)})$ is the homological dilatation of $\Phi_{(K,F)}$.   
In the case where $\Phi_{(K,F)}$ is a pseudo-Anosov map, 
$\lambda(\Delta_{(K,F)})$ and $\lambda(\Phi_{(K,F)})$ are equal 
if $\Phi_{(K,F)}$ is orientable (see Rykken \cite{Rykken99}). 

Any monic reciprocal integer polynomial is equal to $\Delta_{(K,F)}$ for some 
fibered link $(K,F)$ up to multiples of $(t-1)$ and $\pm t$ (see Kanenobu
\cite{Kanenobu81}).
In particular, any reciprocal Perron polynomial\footnote{A monic integer
polynomial $f$ is {\it Perron}
if $f$ has a root $\lambda(f)>1$ such that $\lambda(f)> |\alpha|$ 
for all roots $\alpha \ne \lambda(f)$.}
can be realized.    
On the other hand, if $\Phi_{(K,F)}$ is orientable,  then 
$\lambda(\Phi_{(K,F)})$ is in general strictly greater than $\lambda(\Delta_{(K,F)})$. 

Leininger \cite{Leininger04} exhibited a pseudo-Anosov map 
$\Phi_L\colon\thinspace  F_5 \rightarrow F_5$ with dilatation $\lambda_L$, where 
$$\log (\lambda_L) = 0.162358.$$ 
A comparison shows that this number is strictly less than our candidate for the least 
element of $\Sigma (\sB(D,{2g+1})$ for $g=5$: 
$$\log(\lambda(\sigma_{4,6}) )= 0.240965.$$
The pseudo-Anosov map $\Phi_L$ 
is realized as the monodromy of the fibered $(-2,3,7)$--pretzel knot. 
Its dilatation $\lambda_L$ is the smallest known Mahler measure greater
than $1$ among monic integer polynomials (see Boyd \cite{Boyd81} and
Lehmer \cite{Lehmer33}).

In the rest of this section, we will construct fibered links whose monodromies are obtained by 
lifting the spherical mapping classes associated to $\beta_{m,n}$.  
We set $g = \lfloor{\frac{m+n}{2}}\rfloor$. 
Let $\sS$ be the set of marked points on int$(D)$ corresponding to the strands of $\beta_{m,n}$, 
and $F$ the double covering of $D$, branched over $\sS$.  Then $F$ has 
one boundary component if $m+n$ is even and two boundary components if $m+n$ is odd. 
Let $\Phi'_{m,n}$ be the lift of the pseudo-Anosov representative $\Phi_{\beta_{m,n}}$ 
of $\phi_{\beta_{m,n}} \in \sM(D; \sS)$  to $F$.  
Using an argument similar to that in the proof of \fullref{spectrum-prop}, we have 
$$
\lambda(\Phi'_{m,n}) =\lambda(\Phi_{\beta_{m,n}})= \lambda(\beta_{m,n}).
$$
Note that  $\Phi'_{m,n}$ is $1$--pronged near each of the boundary of $F$ if $m+n$ is odd.

Let $K_{m,n}$ be the two--bridge link given in \fullref{twobridge-fig}. 
By viewing $(S^3,K_{m,n})$ as the result of a sequence 
of Hopf plumbings see Hironaka \cite[Section 5]{Hironaka:Salem-Boyd},
one has the following.

\begin{figure}[htbp]
\labellist\tiny
\pinlabel {$N$} at 22 142
\pinlabel {($N$ positive half twists)} [l] at 0 106
\pinlabel {($N$ negative half twists)} [l] at 0 12
\pinlabel {$-N$} at 22 50
\pinlabel {$-n$} at 213 130
\pinlabel {$m{+}1$} at 218 44
\pinlabel {$=$} at 50 142
\pinlabel {$=$} at 50 50
\endlabellist
\begin{center}
\includegraphics[height=1.5in]{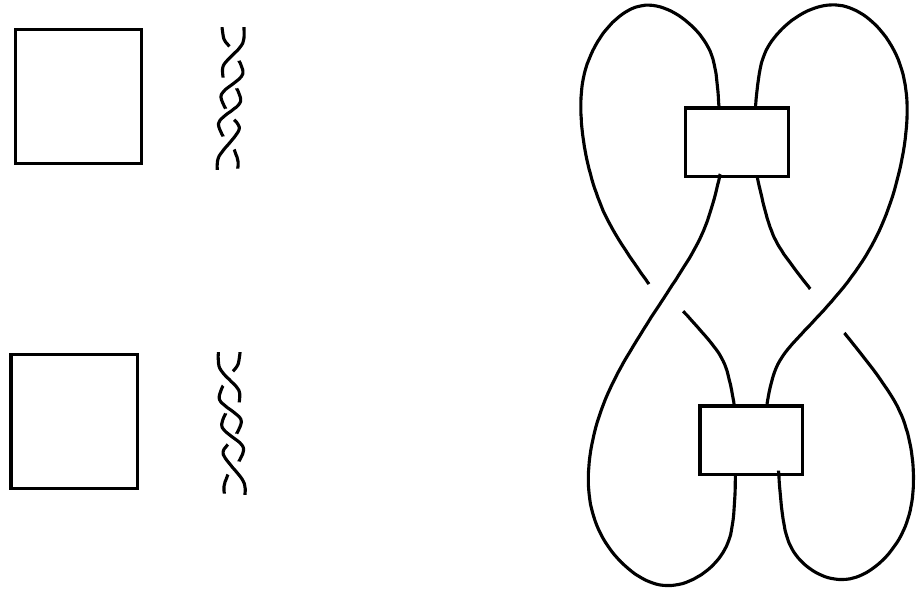}
\caption{Two--bridge link associated to $\beta_{m,n}$}
\label{twobridge-fig}
\end{center}
\end{figure}
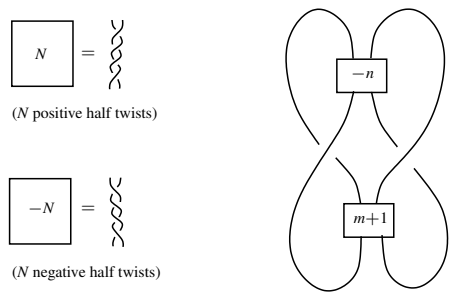

\begin{prop}  
The complement of a regular neighborhood of $K_{m,n}$ in $S^3$ is 
a mapping torus for $\Phi'_{m,n}$.
\end{prop}
The fibered links $K_{m,n}$ and the dilatations of $\Phi'_{m,n}$ were
also studied by Brinkmann \cite{Brinkmann04}.

Let $\Delta_{m,n}$ be the Alexander polynomial for $K_{m,n}$. 
Salem--Boyd sequences for $\Delta_{m,n}$ were computed by Hironaka \cite{Hironaka:Salem-Boyd}. 
\fullref{Bmn-lift-prop} implies the following. 

\begin{lem} 
If $m$ and $n$ are both odd, then 
$\lambda(\beta_{m,n}) = \lambda(\Phi'_{m,n}) =  \lambda(\Delta_{K_{m,n}})$. 
\end{lem}

\begin{ques} 
Let $\Phi_{\sigma_{m,n}}$ be the pseudo-Anosov representative of $\phi_{\sigma_{m,n}}$. 
Is there a fibered link $K$ in $S^3$ such that 
the complement of a regular neighborhood of $K$ in $S^3$ is a mapping torus for 
a lift of $\Phi_{\sigma_{m,n}}$  ? 
%
\end{ques}

\bibliographystyle{gtart}
\bibliography{link}

\end{document}